\documentclass[12pt,a4paper]{article}
\usepackage{amsfonts,amssymb,amsmath,amscd,stmaryrd,latexsym,makeidx,theorem,psfrag}
\usepackage{amsmath,pstricks,latexsym,theorem,epsfig} 
\usepackage{amsfonts,amssymb,latexsym,epsfig,amsmath}
\usepackage{amsmath,latexsym,theorem} 
\usepackage{PSFigure,psfrag}
\usepackage{graphicx}
\newtheorem{Theorem}{Theorem}[section] 
\newtheorem{Definition}{Definition}[section] 
\newtheorem{Proposition}{Proposition}[section] 
\newtheorem{Lemma}{Lemma}[section] 
\newtheorem{Corollary}{Corollary}[section] 
 
\theorembodyfont{\rmfamily} 
\newtheorem{Remark}{Remark}[section] 

\newcommand{\rec}[1]{{(\ref{#1})}}

\def \R{\mathbb{R}} 
\def \C{\mathbb{C}}
 
\def \Z{\mathbb{Z}} 
 
\def\11{1\!\!1}

\newcommand{\de}{\partial}
\newcommand{\ve}{\varepsilon}
\newcommand{\weak}{ \stackrel{*}{\rightharpoonup}}

\newcommand{\lapfr}{(-\Delta)^\frac{1}{2}}

\newenvironment{proof}{\noindent {\bf Proof.}}{\hfill$\square$\medskip}

\DeclareMathOperator{\sign}{sign}

\DeclareMathOperator{\loc}{loc}

\DeclareMathOperator*{\dist}{dist}

\DeclareMathOperator{\Jac}{Jac}
\def\Intm
   {\mkern12mu\hbox{\vrule height4pt
   depth-3.2pt width5pt}%
   \mkern-16.5mu\int}

\def\XXint#1#2#3{{\setbox0=\hbox{$#1{#2#3}{\int}$ }
\vcenter{\hbox{$#2#3$ }}\kern-.6\wd0}}

\begin{document}
\title{Blow-up analysis of a nonlocal Liouville-type equation}
\author{ Francesca Da Lio\thanks{Department of Mathematics, ETH Z\"urich, R\"amistrasse 101, 8092 Z\"urich, Switzerland.}  \and Luca Martinazzi\thanks{Department of Mathematics and Computer Science, Universit\"at Basel, Spiegelgasse 1, 4051 Basel, Switzerland. Supported by the Swiss National Foundation.} \and   Tristan Riviere$^*$   }

\maketitle

\begin{abstract}
In this paper we perform a blow-up and quantization analysis of  the following nonlocal Liouville-type  equation
\begin{equation}
  (-\Delta)^\frac12 u= \kappa e^u-1~\mbox{in $S^1\,,$}\label{eqs}\end{equation}
  where   $(-\Delta)^\frac{1}{2}$ stands for the fractional
 Laplacian and $\kappa$  is a bounded function. We interpret equation
 \rec{eqs} as the prescribed curvature equation to  a curve in conformal parametrization. We also establish a relation between equation \rec{eqs} and the analogous equation in $\R$
\begin{equation}
 (-\Delta)^\frac12 u =Ke^{u}\quad \text{in }\R\,,\label{eqr}
\end{equation}
with $K$ bounded on $\R$.
\end{abstract}

 \tableofcontents 
 \section{Introduction}
 The equation  prescribing the scalar curvature within a 
conformal class of manifolds is an old problem which has stimulated a lot of works in geometry and analysis. In dimension $n=2$ it is the so-called Liouville equation. More precisely  if $(\Sigma, g_0)$  is a smooth, closed Riemann surface with  Gauss curvature $K_{g_0}$, an easy computation shows that a function $K(x)$  is the Gauss curvature for some {metric $g=e^{2u}g_0$ conformally equivalent} to the  metric $g_0$  with $u \colon\Sigma \to\R, $ if and only if there exists a solution $u = u(x)$ of
\begin{equation}\label{ld2}-\Delta_{g_0} u=K e^{2u}-K_{{g_0}},~~\text{on } \Sigma\end{equation}
where $\Delta_{g_0}$ is the Laplace Beltrami operator on $(\Sigma, g_0)$\,, (see e.g. \cite{ch} for more details).
\par
In particular when $\Sigma =\R$ or $\Sigma=S^2$ equation \rec{ld2} reads respectively
\begin{equation}\label{ld2r} -\Delta u=K e^{2u} ,~~\text{on }\R^2\end{equation}and
\begin{equation}\label{ld2s}-\Delta_{S^2} u=K e^{2u}-1,~~\text{on } S^2\,.\end{equation}\par
Both equations \rec{ld2r} and \rec{ld2s} have been largely studied in the literature. Here 
we would like to recall the famous blow-up result by Br\'ezis and Merle in \cite{BM} concerning Equation \rec{ld2r}\,.

\begin{Theorem}[Thm 3, \cite{BM}]\label{thbm}
Assume {$(u_k)\subset L^1(\Omega)$, $\Omega$ open subset of $\R^2$}, is a sequence of solutions to \rec{ld2r} satisfying for some $1<p\le  \infty$, $K_k\ge 0$,
$\|K_k\|_{L^p}\le C_1\,,$ and $\|e^{u_k}\|_{L^{p'}}\le C_2\,.$ Then up to subsequences the following alternatives hold:
either $(u_k)$ is bounded in $L^{\infty}_{\loc}(\Omega)$, or $u_k(x)\to -\infty$  uniformly on compact subsets of $\Omega$, or
there is a finite nonempty set $B=\{a_1,\ldots,a_N\}\subset\Omega$  (blow-up set) such that
$u_k(x)\to -\infty$ on compact subsets of $\Omega\setminus B$. In addition in this last case
$K_k e^{2u_k}$ converges in the sense of measure on $\Omega$ to $\sum_{i=1}^N\alpha_i\delta_{a_i},$ with 
$\alpha_i\ge{\frac{2\pi}{{p^{\prime}}} }\,.$
\end{Theorem}

The purpose of this work is to investigate an analogue  prescribed
 curvature problem  in dimension $1$\,. Even if this is a classical  problem, it has never been studied
so far (up to our knowledge) from the point of view of conformal geometry.
 In the case for instance of a planar Jordan curve (namely a continuous closed and simple curve) there is the possibility to parametrize it through the trace of the Riemann mapping between the disk $D^2$ and the simply connected domain enclosed by the curve. The equation corresponding to such a parametrization reads as follow
\begin{equation}\label{modelequbis}
  (-\Delta)^\frac12 \lambda= \kappa e^{\lambda}-1~\mbox{in $S^1$}\end{equation}
  where $ e^{\lambda} d\theta$ and $\kappa e^{\lambda} d\theta$ are respectively  the length form and the curvature {density} of the curve in this parametrization. {The definition and relevant properties of the operator $(-\Delta)^\frac 12$ will be given in the appendix.}\par
  
  In this paper we are going to study the space of solutions to Equation \rec{modelequbis} and show that, modulus the action of M\"obius transformations of the disk, there is correspondence between these solutions and the boundary traces of conformal immersions of the disk. This property can be seen as a sort of 
    generalized  Riemann Mapping Theorem.\par
  This permits us to perform a complete blow-up analysis of equation \rec{modelequbis} in the spirit of Theorem \ref{thbm},  even if we do not get exactly the same dichotomy. More precisely our first main
  result is  the following theorem.
\begin{Theorem}\label{mainth}
Let $(\lambda_k)\subset L^1(S^1,\R)$ be a sequence with
\begin{equation}\label{boundlength}
L_k:=\|e^{\lambda_k}\|_{ L^1(S^1)}\le \bar L
\end{equation}
satisfying  
\begin{equation} \label{liouvfrack}
(-\Delta)^\frac12 \lambda_k=\kappa_k e^{\lambda_k}-1\quad \text{in }S^1,
\end{equation}
where $\kappa_k\in L^\infty(S^1,\R)$ satisfies
\begin{equation}\label{boundcurv}
\|\kappa_k\|_{L^{\infty}(S^1)}\le \bar \kappa\,.
\end{equation}
Then up to subsequence we have  $\kappa_k e^{ \lambda_k} \rightharpoonup \mu $ weakly  in $ W^{1,p}_{\loc}(S^1\setminus B)$ for every $p<\infty$, where $\mu$ is a  Radon measure, $B:=\{a_1,\ldots, a_N\}$ is a (possibly empty) subset of $S^1$ and
$\kappa_k\stackrel{*}{\rightharpoonup}\kappa_{\infty}$ in $L^{\infty}(S^1)\,.$ Set $\bar\lambda_k:=\frac{1}{2\pi}\int_{S^1}\lambda_kd\theta$. Then one of the following alternatives holds:
\par
i)  $\bar \lambda_k\to -\infty$ as $k\to\infty$, $N=1$ and $\mu=2\pi\delta_{a_1}$. In this case
$$v_k:=\lambda_k-\bar\lambda_k \rightharpoonup v_\infty \quad \text{ in }W^{1,p}_{\loc}(S^1\setminus\{a_1\})\text{ for every }p<\infty,$$
where  $v_{\infty}(e^{i\theta})= -\log(2(1-\cos(\theta-\theta_1)))$ for $a_1=e^{i\theta_1}$, solving
\begin{equation} \label{liouvlimit2}
(-\Delta)^\frac12 v_\infty=-1+2\pi \delta_{a_1}\quad \text{in }S^1\,.
\end{equation}\par
ii)  $\bar \lambda_k\to -\infty$ as $k\to\infty$, $N=2$ and $\mu =\pi(\delta_{a_1}+\delta_{a_2})$. In this case
$$v_k:=\lambda_k-\bar\lambda_k \rightharpoonup v_\infty \quad \text{ in }W^{1,p}_{\loc}(S^1\setminus\{a_1,a_2\})\text{ for every }p<\infty,$$
where
$$v_{\infty}(e^{i\theta})= -\frac12\log(2(1-\cos(\theta-\theta_1)))-\frac12\log(2(1-\cos(\theta-\theta_2))), \quad a_1=e^{i\theta_1},\;a_2=e^{i\theta_2}$$
solves
\begin{equation} \label{liouvlimit3}
(-\Delta)^\frac12 v_\infty=-1+\pi \delta_{a_1}+\pi\delta_{a_2}\quad \text{in }S^1\,.
\end{equation}
\par 
iii) $|\bar\lambda_k|\le C$ and $\mu=\kappa_\infty e^{\lambda_\infty}+\pi(\delta_{a_1}+\dots+\delta_{a_N})$ for some $\lambda_\infty\in W^{1,p}_{\loc}(S^1\setminus B)$, with $\lambda_\infty,e^{\lambda_\infty}\in L^1(S^1)$ and 
\begin{equation} \label{liouvlimit}
(-\Delta)^\frac12 \lambda_\infty=\kappa_{\infty} e^{\lambda_\infty}-1+\sum_{i=1}^{N}\pi \delta_{a_i}\quad \text{in }S^1\,.
\end{equation}
\end{Theorem}
  We would like to stress that  we obtain  a \emph{quantization-type} result, namely the curvature concentrating at each blow-up point is precisely $\pi$, without any assumption on the sign of the curvature (this hypothesis is crucial  in \cite{BM}) and on the convergence of the  $\kappa_k$\,. Actually several works on equations \eqref{ld2r} and \eqref{ld2s} have extended the result of Br\'ezis and Merle showing that, under the crucial assumption that the prescribed curvatures $K_k$ converge in $C^0$, the amount of curvature concentrating at each point is a multiple of $4\pi$, i.e. a multiple of the total Gaussian curvature of $S^2$, see e.g. \cite{LS} (Also higher-dimensional extensions were studied under the same strong assumptions of convergence of $K_k$ in $C^0$ or even $C^1$, see e.g. \cite{DR}, \cite{mal} and \cite{mar3}.)
   In \cite{BM} the functions $K_k$ can belong to $L^p(\R)$, with $1<p\le +\infty\,.$ We believe that
in the case of the nonlocal Liouville equation \rec{modelequbis}   the quantization result by $\pi$ does
not hold once we replace $\kappa\in L^{\infty}$ by $\kappa\in L^p\,$ with $1<p<+\infty$.  
\par
   The fact that we are able to prove a quantization result only under the minimal (and geometrically meaningful) bounds \eqref{boundlength}-\eqref{boundcurv} is   better understandable through the geometric interpretation that we  give to equation \eqref{liouvfrack}.
Indeed given a solution $\lambda$  to the equation \rec{modelequbis}, with $\kappa\in L^{\infty}(S^1)$, the function $e^{\lambda}$ provides
a ``conformal" parametrization of a closed curve $\gamma\colon S^1\to\C$ in normal parametrization and  whose curvature at the point $\gamma(z)$ is exactly
$\kappa(z)$\,.\par

Precisely let us define:

\begin{Definition} A function $\Phi\in C^1(\bar D^2,\C)$ is called a holomorphic immersion if $\Phi$ is holomorphic in $D^2$ and $\Phi'(z):=\de_z\Phi(z)\ne 0$ for every $z\in \bar D^2$. \par
A curve $\gamma\in C^1(S^1,\C)$ is said to be in normal parametrization if $|\dot\gamma|\equiv const$ and in conformal parametrization if there exists a holomorphic immersion $\Phi\in C^1(\bar D^2,\C)$  with 
 $\Phi|_{S^1}=\gamma$.
\end{Definition}

Then we have the following characterization:
\begin{Theorem}\label{propcostrlambdaintr} 
A function  $\lambda\in  L^1(S^1,\C)$ with $L:=\|e^{\lambda}\|_{L^1(S^1)}<\infty$  satisfies 
\begin{equation} \label{liouvfracn1}
(-\Delta)^\frac12 \lambda=\kappa e^{\lambda}-1\quad \text{in }S^1
\end{equation}
for some function $\kappa\colon S^1\to \R$, $\kappa\in L^{\infty}(S^1)$, if and only if there exists  a closed curve 
$\gamma\in W^{2,\infty}(S^1,\C)$, (see  Fig.\ref{fig4} and Fig.\ref{fig5},  Fig.\ref{fig:selfint},) with   $|\dot\gamma|\equiv\frac{L}{2\pi}$, a holomorphic  immersion $\Phi\colon \bar D^2\to \C$ a diffeomorphism $\sigma\colon S^1\to S^1$    ,  such that $\Phi\circ\sigma(z)=\gamma(z)$ for all $z\in S^1\,,$
 \begin{equation}\label{cond1}
|\Phi'  (z)|=e^{\lambda(z)},~~z\in S^1\,,
\end{equation}
and the curvature of $\Phi(S^1)$ is $\kappa$\,.
  While $\Phi$ uniquely determines $\lambda$ via \eqref{cond1}, $\lambda$ determines $\Phi$ up to a rotation and a translation.
Moreover it holds
\begin{equation}\label{lambdaPhi}
|\Phi'(z)|=e^{\tilde\lambda(z)},\quad z\in \bar D^2,
\end{equation}
where $\tilde\lambda:D^2\to \R$ is the harmonic extension of $\lambda$. ~~\hfill$\Box$
\end{Theorem}
\psfrag{Curva}{\bf Jordan Curve}
\psfrag{D}{$D^2$}
\psfrag{C}{\bf conformal}
\psfrag{phi}{$\Phi_k$}
\psfrag{gamma}{$\gamma_k$}
\psfrag{n}{\bf normal}
\psfrag{s}{$\sigma_k$}
\psfrag{S}{$S^1$}
\psfrag{diff}{\bf diffeomorphism}
\begin{figure}
\begin{center}
\includegraphics[width=6cm]{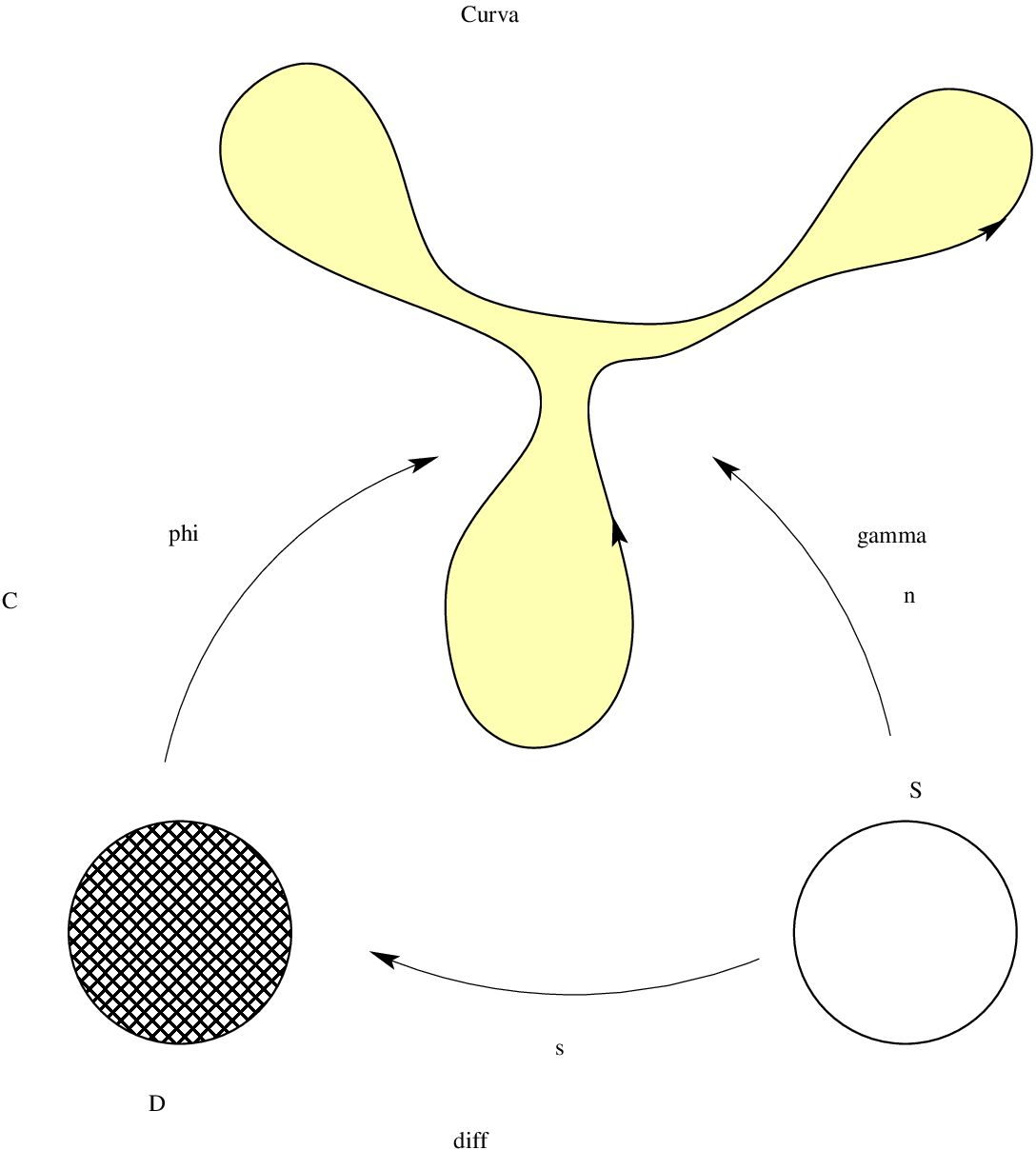}
 
\end{center}
\caption{A Jordan curve}\label{fig4}
\end{figure}

 \psfrag{D}{$D^2$}
\psfrag{C}{\bf conformal}
\psfrag{phi}{$\Phi_k$}
\psfrag{gamma}{$\gamma_k$}
\psfrag{n}{\bf normal parametrization}
\psfrag{s}{$\sigma_k$}
\psfrag{S}{$S^1$}
\psfrag{diff}{\bf diffeomorphism}
\begin{figure}
\begin{center}
\includegraphics[width=6cm]{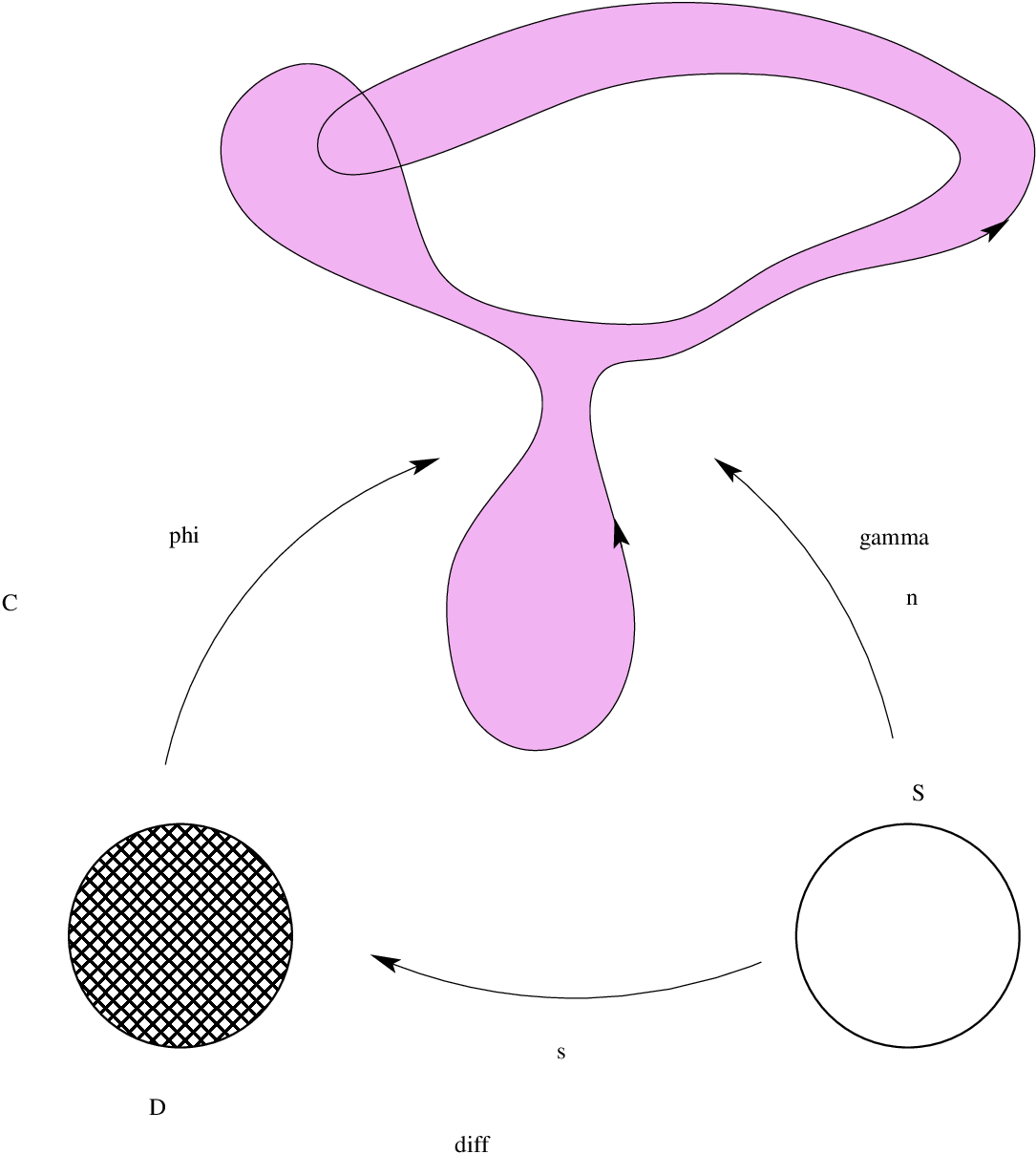}
\end{center}
\caption{A curve with self-intersections of degree $1$}\label{fig5}
\end{figure}

Theorem \ref{propcostrlambdaintr}  allows us to interpret and re-formulate Theorem \ref{mainth} from   the point of view of the behavior of the sequences of the curves $\gamma_k$ (in normal parametrization) and of the immersions $\Phi_k$ corresponding to
a sequence of solutions to \rec{liouvfrack}, see Fig. \ref{fig6}, \ref{fig7}.

 \psfrag{D}{$D^2$}
\psfrag{C}{\bf conformal}
\psfrag{phi}{$\Phi_\infty$}
\psfrag{gamma}{$\gamma_\infty$}
\psfrag{n}{\bf normal parametrization}
\psfrag{S}{$S^1$}
\begin{figure}
\begin{center}
\includegraphics[width=6cm]{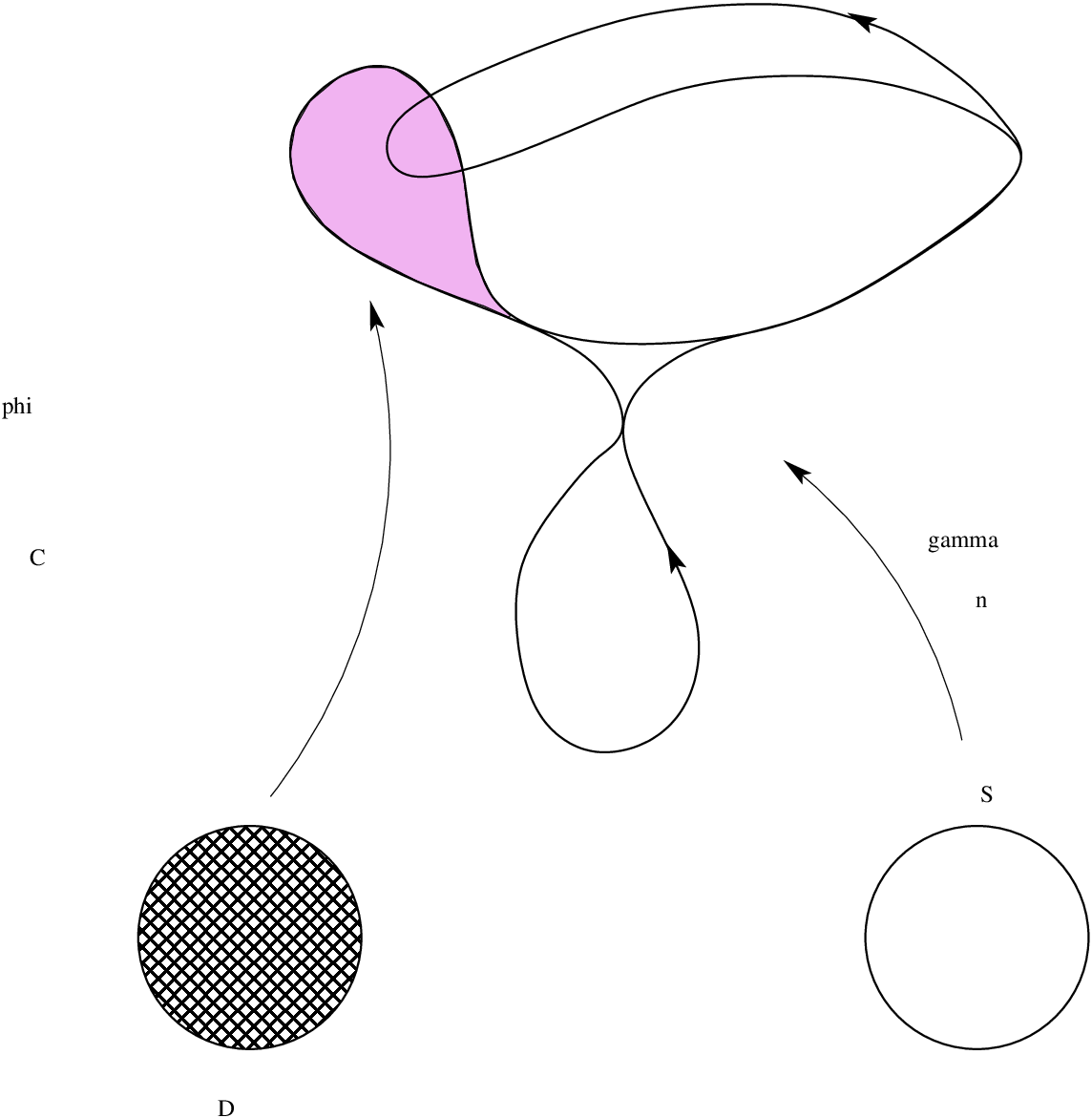}
\end{center}
\caption{Asymptotics of curves }\label{fig6}
 \end{figure}

\psfrag{D}{$D^2$}
\psfrag{C}{\bf conformal}
\psfrag{phi}{$\!\!\!\!\!\!\!\!\!\!\!\!\!\!\!\!\!\!\!\!\!\!\!\!\!\!\!\!\tilde\Phi_\infty=\lim_{k\to+\infty}\Phi_k\circ f_k^{j}$}
\psfrag{gamma}{$\gamma_\infty$}
\psfrag{n}{\bf normal parametrization}
\psfrag{S}{$S^1$}
\begin{figure}
\begin{center}
\includegraphics[width=8cm]{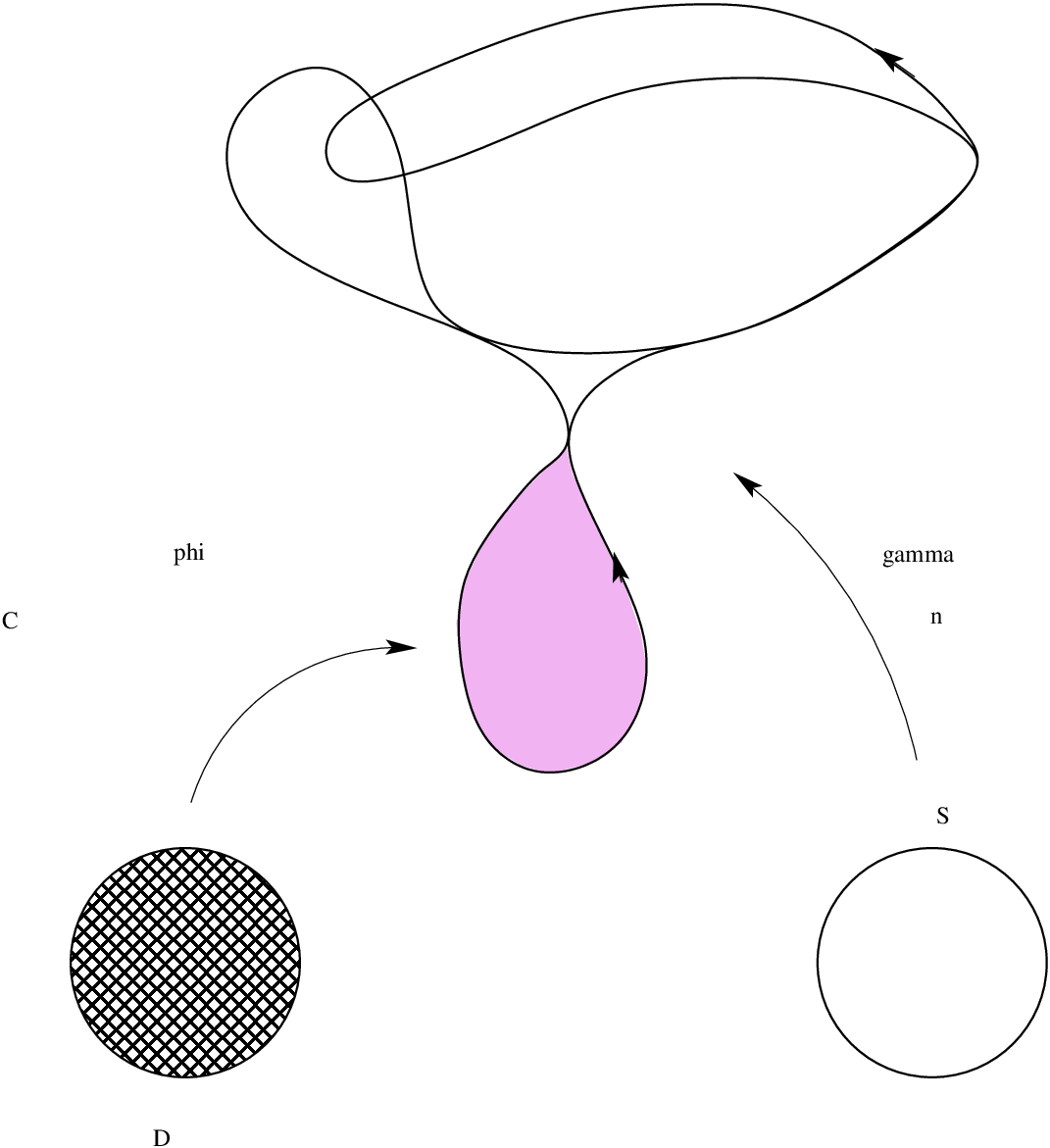}
\end{center}
\caption{Asymptotics of curves up to M\"obius Transformations}\label{fig7}

\end{figure}

 \begin{Theorem}\label{mainth2}
Let a sequence $(\lambda_k)\subset  L^1(S^1,\R)$ satisfy \eqref{boundlength}-\eqref{liouvfrack}-\eqref{boundcurv}, 
and let $\Phi_k\colon\bar D^2\to\C$ be a holomorphic immersion satisfying \eqref{cond1}, and $\sigma_k$, $\gamma_k$ with $\gamma_k=\Phi_k\circ\sigma_k$ be as given by Theorem \ref{propcostrlambdaintr}.  Then, up to extracting a subsequence, there exists an at most countable  family $J$ {such that for every $j\in J$ there exist a sequence of} M\"obius transformations $f_k^j:\bar D^2\to\bar D^2$ and a finite set finitely many points $B_j=\{a^j_1,\ldots a^j_{N_j}\}\subset S^1\,$ such that
$$\gamma_k\rightharpoonup \gamma_{\infty}~~\mbox{in $W^{2,p}(S^1)$\,,}\quad \tilde\Phi_k^j:= \Phi_k\circ f_k^j \rightharpoonup \tilde\Phi_{\infty}^j~~\mbox{in $W_{\loc}^{2,p}(\bar D^2\setminus B_j)$}\,.$$
 where  $p<\infty$, $\tilde\Phi_{\infty}^j\colon \bar D^2\setminus B_j\to \C$ are holomorphic  immersions satisfying
 \begin{equation}\label{current}
  (\gamma_{\infty})_{*}[S^1]=\sum_{j\in J} (\tilde\Phi_{\infty}^j)_{*}[S^1\setminus B_j]\,,
\end{equation}
where for every $\phi\colon S^1\to\C$ and for every differential form $\omega$ on $\C$ 
$$\langle \phi_*[S^1],\omega\rangle:=\int_{S^1}\phi^*\omega\,.$$
If $\lambda_k^j:=\log|(\tilde\Phi_k^j)'|_{S^1}|$, then up to a subsequence
$\lambda_k^j \rightharpoonup \lambda^j_{\infty}$ in $W^{1,p}_{\loc}(S^1\setminus B_j)$,
where 
$$
 (-\Delta)^\frac12 \lambda^j_{\infty}=\kappa_{\infty}^je^{v_{\infty}}-1-\sum_{i=1}^{N_j}\pi \delta_{a_i^j}\, ,$$
 and  $\kappa_k\circ f_k^j \stackrel{*}{\rightharpoonup}  \kappa^j_\infty$ in $L^\infty(S^1,\R)$ as $k\to +\infty\,.$ ~~\hfill$\Box$
\end{Theorem}

Theorem \ref{mainth2} says that it is always possible, up to  sequences of M\"obius transformations, to recover all the connected components enclosed by the limiting curve $\gamma_{\infty}$  (see in particular
\rec{current})\,.  Moreover these components    are separated by what we called  {\em pinched points}, (see Definition \ref{pinched}), namely (roughly speaking)  couple of  points $p\ne p'\in S^1$ such that $\gamma_{\infty}(p)=\gamma_\infty(p')$.  The 
angle between the tangent vectors in these couples of  points is shown to be necessarily $\pi$\,.\par
\par
 We finally prove a link between the equation \rec{modelequbis} and the analogous nonlocal equation in $\R\,.$ Precisely if $u\in L_\frac12(\R)$ (see \eqref{L12}), $e^u\in L^1(\R)$ and $u$ satisfies
\begin{equation}\label{liou5int}
(-\Delta)^\frac{1}{2}u= Ke^u\quad \text{in }\R
\end{equation}
for some $K\in L^\infty(\R)$, then $\lambda(z):=u(\Pi(z))-\log(1+\sin z)$ ($\Pi\colon S^1\setminus \{-i\}\to \R$ is the stereographic projection) satisfies
\begin{equation}\label{liou6}
(-\Delta)^\frac{1}{2}\lambda=K\circ\Pi\, e^{\lambda}-1 +\left(2\pi -\|\lapfr u\|_{L^1}\right)\delta_{-i}\quad \text{in }S^1.
\end{equation}
Owing to this correspondence from Theorem \ref{mainth} (see also Proposition \ref{propcostrlambdabis} below) we can deduce the following {compactness} result in  $\R\,.$
\begin{Theorem}\label{trm2} Let $u_k\in L_\frac12(\R)$ be a sequence of solutions to 
$$(-\Delta)^\frac{1}{2}u_k=K_ke^{u_k}\quad \text{in } \R{}$$
 with $\|K_k\|_{L^\infty}\le C$ and $\|e^{u_k}\|_{L^{1}}\le C \,. $
Then \par
1. Up to subsequence we have  $K_ke^{u_k} \rightharpoonup \mu $ weakly  in $ W^{1,p}_{\loc}(\R \setminus  B)$ for every $p<\infty$, where $\mu$ is a finite  Radon measure in $\R,$  $B:=\{a_1,\ldots, a_N\}$ is a (possibly empty) subset of $\R$ and
$K_k\stackrel{*}{\rightharpoonup}  K_{\infty}$ in $L^{\infty}(\R)\,.$ Moreover the following alternatives holds:
\par
i) $\mu|_{\R\setminus B}=K_\infty e^{u_\infty} $ for some $u_\infty\in W^{1,p}_{\loc}(\R\setminus B)$ satisfying.

 \begin{equation} \label{liouvlimit}
(-\Delta)^\frac12 u_\infty=K_{\infty} e^{u_\infty}+\sum_{i=1}^{N}\pi  \delta_{a_i}\quad \text{in }\R\,.
\end{equation}

ii) $\mu|_{\R\setminus B} \equiv 0$, $N\le 2$ and    $u_k\to -\infty$ locally uniformly in $\R\setminus B$. ~~\hfill$\Box$
\end{Theorem}

{Notice that we do not assume that $K_k$ is positive, or that it converges in $C^0_{\loc}(\R)$}. The quantization analysis  in the case $\mu=0$ is analogous to that of Theorem \ref{mainth}, it requires only some more 
 work due to the fact that the curvature of the corresponding curves is not anymore uniformly bounded in $S^1$\,.   
\par
In particular we can  deduce the following
\begin{Corollary}
Under the hypotheses of Theorem \ref{trm2} if $K_k\ge 0$, and
$$\int_{\R}K_ke^{u_k}dx\le 2\pi,$$
then either $N=1$ and $u_k\to -\infty$ locally uniformly $\R\setminus\{a_1\}$ or $N=0$ and $u_k\rightharpoonup u_{\infty}$
in $W^{1,p}(\R)$ as $k\to +\infty$ where
$u_\infty$ solves
\begin{equation} 
(-\Delta)^\frac12 u_\infty=K_{\infty} e^{u_\infty}\,.
\end{equation}
\end{Corollary}
 It would be interesting to compare Theorems \ref{mainth} and \ref{mainth2} to the blow-up analysis    obtained recently by Mondino and the third author in \cite{MR} in the case 
of  sequences of weak   conformal immersions from   $S^2$ into $\R^m$.
In \cite{MR} the authors
study the possible limit of the Liouville equation
\begin{equation}\label{ld2bis}
-\Delta_{g_0} u=K e^{2u}-1,~~\text{on } S^2
\end{equation}
satisfied by the conformal factor
of the  immersion $\Phi$ $ (g_{\Phi} =e^{2u}g_0)$ under  the assumption that  
the second fundamental  form   is   bounded in $L^2$. Also in their case a sort of bubbling phenomenon occurs and   the choice of different sequences of M\"obius trasformations of $S^2$ permits to detect all the limiting enclosed currents.
However the $2$-dimensional blow-up analysis differs substantially from the $1$-dimensional case: in the $2$-dimensional case the area is {quantized}, namely there is no production of area in  the neck region between the
different bubbles, whereas in the $1$-dimensional case the {quantization} of the length does not hold. Precisely
in \cite{MR} the authors show that
$$
\sum_{\text{``Bubbles"}}\int_{S^2}e^{2u_{\infty}}dv=\liminf_{k\to+\infty} \int_{S^2}e^{2u_k}dv,
$$
whereas in the present situation one can produce examples such that
$$
\sum_{\text{``Bubbles"}}\int_{S^1}e^{\lambda_{\infty}}d\theta<\liminf_{k\to +\infty} \int_{S^1}e^{\lambda_k}d\theta\,.
$$

  We insists on the fact that ``conformal" parametrizations of planar curves are relevant in different applications. For instance it should be one  of the main tools of the Willmore Plateau problem, of  the analysis of the renormalizing area of surfaces in hyperbolic space $\mathcal{H}^2\,$ and of the free-boundaries.
 In particular for the latter the first author has observed in 
 \cite{DL} that  there is a one to one correspondence between free boundaries and $1/2$-harmonic maps and here we show  that the holomorphic immersion $\phi$  for which $e^{u(z)}=|\frac{\partial}{\partial\theta}\phi(z)|,$ $z\in S^1$, is a $1/2$-harmonic map into $\phi(S^1)\,.$\par

\medskip

An interesting consequence of Theorem \ref{propcostrlambdaintr} is a proof of the classification of the solutions to the non-local equation
\begin{equation}\label{liou}
(-\Delta)^\frac{1}{2} u = e^{u}\quad\text{in }\R,
\end{equation}
under the integrability condition 
\begin{equation}\label{area}
L:=\int_{\R{}}e^u dx<\infty.
\end{equation}
Equation \ref{liou} is a special case of the problem
\begin{equation}\label{lioun}
(-\Delta)^\frac{n}{2} u= (n-1)! e^{nu}\quad\text{in }\R^n, \quad V:=\int_{\R^n} e^{nu}dx<\infty,
\end{equation}
which has been studied by several authors in the last decades (see e.g. \cite{CL}, \cite{CY2}, \cite{lin1}, \cite{JMMX} and \cite{mar1}). 
Geometrically if $u$ solves \eqref{lioun} and $n\ge 2$, then the metric $e^{2u}|dx|^2$ on $\R^n$ has constant $Q$-curvature $(n-1)!$ and volume $V$, see e.g. \cite{cha}. All the above mentioned works rely on the application of a moving-plane technique, in order to show that  under certain growth conditions at infinity (needed only when $n\ge 3$) the solutions to \eqref{lioun} have the form
\begin{equation}\label{specsol}
u_{\mu,x_0}(x):=\log\bigg(\frac{2\mu}{1+\mu^2|x-x_0|^2}\bigg), \quad x\in \R^n,
\end{equation}
for some $\mu>0$ and $x_0\in\R^n$. 
For the case $n=1$, instead of using the moving plane technique, we will use the stereographic projection to transform \eqref{liou} into Equation \eqref{liouvfracn1}, and use the geometric interpretation of the latter (Theorem \ref{propcostrlambdaintr}) to compute all its solutions (Corollary \ref{lemmaS1} below). This will yield
\begin{Theorem}\label{trm1} Every function $u\in L_{\frac{1}{2}}(\R)$ solving to \eqref{liou}-\eqref{area} is of the form \eqref{specsol} for some $\mu>0$ and $x_0\in\R{}$.
\end{Theorem}
We also remark that by changing the sign of the nonlinearity in \eqref{liou} the problem has no solutions. More precisely:

\begin{Proposition}\label{nonex} Given a function $K\in L^\infty(\R{})$ with $K\le 0$, the equation
$$(-\Delta)^\frac{1}{2}u=K e^u\quad \text{in }\R{}$$
has no solution satisfying \eqref{area}.
\end{Proposition}

The proof of Proposition \ref{nonex} is a simple application of the maximum principle for the operator $(-\Delta)^\frac12$, but it is worth remarking that for $n\ge 4$ even solutions to Problem \eqref{lioun} with $(n-1)!$ replaced by $-(n-1)!$ (or any negative constant) do exist, as shown in \cite{mar2}.

\medskip

 The paper is organized as follows. In the Section 2 we introduce the nonlocal Liouville equation \rec{modelequbis}  in
 $S^1$ and we explain its geometric interpretation. In Section 3 we perform the blow-up and quantization analysis of the equation \rec{modelequbis} and in particular we prove Theorems \ref{mainth} and \ref{mainth2}\,. Section 4 is devoted to the description of the relation between the equations
  \rec{modelequbis}  and \rec{liou5int}\,. Finally in Section 5 we prove Theorem \ref{trm1} and Proposition \ref{nonex}.

\par
\medskip

{\bf Notations.} Given $x,y\in \R^N$ we denote by $\langle x,y\rangle$ the scalar product of $x,y$.
Let $h:\Omega\subset\C\to \R$, and $\gamma\colon S^1\to \C$ a curve. We denote
by $\int_{\gamma}h(z)|dz|$ or by $\int_{\gamma}h(z)d\theta$ the line integral of $h$ along $\gamma\,.$\par
  
\section{Nonlocal Liouville equation in $S^1$}
 In this section we study the following nonlocal Liouville type equation 
 on $S^1$
$$(-\Delta)^\frac12 u= \kappa e^u-1~\mbox{in $S^1$}$$
  where $u\in L^1(S^1)$,   $(-\Delta)^\frac{1}{2}u$ stands for the fractional
 Laplacian and $\kappa\colon S^1\to \R$  is a bounded function. In the Appendix \ref{fraclaps1} we recall the definition and some properties of the fractional Laplacian in  $S^1\,.$\par
  
 \subsection{Geometric Interpretation of the Liouville equation in $S^1$}
The first key step in our analysis is the geometric interpretation of the equation
 \rec{modelequbis}\,.
Roughly speaking such an equation prescribes the curvature of a closed curve in conformal parametrization.

It is easy to verify that for  $\phi\in L^1(S^1)$ we have
\begin{equation}\label{lapHilb}
(-\Delta)^\frac12 \phi(\theta)=\sum_{n\in\mathbb{Z}}|n|\hat\phi(n)e^{in\theta}={\cal{H}}\left(\frac{\partial \phi}{\partial\theta}\right)=\frac{\de \mathcal{H}(\phi)}{\de\theta},
\end{equation}
where ${\cal{H}}$ is the Hilbert Transform  on $S^1$ defined by
$${\cal{H}}(f)(\theta) :=\sum_{n\in \Z} -i \sign (n)\hat{f}(n)e^{in\theta},\quad f\in\mathcal{D}'(S^1) \,.$$
\par
We recall that the Hilbert transform has the following property, a proof of which can be found in \cite[Chapter III]{Kat}.

\begin{Lemma}\label{LemmaHilbert}  The Hilbert transform $\cal{H}$ is  bounded from $L^p(S^1)$ into itself, for $1<p<+\infty$, and it is of weak type $(1,1)\,.$ A function  $f:=u +i v$ with $u,v\in L^1(S^1,\R)$ can be extended to a holomorphic function in $D^2$ if and only if $v=\mathcal{H}(u)+a$ for some $a\in \mathbb{C}$ .
\end{Lemma}



\noindent {\bf Proof of Theorem \ref{propcostrlambdaintr}.} 
1. Let $\Phi\in C^1( \bar D^2, \C)$ be a holomorphic immersion with. Set $\lambda:=\log |\Phi'|_{S^1}|$.
Since $\Phi'\colon D^2\to \C\setminus\{0\}$ is  holomorphic, it  holds
$\Phi'|_{S^1}=e^{\lambda+i\rho+i\theta_0}$, for some $\theta_0\in[0,2\pi)$ where $\rho:=\mathcal{H}(\lambda)$ is the Hilbert transform of $\lambda$. Indeed by Lemma \ref{LemmaHilbert} the function $f:=\lambda+i\rho$ has a holomorphic extension $\tilde f$ to $D^2$, hence $e^{\tilde f}$ is holomorphic in $D^2$ and $e^{\tilde f}|_{S^1}=e^f=e^{\lambda+i\rho}$. But $|e^f|=e^{\lambda}=|\Phi'|_{S^1}|$, so that by Lemma \ref{lemmahconst}  we have $\Phi'/e^{\tilde f}=e^{i\theta_0}$ for some constant $\theta_0$. Up to a rotation of $\Phi$ we can assume that $\theta_0=0$. Up to such a rotation and a translation $\Phi$ is determined by $\lambda$.

\begin{equation}\label{deflambda}
\frac{\partial \Phi (z)}{\partial\theta}(z)= ie^{\lambda(z)+i\rho(z)+i\theta}\,.
\end{equation}
Now let 
$$s(\theta):=\int_0^\theta \left|\frac{\partial \Phi ( e^{i\theta'})}{\partial\theta'}\right| d\theta^{\prime}.$$
 We have  $s\colon[0,2\pi]\to [0,L],$ where $L=\|\frac{\partial \Phi}{\partial\theta} \|_{L^1(S^1)}$ is the length of the curve $\Phi(S^1)$, and up to a scaling we will assume that $L=2\pi$.
Let $ \theta:=s^{-1}\colon [0,2\pi]\to [0,2\pi]$. One can easily also see that $\theta\in C^1([0,2\pi],[0,2\pi])$\,. Then using \eqref{deflambda}  and that
$$\dot s(\theta)=|\Phi'(e^{i\theta})|=e^{\lambda(e^{i\theta})}>0,\quad   \dot \theta(s)=e^{-\lambda(e^{i\theta(s)})}$$
we compute
$$\tau(s): =\frac{d}{d\theta} \Phi(e^{i\sigma(s)})=\Phi'(e^{i\theta(s)})ie^{i\theta(s)}\dot \theta(s)=\frac{\partial \Phi}{\partial\theta}(e^{i\theta(s)}) e^{-\lambda(e^{i\theta(s)})}\,.$$
Notice that $|\tau|\equiv 1$, i.e. the curve $\gamma\colon e^{is}\mapsto \Phi(e^{i\theta(s)})$ is parametrized by arc-lenght, and $\tau$ is its unit tangent vector. Using \eqref{lapHilb}, \eqref{deflambda}  and identifying $s$ with $e^{is}$, 
the curvature of $\gamma$  
\begin{equation*}
\begin{split}
\kappa(s)=\langle i\tau(s), \dot{\tau}(s)\rangle
&=\langle i \tau(s)\frac{d }{d\theta}\left(ie^{i\rho(e^{i\theta(s)})+i\theta(s)}\right)\rangle\\
&=\left(\frac{d\rho (e^{i\theta(s)})}{d\theta} +1\right)\dot\theta(s) \\
&= \left((-\Delta)^\frac12 \lambda (e^{i\theta(s)})+1\right)\ e^{-\lambda(e^{i\theta(s)})},
\end{split}
\end{equation*}
i.e. $\lambda$ satisfies \eqref{liouvfracn1} with $\kappa(e^{i\theta}):=i\tau(s(\theta))\cdot \dot{\tau}(s(\theta))$\,.   {Since $|\kappa(s)|=|\ddot\gamma(e^{is})|\in L^\infty(S^1)$ we also have $\gamma\in W^{2,\infty}(S^1,\C)$.}

\medskip
2. Conversely, let us assume that $\lambda\in L^1(S^1)$ with $e^{\lambda}\in L^1(S^1)$ weakly satisfies \eqref{liouvfracn1} for some $\kappa\in L^\infty(S^1)$. By regularity theory
$\lambda\in W^{1,p}(S^1) $ for any $p<\infty$.
We set $\rho:=\mathcal{H}(\lambda)$. Let $\phi\in W^{1,p}(\bar D^2,\C) $ be the holomorphic extension of  function $e^{\lambda+i\rho}\in W^{1,p}(S^1)$ and set
\begin{equation}\label{defphi}
\Phi(z):=\int_{\Sigma_{0,z}}\phi(w) dw, ~z\in\bar D^2
\end{equation} 
where $\Sigma_{0,z}$ is any path in $\bar D^2$ connecting $0$ and $z\,.$
Then $\Phi\in W^{2,p}(\bar D^2,\C) $ satisfies \eqref{deflambda}. 
 {From part 1 we see that  $\kappa$ is  the curvature of the curve $\Phi(S^1)$ in normal parametrization.}

Let $\hat\Phi\colon\bar D^2\to \C$ be another holomorphic immersion such that $|\hat\Phi'(z)|=e^{\lambda(z),}$, $z\in S^1\,.$  We claim that
\begin{equation}\label{phiphi1}
\Phi= e^{i\theta_0}\hat \Phi+ a\quad \text{in }\bar D^2,\quad\text{for some }\theta_0\in \R,\,a\in\C\,.
\end{equation}
Indeed the function 
$h:=\frac{\Phi'}{\hat \Phi'}$ never vanishes in $\bar D^2$ and
satisfies
$$|h(z)|=\frac{|\Phi'(z)|}{|\hat\Phi'(z)|} =\frac{e^{\lambda(z)}}{e^{\lambda(z)}}=1,\quad z\in S^1.$$
It follows from Lemma \ref{lemmahconst}  that $h$ is a constant of modulus $1$, say $h\equiv e^{i\theta_0}$, and \eqref{phiphi1} follows at once.
\hfill~$\Box$
\begin{Remark}
In Theorem \ref{propcostrlambdaintr},  we cannot expect that $\Phi$ is a biholomorphism from $\bar D^2$ onto $\Phi(\bar D^2)$. For instance the function $\Phi(z):=e^{az}$ for any $a>0$  is an immersion and  $\Phi({S^1)}$ has self-intersections whenever $a\ge \pi$, as easily seen by writing
$$\Phi(e^{i\theta})=e^{a\cos\theta}(\cos(a \sin\theta)+i\sin(a \sin\theta)),$$
see Figure \ref{fig:selfint}.

\begin{figure}\label{f:curve}
\begin{center}
\includegraphics[width=7cm,height=7cm]{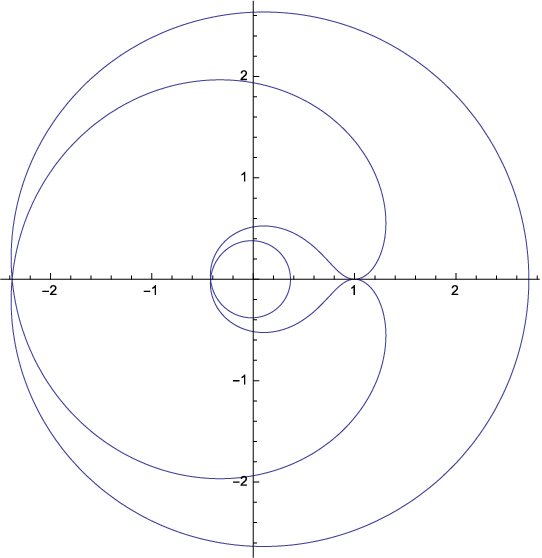}
\end{center}
\caption{\small Plot of the curve $e^{\cos\theta}(\cos(2\pi \sin\theta)+i\sin(2\pi \sin\theta))$, $\theta\in [0,2\pi]$. It has the same kind of self-intersections as the curve $\Phi(e^{i\theta})=e^{2\pi e^{i\theta}}$, whose plot is difficult to inspect, since $|\Phi(z)|$ oscillates between $e^{2\pi}$ and $e^{-2\pi}$.}
\label{fig:selfint}
\end{figure}\par
\end{Remark}

\begin{Corollary}\label{lemmaS1}
All functions $\lambda\in L^1(S^1)$ with $e^\lambda\in L^1(S^1)$ solutions to 
\begin{equation}\label{eqlambda2}
(-\Delta)^\frac12\lambda=C_0 e^{\lambda}-1\quad\quad\mbox{ on } S^1,
\end{equation}
where $C_0$ is an arbitrary positive constant, are given by
\begin{equation}\label{solutions}
\lambda(\theta)=\log\left( \left|\frac{\partial}{\partial\theta}  \frac{z-a_1}{1-\bar a_1 z}\right|\right) -\log C_0
\end{equation}
for some $a_1$ in $D^2$.
\end{Corollary}

\noindent{\bf Proof.} Up to the translation $\tilde \lambda=\lambda+\log C_0$ we can assume $C_0=1$. By Theorem \ref{propcostrlambdaintr} the function $\lambda$ determines a holomorphic immersion $\Phi\in C^1(\bar D^2, \C)$, such that $\Phi(S^1)$ is  curve  of curvature $1$, hence up to a translation $\Phi(S^1)\subset S^1$, and therefore it is M\"obius transformation of the disk.
From \eqref{cond1}  we infer that $\lambda=\log\left(\left|\Phi'|_{S^1}\right|\right)$, and we conclude.~\hfill$\Box$

\begin{Corollary}\label{lemmaS1}
All functions $\lambda\in L^1(S^1)$ with $e^\lambda\in L^1(S^1)$ solutions to 
\begin{equation}\label{eqlambda2}
(-\Delta)^\frac12\lambda=C_0 e^{\lambda}-1\quad\quad\mbox{ on } S^1,
\end{equation}
where $C_0$ is an arbitrary positive constant, are given by
\begin{equation}\label{solutions}
\lambda(\theta)=\log\left( \left|\frac{\partial}{\partial\theta}  \frac{z-a_1}{1-\bar a_1 z}\right|\right) -\log C_0
\end{equation}
for some $a_1$ in $D^2$.
\end{Corollary}

\begin{proof} Up to the translation $\tilde \lambda=\lambda+\log C_0$ we can assume $C_0=1$. By Theorem \ref{propcostrlambdaintr} the function $\lambda$ determines a holomorphic immersion $\Phi\in C^1(\bar D^2, \C)$, such that $\Phi(S^1)$ is  curve  of curvature $1$, hence up to a translation $\Phi(S^1)\subset S^1$, and therefore it is M\"obius transformation of the disk.
From \eqref{cond1}  we infer that $\lambda=\log\left(\left|\Phi'|_{S^1}\right|\right)$, and we conclude.

\end{proof}

The following corollary is an easy consequence of Theorem \ref{propcostrlambdaintr}.

\begin{Corollary}\label{cormob}
Let $\Phi$, $\lambda$ and $\kappa$ be as in Theorem \ref{propcostrlambdaintr}, and let $f: \bar D^2\to \bar D^2$ be a M\"obius diffeomorphism. Set $\tilde\Phi:=\Phi\circ f$, $\tilde \lambda:= \log|\tilde\Phi'|_{S^1}|$ and $\tilde \kappa:= \kappa\circ f|_{S^1}$. Then
$$\tilde\lambda =\lambda \circ f|_{S^1} +\log(|f'|_{S^1}|)$$
and
$$(-\Delta)^\frac12 \tilde\lambda=\tilde \kappa e^{\tilde\lambda}-1.$$
\end{Corollary}

 \begin{Remark}
 One can give an analogous geometric characterization also for an equation of the type
 \begin{equation}\label{equn}
 (-\Delta)^\frac12\lambda=\kappa e^\lambda-n~~~\mbox{in $S^1$}\,,\end{equation}
 with $n>1$\,. In this case there is a correspondence between
 the solutions of (\ref{equn}) and holomorphic functions 
   $\Phi \colon D^2\to \C$ of the form
 $\Phi(z)=\Psi(z)h(z)$ where
  $\Psi$ is  Blaschke product
$$\Psi(z):=\prod_{k=1}^{n-1} \frac{z-a_k}{1-\bar a_k z},\quad a_1,\dots, a_{n-1}\in D^2, $$
and $h'(z)\ne 0$ for every $z\in \bar D^2\,.$
In this case $n-1=i\Psi \cdot \frac{\de \Psi}{\de\theta}=\deg(\Psi)$ \,.\end{Remark}\par
\medskip
Next  we show that 
  the existence of a holomorphic immersion of the disk $\bar D^2\,,$     is equivalent to the existence of a positive diffeomorphism  of   the disc $\bar D^2$. Such a result can be seen as a sort of generalization Riemann Mapping Theorem in the case of closed  curves which are not necessarily
 injectives.
 We premise the following Lemma  giving a better regularity up to the boundary of a holomorphic immersion $u\colon D^2 \to \C$ under the assumption that the curve $u|_{S^1}$ has a $W^{2,\infty}$ constant speed parametrization.
 
 \begin{Lemma}\label{regbd}
Let $u \in C^0( \bar D^2, \C)$ be holomorphic in $D^2$ with $\partial_z u\ne 0$ in $D^2$ and suppose there is $\gamma\in W^{2,\infty}(S^1,\C)$ with $|\dot \gamma|\equiv const$ and a homeomorphism $\sigma:S^1\to S^1$ such that $\gamma=u \circ \sigma$. Then
  $u\in W^{2,p}(\bar D^2,\C)$ for every $p<+\infty$ and $\partial_z u(z)\ne 0$ for all $z\in S^1$ \,.
\end{Lemma}

\noindent  {\bf Proof.} 
Let $z_0\in S^1$. Since $\dot \gamma(z_0)\ne 0$, we can find some $\rho>0$ such that $\gamma(S^1\cap B(z_0,\rho))$ coincides up to a rotation with a piece of the graph of a function $\varphi\in C^{1,\alpha}(\R)$ with $\varphi'(u_1(x_0))=0$. We may also assume that $u=u_1+iu_2$ takes value into the set $\{(\xi,\eta)\in\R^2:~\eta\ge \varphi(\xi)\}\,.$ Define
 $$
 \hat u=\hat u_1+i\hat u_2,\quad\hat u_1:=u_1,~\hat u_2:=u_2-\varphi(u_1)\,.$$
{\bf Claim:} $\hat u_2$ satisfies
\begin{equation}\label{equ2}
\left\{\begin{array}{cc}
 \partial_{x_i}(a_{ij}\partial_{x_j}\hat u_2)=0,&~~\mbox{in $B(x_0,\rho)\cap  D^2$}\,\\ \hat u_2=0,&
 ~~\mbox{in $B(x_0,\rho)\cap  S^1$}\end{array}\right.
 \end{equation}
where the matrix
 \begin{equation}\label{matrixA}
(a_{ij})= \begin{pmatrix}
 1-\frac{1}{1+(\varphi^{\prime})^2(u_1)}&\frac{\varphi^{\prime}(u_1)}{1+(\varphi^{\prime})^2(u_1)}\\[5mm]
- \frac{\varphi^{\prime}(u_1)}{1+(\varphi^{\prime})^2(u_1)}& 1- \frac{1}{1+(\varphi^{\prime})^2(u_1)}\end{pmatrix}\end{equation}
is in $L^{\infty}(\bar D^2)$ and 
uniformly elliptic\,.\par
\noindent {\bf Proof of the claim:}
We can write $u=\hat u+i\varphi(u_1)$.
Since by hypothesis  $\partial_{\bar z} u(z)=0$, for all $z\in D^2$, the following  estimates hold
\begin{eqnarray*}
\partial_{\bar z}u_1&=&-i\partial_{\bar z}u_2\nonumber\\
\partial_{\bar z} \hat u(z)&=&-i\varphi^{\prime}(u_1)\partial_{\bar z}u_1=-\varphi^{\prime}(u_1)\partial_{\bar z}u_2\nonumber\\
\partial_{\bar z}u_1+i\partial_{\bar z} \hat u_2(z)&=&-i\varphi'(u_1)\partial_{\bar z}u_1\nonumber\\
\partial_{\bar z}u_1&=&-\frac{i}{1+i\varphi^{\prime}(u_1)}\partial_{\bar z} \hat u_2(z)\nonumber\\
\partial_{\bar z}\hat u&=&-\frac{\varphi^{\prime}(u_1)}{1+i\varphi^{\prime}(u_1)}\partial_{\bar z} \hat u_2(z)\,.\nonumber
\end{eqnarray*}
Therefore
\begin{equation}\label{lapu2}
\Delta \hat u_2=4\Im(\partial_z\partial_{\bar z} \hat u )=-4\Im\left[\partial_{z}\left[\frac{\varphi^{\prime}(u_1)}{1+i\varphi^{\prime}(u_1)}\partial_{\bar z} \hat u_2(z)\right]\right].
\end{equation}
Writing 
\[
\frac{\varphi^{\prime}(u_1)}{1+i\varphi^{\prime}(u_1)} \partial_{\bar z} \hat u_2(z)=\frac{\varphi^{\prime}(u_1)}{1+(\varphi^{\prime})^2(u_1)}
\frac{\partial_{x_1}\hat u_2+\varphi'(u_1)\partial_{x_2}\hat u_2 +i(\partial_{x_2}\hat u_2-\varphi'(u_1)\partial_{x_1}\hat u_2)}{2},
\]
we compute the right hand side of
\rec{lapu2} and get
\[\Delta \hat u_2=-\Im\left[(\partial_{x_1}-i\partial_{x_2})\frac{\varphi^{\prime}(u_1)}{1+(\varphi^{\prime})^2(u_1)}
[(\partial_{x_1}\hat u_2+\varphi^{\prime}(u_1)\partial_{x_2}\hat u_2)+i(\partial_{x_2}\hat u_2-\varphi^{\prime}(u_1)\partial_{x_1}\hat u_2)]\right]\,.
\]
Therefore
$\hat u_2$ satisfies \eqref{equ2}-\rec{matrixA} and the claim is proven\,.\par

Elliptic estimates imply that $\hat u_2\in W^{2,p}(\bar {B}(z_0,r/4)\cap \bar D^2)\,,$ for every $p<+\infty\,,$ in particular it is in $C^{1,\alpha}(\bar {B}(z_0,r/4)\cap \bar D^2)$ for every $\alpha\in (0,1)\,.$
Now since $\hat u_2\ge 0$ in $\bar D^2$ and $\hat u_2(z_0)=0$, Hopf's Lemma yields that
$\partial_r\hat u_2(z_0)\ne 0$. Since $u=\hat u+i\varphi(u_1)$,
it follows that
 $$\partial_r u(z_0)=\de_r\hat u_1(z_0)+i\de_r\hat u_2(z_0)+i\underbrace{\varphi'(u_1(z_0))}_{=0}\de_r \hat u_1(z_0)\ne 0,$$
and since $z_0\in S^1$ was arbitrary, we conclude that $\de_r u\ne 0$ everywhere on $S^1$. Then since $u$ is conformal up to the boundary 
we also have $\partial_{z} u\ne 0$ on $S^1$. \hfill$\Box$
 \par
 \medskip
 We introduce the following set
   \begin{eqnarray*}
   \mathcal{T}&:=&\{\gamma\colon S^1\to \C, ~\gamma\in W^{2,\infty},~|\dot\gamma|\equiv const,\\   &&~~~\mbox{such that there is   $\Psi\in C^1(\bar D^2,\C)$,~ $\det(\mathrm{Jac}(\Psi(z)))>0$,  $z\in D^2$,} \\
   &&~~~\mbox{$(\Psi\circ\sigma)(z)=\gamma(z)$, $z\in S^1$ for some diffeomorphism $
 \sigma\colon S^1\to S^1$}\}.
   \end{eqnarray*}


\begin{Theorem}[Generalized Riemann Mapping Theorem ] \label{GRMT}
 A curve $\gamma\in \mathcal{T} $ if and only if  
there exists a holomorphic immersion  $\Phi\colon \bar D^2\to\C$ and a diffeomorphism $\sigma: S^1\to S^1$  such that  $\Phi \circ \sigma=\gamma$\,.\end{Theorem}
{\bf Proof.} \par
1. Suppose that there exists a holomorphic immersion  $\Phi\colon \bar D^2\to\C$  and a diffeomorphism $\sigma: S^1\to S^1$ such that  $\Phi\circ \sigma=\gamma$. 
The one can take $\Psi=\Phi$. 
Therefore $\gamma\in \mathcal{T}$.
\par
2. Conversely let $\Psi\in C^{1 }(\bar D^2,\C)$,
$\Psi|_{S^1}= \gamma$ with $\det(\Jac(\Psi))>0$ in $D^2\,.$\par
2i) Consider the pull back of the Euclidean metric $g$ on $\R^2$ by $\Psi$:
$$h_{ij}:=\langle\partial_{x_i}\Psi,\partial_{x_j}\Psi\rangle\,.$$
Since
$\det(\mathrm{Jac}(\psi))>0$ we have
$$c^{-1}\delta_{ij}\le (h_{ij})\le c\delta_{ij}\,.$$
We can write
\begin{equation}\label{metric}
h=h_{11}dx^2+2h_{12}dxdy+h_{22}dy^2\,.\end{equation}
Setting $z=x+iy$ on can write $h$  in the form
$$
h=\nu|dz+\mu d\bar z|^2$$
where $\nu$ is a positive continuous function on $U$ and $\mu$ is a complex-valued   continuous  function with
$\|\mu\|_{L^{\infty}(\bar D^2)}<1$ on $U$\,. Actually $\nu$ and $\mu$ are given by
\begin{eqnarray*}
\nu&=&\frac{1}{4}\left(h_{11}+h_{22}+2\sqrt{h_{11}h_{22}-h_{12}^2}\right),\\
\mu&=&\frac{h_{11}-h_{22}+2i h_{12}}
{h_{11}+h_{22}+2\sqrt{h_{11}h_{22}-h_{12}^2}}.
\end{eqnarray*}
Moreover $\Psi$  solves the following equation
\begin{equation}\label{beltrami}
\frac{\partial_{\bar w}\Psi(w)}{\partial_w\Psi(w)}=\mu(w)\,,~~\mbox{in $D^2$.}
\end{equation}
The function $\mu$ is the so-called Beltrami coefficient associated to the metric $h$\,.
 Now we extend $\mu$ by $0$ outside $\bar D^2$ (we still denote this  extension by $\mu$).  Then there exists  a unique homeomorphism  $\xi\colon\bar\C\to\bar \C$ (here $\bar C=\C\cup\{\infty\}\simeq S^2$) which satisfies in distributional sense 
 $$
 \partial_{\bar z}\xi=\mu(z)\,\partial_z \xi,~~\mbox{in $\C$}
 $$
 and the following normalization conditions  $$
 \xi(0)=0,~\xi(1)=1, ~\xi(\infty)=\infty\,.$$ Moreover $\xi\in W^{1,p}_{\loc}(\C)$ for some $p> 2,$ $\partial_z \xi\ne 0\,,$ a.e in  $\C\,.$ The function $\xi $ is called a quasiconformal map with dilation coefficient $\mu$, (see e.g. Theorem   4.30 in \cite{IT}).\par Since $\xi$ is a homemorphism, 
  $\xi(S^1)$ is  a Jordan curve

 2ii) Consider now $\tilde\Psi:=\Psi\circ \xi^{-1}\colon \xi(\bar D^2)\to\C$\,.
From \cite[Proposition 4.13]{IT} it follows that the complex dilatation of $\tilde \Psi$ is $0$ in $\xi(D^2)$, therefore $\partial_{\bar z} \tilde\Psi=0$ and $\tilde\Psi$ is holomorphic in
 $\xi(D^2)\,,$ see \cite[Lemma 4.6]{IT}.

   2iii) Now we apply the Riemann Mapping Thereorem: there exists $u$ 
biholomorphic map from $D^2$ onto $\xi(D^2)\,.$ In particular $\partial_z u\ne 0$ in $D^2\,.$
 Take $\Phi:=\Psi\circ\xi^{-1}\circ u\,.$  
We observe that $\det(\Jac(\Psi))>0$ implies
$\partial_z \Psi\ne 0 $ in $\bar D^2$. 
Therefore  it holds 

 \begin{eqnarray*}
 \partial_z\Phi&=&\partial_{w}(\Psi\circ\xi^{-1})\partial_{z} u+
 \partial_{\bar w}(\Psi\circ\xi^{-1})\partial_{z} \bar {u}\\
 &=&\partial_{w}(\Psi\circ\xi^{-1})\partial_{z} u+
  \partial_{\bar w}(\Psi\circ\xi^{-1})\overline{\partial_{\bar z}  {u}}\\
  &=& \partial_{w}(\Psi\circ\xi^{-1})\partial_{z} u\,.
  \end{eqnarray*}
  We observe that $\Phi$ is  {holomorphic} in $D^2$  {because it is the composition of two holomorphic maps} and $\partial_z \Phi\ne 0$ in $  D^2$\,. From Lemma \ref{regbd} it follows that $\partial_z \Phi\ne 0$ in $ \bar D^2$ and we conclude the proof of Theorem \ref{GRMT}\,.~\hfill $\Box$


 \par\medskip
From the next Lemma we can deduce that   if $\gamma\in \mathcal{T}$ then the winding number (or equivalently the degree) of $\gamma$ is $1$\,.
This is a consequence of the following lemma.

\begin{Lemma}\label{la:degree}
Let $\Phi\in W^{2,p}(\bar D^2,\C)$, for some $1<p\le +\infty$  be a holomorphic function such that $\partial_z\Phi\ne 0$ in $\bar D^2$. Then
\begin{equation}\label{index}
\deg\Phi= \frac{1}{2\pi}\int_0^{2\pi}\frac{\langle i\partial_{\theta}\Phi,\partial^2_{\theta}\Phi\rangle}{|\partial_{\theta}\Phi|^2}d\theta=1+\frac{1}{2\pi i}\int_{S^1}\frac{f^{\prime}(z)}{f(z)} dz =1\,,
\end{equation}
  where $f(z)=\Phi^{\prime}(z)$\,. 
  \end{Lemma}
  {\bf Proof.} We recall that
  $$
  \Phi^{\prime}(z)=\frac{e^{-i\theta}}{2}\left(\frac{\partial\Phi}{\partial r}-\frac{i}{r}\frac{\partial\Phi}{\partial \theta}\right)=:f(z)\,.
  $$
  Since $\Phi$ is holomorphic we have 
  \begin{equation}\label{linkpartder}
  \frac{\partial\Phi}{\partial r}=-\frac{i}{r}\frac{\partial\Phi}{\partial \theta}\,.\end{equation}
  
  Hence
  \begin{eqnarray}\label{estzeri}
  \int_{S^1}
  \frac{f^{\prime}(z)}{f(z)} dz&=&
  \int_{S^1} \frac{\frac{e^{-i\theta}}{2}(\frac{\partial}{\partial r}-\frac{i}{r}\frac{\partial}{\partial \theta})\frac{e^{-i\theta}}{2}(\frac{\partial\Phi}{\partial r}-\frac{i}{r}\frac{\partial\Phi}{\partial \theta})}{\frac{e^{-i\theta}}{2}(\frac{\partial\Phi}{\partial r}-\frac{i}{r}\frac{\partial\Phi}{\partial \theta})}\,dz\\
  &\underbrace{=}_{\text{by }\rec{linkpartder}}&
    \int_{S^1}\frac{(\frac{\partial}{\partial r}-\frac{i}{r}\frac{\partial}{\partial \theta})(-\frac{i}{r}e^{-i\theta}\frac{\partial\Phi}{\partial\theta})}{ (\frac{\partial\Phi}{\partial r}-\frac{i}{r}\frac{\partial\Phi}{\partial \theta})}\,dz\nonumber
       \\&=&\int_{S^1}e^{-i\theta}\frac{\frac{2i}{r^2}\frac{\partial\Phi}{\partial \theta}-\frac{i}{r}\frac{\partial^2\Phi}{\partial r\partial\theta}-\frac{1}{r^2}\frac{\partial^2\Phi}{\partial^2\theta}}{\frac{-2i}{r}\frac{\partial\Phi}{\partial\theta}}\,dz\nonumber
\\
  &\underbrace{=}_{r=1 \text{ on } S^1}&       -\int_{S^1}e^{-i\theta}\ dz+\int_{S^1}e^{-i\theta}
       \frac{\frac{\partial^2\Phi}{\partial r\partial\theta}}{-2i\frac{\partial\Phi}{\partial\theta}}\,dz
       \int_{S^1}e^{-i\theta}\frac{\frac{\partial^2\Phi}{\partial^2\theta}}{-2i\frac{\partial\Phi}{\partial\theta}}\,dz\nonumber
\\
&=&
       -2\pi i-\frac{i}{2}\int_0^{2\pi}\frac{\frac{\partial^2\Phi}{\partial r\partial\theta}}{\frac{\partial\Phi}{\partial\theta}}\,d\theta-\frac{1}{2}\int_0^{2\pi}\
       \frac{\frac{\partial^2\Phi}{\partial \theta\partial\theta}}{\frac{\partial\Phi}{\partial\theta}}\,d\theta\nonumber
\\  &\underbrace{=}_{\text{by }\rec{linkpartder}}&        -2\pi i-\int_0^{2\pi}\
        \frac{\frac{\partial^2\Phi}{\partial \theta\partial\theta}}{\frac{\partial\Phi}{\partial\theta}}\,d\theta\,.\nonumber
\end{eqnarray}
 On the other hand
 we have
 \begin{eqnarray}\label{estindex1}
 \int_0^{2\pi}\frac{\langle i\partial_{\theta}\Phi,\partial^2_{\theta}\Phi\rangle}{|\partial_{\theta}\Phi|^2}d\theta&=&
 \frac{1}{2}\int_0^{2\pi}\frac{-i\overline{\partial_{\theta}\Phi}\partial^2_{\theta^2}\Phi}{
 \overline{\partial_{\theta}\Phi}{\partial_{\theta}\Phi}}\,d\theta+
  \frac{1}{2}\int_0^{2\pi}\frac{i{\partial_{\theta}\Phi}\overline{\partial^2_{\theta^2}\Phi}}{
 \overline{\partial_{\theta}\Phi}{\partial_{\theta}\Phi}}\,d\theta\,.
 \end{eqnarray}
 We observe that
 \begin{eqnarray}\label{estindex2}
 \frac{1}{2}\int_0^{2\pi}\frac{i{\partial_{\theta}\Phi}\overline{\partial^2_{\theta^2}\Phi}}{
 \overline{\partial_{\theta}\Phi}{\partial_{\theta}\Phi}}\,d\theta&=&
- \frac{i}{2}\int_0^{2\pi}\overline{\partial_{\theta}\Phi}\frac{\partial^2_{\theta^2}\Phi}{|\partial_{\theta}\Phi|^2}d\theta\\
&-&  \frac{i}{2}\int_0^{2\pi}{|\partial_{\theta}\Phi|^2}\partial_\theta\left({|\partial_{\theta}\Phi|^{-2}}\right)\,d\theta \nonumber
\\
&=&
- \frac{i}{2}\int_0^{2\pi}\frac{\partial^2_{\theta^2}\Phi}{\partial_{\theta}\Phi}d\theta\,.\nonumber
\end{eqnarray}
It follows that
 \begin{eqnarray}\label{estindex3}
 \int_0^{2\pi}\frac{\langle i\partial_{\theta}\Phi,\partial^2_{\theta}\Phi\rangle}{|\partial_{\theta}\Phi|^2}d\theta&=&-i\int_0^{2\pi}\frac{\partial^2_{\theta^2}\Phi}{\partial_{\theta}\Phi}d\theta\,.
 \end{eqnarray}
 By combining the estimates \rec{estzeri},\rec{estindex1},\rec{estindex2},\rec{estindex3} we get
  \begin{eqnarray*}
  \int_{S^1}
  \frac{1}{2\pi i}\frac{f^{\prime}(z)}{f(z)} dz&=&-1-  \frac{1}{2\pi i}\int_0^{2\pi}\frac{\partial^2_{\theta^2}\Phi}{\partial_{\theta}\Phi}d\theta\\
  &=&-1+\frac{1}{2\pi} \int_0^{2\pi}\frac{\langle i\partial_{\theta}\Phi,\partial^2_{\theta}\Phi\rangle}{|\partial_{\theta}\Phi|^2}d\theta\,.
  \end{eqnarray*}
  We conclude the proof.\hfill $\Box$
\begin{Remark}Lemma \ref{la:degree} can also be obtained as a corollary of Theorem \ref{propcostrlambdaintr}. Indeed $\deg\Phi|_{S^1}=\frac{1}{2\pi}\int_{S^1}\kappa |\Phi'|d\theta$, but since $(-\Delta)^\frac12 \lambda =\kappa e^\lambda-1$, integrating we get $\int_{S^1}\kappa e^{\lambda}d\theta =2\pi$. 
\end{Remark}

  \subsection{Connection with half-harmonic maps}

In this subsection we show an interesting connection between the solutions of \rec{modelequbis} and the half-harmonic maps into a given curve $\Gamma\,.$\par

Let $\tilde\phi=\Phi\in C^1(\bar D^2, \C)$ be the map given by Theorem \ref{GRMT} and set $\phi:=\Phi|_{S^1}$. Then $\Phi$ is conformal up to the boundary, i.e. $\frac{\partial \phi}{\partial\theta}\cdot \frac{\partial\tilde \phi}{\partial r}=0$ on $S^1$. Since $ \frac{\partial\tilde \phi}{\partial r}\Big|_{r=1}= \lapfr\phi$, we deduce 
 \begin{equation}\label{halfharmonic}
\lapfr \phi\perp T_{\phi}\Gamma\,,~\mbox{i.e.}~
\frac{\partial \phi}{\partial\theta}\cdot (-\Delta)^{1/2} \phi=0~~\mbox{on ${\cal{D}}^{\prime}(S^1)$\,.}\end{equation}
Equation (\ref{halfharmonic}) says that  $\phi$ is a $1/2$-harmonic map into $\Gamma$ (see \cite{DLR1}). 

We would like to recall a characterization of $1/2$-harmonic maps 
 of $S^1$ into submanifolds of $\R^n$, which has been already observed in \cite{DL}\,.

\begin{Theorem}[\cite{DLR}]\label{mindisk} 
 Let    $u\in H^\frac12 (S^1,{\cal{N}})$, where ${\cal{N}}$ is a $k$-dimensional smooth submanifold of $\R^m$ without boundary . Then $u$ is a weak  $1/2$-harmonic map i.e.
$\lapfr u \perp T_u\cal{N},$
 if and only if its harmonic extension $\tilde u \in W^{1,2}(D^2 , \R^m)$ is conformal,  {in which case}
\begin{equation}\label{freeb}
\partial_r\tilde u \perp T_u\cal{N}\quad \text{in }{\cal{D}}^{\prime}(S^1).
\end{equation}  
\end{Theorem}

\begin{proof}
Let  $u\in H^\frac12 (S^1,{\cal{N}})$
be a weak  $1/2$-harmonic map and 
 let $\tilde u\in W^{1,2}(D,\R^m)$ be the harmonic extension of $u$\,. Then it holds
 $$E(u):=\int_{S^1}|(-\Delta)^\frac14 u|^2 |dz| =\int_{D^2}|\nabla\tilde u|^2|dz|\,.$$
 {\bf Claim: } For every  $\tilde X\in C^{\infty}(\bar D^2, {\R^2})$ such that $  \tilde X(z)\cdot z =0$ for $z\in S^1$ it holds
 \begin{equation}\label{critictildeu}
 \left(\frac{d}{dt}\int_{D^2}|\nabla\tilde u(z+t\tilde X(z))|^2 |dz|\right)\bigg|_{t=0}=0\,.
 \end{equation}
 {\bf Proof of the Claim.}\par
 
 It has been proved in \cite{DLR1} that if $u$ is $1/2$-harmonic, then
$u\in C^{\infty}(S^1)$, in particular $u$ satisfies
\begin{equation}\label{critic}
\left(\frac{d}{dt}\int_{S^1}|(-\Delta)^\frac14 u(z+tX(z))|^2 |dz|\right)\bigg|_{t=0}=0\,.
\end{equation}
 for every  $  X\in C^{\infty}(S^1)\,.$\par
 Let $\tilde X\in C^{\infty}(\bar D^2, {\R^2})$ such that $\tilde X(z)\cdot z=0$  for $z\in S^1$.
 We observe that
 for all $z\in S^1$, $Y:=d\tilde u\cdot \tilde X=du\cdot \tilde X\in T_u\cal{N}$ and
 \begin{eqnarray*} 
  \left(\frac{d}{dt}\int_{D^2}|\nabla\tilde u(z+t\tilde X(z))|^2 |dz|\right)\bigg|_{t=0}&=&\int_{D^2}\nabla\tilde u \cdot\nabla Y |dz|\\
&=&\int_{S^1} \partial_r\tilde u\cdot Y|dz|\\
&=&-\int_{S^1}(-\Delta)^\frac12 u\cdot Y|dz|=0\,,
\end{eqnarray*}
 where the last equality follows from (\ref{critic})\,.\par
 From  Proposition \ref{propRiv} below and the regularity of $\tilde u$ up to the boundary it follows that $\tilde u$ is also conformal in $\bar D^2$  i.e.
$$ |\partial_{x_1} \tilde u|=|\partial_{x_2} \tilde u|,\quad 
  \partial_{x_1} \tilde u\cdot \partial_{x_2} \tilde u  =0$$
Conversely, suppose that the harmonic extension $\tilde u$  of $u$
is  conformal and satisfies \eqref{freeb}\,.
 Since $ \partial_r \tilde u=-(-\Delta)^\frac12 u$  we deduce that 
  $u$ is $1/2$-harmonic\,.
\end{proof}

\begin{Proposition}[Prop. II.2 in \cite{Riv}]\label{propRiv}
Let $\tilde u$ be a map in $W^{1,2}(D^2,\R^m)$  satisfying
$$
\left(\frac{d}{dt}\int_{D^2}|\nabla\tilde u_t|^2 |dz|\right)\bigg|_{{t=0}}=0,\,\quad u_t(x):=u(x+tX(x))$$
for every $X\in C^{\infty}(\bar D^2,\R^2)$ such that $\langle X(x),x\rangle=0$ for $x\in S^1$.
Then $\tilde u$ is conformal in $D^2$\,.
\end{Proposition}

 \medskip
 In the case of $1/2$-harmonic maps $u\colon S^1\to S^1$ we deduce from Theorem \ref{mindisk} the following 
\begin{Corollary}
Let  $u\in H^\frac12 (S^1, S^1)$ with $deg(u)=1$. Then $u$ is a weak  $1/2$-harmonic map if and only if its harmonic extension $\tilde u\colon \bar D^2 \to \bar D^2$
 is a M\"obius map, namely it has the form
  $$
  \tilde u(z)=e^{i\theta_0}\frac{z-a}{1-\bar a z},$$
  for some $|a|<1$ and $\theta_0\in [0,2\pi)\,.$
\end{Corollary}

\section{Compactness of the Liouville equation in $S^1$}\label{sec:comps1}
  
In this section we analyse the asymptotics of solutions to the equation \rec{modelequbis}\,.

\subsection{The $\varepsilon$-regularity lemma and first compactness result.}

A key point in the proof of Theorem \ref{mainth} is an {\em $\varepsilon$-regularity Lemma}, asserting roughly speaking that if the $L^1$ norm in conformal parametrization of the curvature ($\kappa_k e^{\lambda_k}$) is small (less than $\pi$) in a neighborhood of a point, then $\lambda_k-C_k$ is uniformly bounded in the same neighborhood, for some constant $C_k$. This result (Lemma \ref{eLemma}) depends on Theorem \ref{MT4} below. \par

\begin{Lemma}[Fundamental solution of $(-\Delta)^\frac12$ on $S^1$]\label{lemmafund4} The function
$$G(\theta):=-\frac{1}{2\pi}\log(2(1-\cos (\theta)))$$
belongs to $ BMO(S^1)$,
can be decomposed as
\begin{equation}\label{Gdec}
G(\theta)=\frac{1}{\pi}\log\frac{\pi}{|\theta|} +H(\theta),\quad \theta\in [-\pi,\pi]\sim S^1, \quad \text{with }H\in C^0(S^1),
\end{equation}
and satisfies
\begin{equation}\label{eqF}
(-\Delta)^\frac12 G=\delta_1-\frac{1}{2\pi} \quad \text{in }S^1,\quad \int_{S^1}G(\theta)d\theta=0,
\end{equation}
and for every function $u\in L^1(S^1)$ with $(-\Delta)^\frac12 u\in L^1(S^1)$ one has
\begin{equation}\label{rapr}
u-\bar u = G*(-\Delta)^\frac{1}{2}u:=\int_{S^1} G(\cdot -\theta)(-\Delta)^\frac{1}{2}u(\theta)d\theta, \quad \text{for almost every }t\in S^1.
\end{equation}
\end{Lemma}

\noindent {\bf Proof.}  Identity \eqref{eqF} follows at once from Lemma \ref{lemmaG}. That $G\in BMO(S^1)$ follows from parametrizing $S^1 =[-\pi,\pi]/\{\pi\sim-\pi\}$, writing $1-\cos(\theta)=\frac{\theta^2}{2}+O(\theta^4)$ as $\theta\to 0$ and therefore
$$G(\theta)=-\frac{1}{2\pi}(\log(\theta^2/2) +\log(1+O(\theta^2)))$$
as $\theta\to 0$. Similarly \eqref{Gdec} follows from the explicit expression of $G$, since
$$H(\theta):=G(\theta)-\frac{1}{\pi}\log\frac{\pi}{|\theta|}= C+\log(1+O(\theta)^2)\to C\quad \text{as }\theta\to 0,$$
and $H(\theta)\to -\frac{1}{2\pi}\log 2$ as $|\theta|\to \pi$, so that $H\in C^0(S^1)$.

To prove  \eqref{rapr} for $u\in C^\infty$ we write
$$u(0)-\bar u=\left\langle \delta_1-\frac{1}{2\pi},u\right\rangle =\langle (-\Delta)^\frac12 G, u\rangle:= \int_{S^1} G(\theta)(-\Delta)^\frac{1}{2}u(\theta)d\theta,$$
and translating one gets \eqref{rapr} also for $t\ne 0$. For a general function $u\in H^{1,1}_\Delta(S^1)$ take a sequence $(u_k)\subset C^\infty(S^1)$ with
$$u_k\to u,\quad (-\Delta)^\frac12 u_k \to (-\Delta)^\frac12 u\quad \text{in }L^1(S^1),$$
which can be easily obtained by convolution. 
Then
$$u \overset{L^1(S^1)}\longleftarrow u_k=\int_{S^1} G(\cdot -\theta)(-\Delta)^su_k(\theta)dy\overset{L^1(S^1)}\longrightarrow \int_{S^1} G(\cdot -\theta)(-\Delta)^s(\theta)d\theta,$$
the convergence on the right following from \eqref{Gdec}, and Fubini's theorem:
\[
\begin{split}
\int_{S^1} &\left|\int_{S^1} G(t-\theta)\left[(-\Delta)^su_k(\theta)- (-\Delta)^su(\theta)\right]d\theta \right|dt\\
&\le \|G\|_{L^1(S^1)} \|(-\Delta)^su_k- (-\Delta)^su\|_{L^1({S^1})}\to 0
\end{split}
\]
as $k\to \infty$. Since the convergence in $L^1$ implies the a.e. convergence (up to a subsequence), \eqref{rapr} follows.
The last claim follows at once from the explicit expression of $G$.~\hfill$\Box$
 \par
 \medskip
The following Theorem is   a generalization of Theorem I in \cite{BM} and it is crucial to prove Lemma \ref{eLemma}\,.\par
\begin{Theorem}\label{MT4} There exist constants $C_1,C_2>0$ such that for any $\ve\in (0,\pi)$ one has
\begin{equation}\label{stimaMT4}
C_1\le  \sup_{u=G*f\,:\, \|f\|_{L^1(S^1)}\le 1}\ve \int_{S^1}e^{(\pi-\ve)|u|}  d\theta \le C_2,
\end{equation}
and in particular
\begin{equation}\label{stimaMT4bis}
C_1\le  \sup_{\stackrel{u\in L^1 (S^1):\, \|(-\Delta)^{1/2} u-\alpha\|_{L^1(S^1)}\le 1 }{\text{for some }\alpha\in\R}}\ve \int_{S^1}e^{(\pi-\ve)|u-\bar u|}  d\theta \le C_2.
\end{equation}
\end{Theorem}

\noindent{\bf Proof of Theorems \ref{MT4}.}
Clearly the second inequality in \eqref{stimaMT4bis} follows from the second inequality in \eqref{stimaMT4} and \eqref{eqF}. Let us now prove \eqref{stimaMT4}. Given $f$ with $\|f\|_{L^1(S^1)}\le 1$
and setting $u=G*f$ we get
\[\begin{split}
|u(t)|&=\left|\frac{1}{\pi}\int_{t-\pi}^{t+\pi}\log\left(\frac{\pi}{|\theta-t|}\right) f(\theta)d\theta +\int_{t-\pi}^{t+\pi}H(\theta-t)f(\theta)d\theta\right|\\
&\le \frac{1}{\pi}\int_{t-\pi}^{t+\pi}\log\left(\frac{\pi}{|\theta-t|}\right)| f(\theta)|d\theta +C.
\end{split}
\]
With Jenses's inequality and Fubini's theorem, and using that $\|f\|_{L^1(S^1)}\le 1$, it follows
\begin{equation}\label{eqpfMT4}
\begin{split}
\int_{-\pi}^\pi e^{(\pi-\ve)|u(t)-\bar u|}dt&\le C\int_{-\pi}^\pi \exp \left(\frac{\pi-\ve}{\pi}\int_{t-\pi}^{t+\pi}\log\left(\frac{\pi}{|\theta-t|}\right)|f(\theta)|d\theta \right)dt\\
&\le C\int_{-\pi}^\pi \int_{t-\pi}^{t+\pi} \exp \left(\frac{\pi-\ve}{\pi}\log\left(\frac{\pi}{|\theta-t|}\right)\right) |f(\theta)|  d\theta dt\\
&= C \int_{t-\pi}^{t+\pi} |f(\theta)| \int_{-\pi}^\pi \left(\frac{\pi}{|\theta-t|}\right)^{1-\frac\ve\pi} dtd\theta\le \frac{C_2}{\ve}.
\end{split}
\end{equation}
This proves the second inequality in \eqref{stimaMT4}.

To prove the first inequalities in \eqref{stimaMT4} and in \eqref{stimaMT4bis} fix $\ve\in (0,\pi)$, choose $(f_k)\subset C^\infty(S^1)$ non-negative such that $f_k\to \delta_0$ weakly in the sense of measures,
 $\|f_k\|_{L^1(S^1)}=1$ and let $u_k$ solve
$$(-\Delta)^\frac12 u_k=f_k-\frac{1}{2\pi}\quad \text{in }S^1,\quad \bar u_k=0.$$
Such $u_k$ can be easily constructed using the Fourier formula for $(-\Delta)^\frac12$, see \eqref{fraclapl5}. Then by Lemma \ref{lemmafund4}
$$|u_k(t)|\ge\int_{S^1}G(t-\theta)f_k(\theta)d\theta \ge\frac{1}{\pi}\int_{t-\pi}^{t+\pi}\log\left(\frac{\pi}{|\theta-t|}\right) f_k(\theta)d\theta -C.$$
Multiplying by $\pi-\ve$, exponentiating, integrating on $S^1$ and taking the limit as $k\to\infty$ one gets
\[\begin{split}
\lim_{k\to\infty} \int_{S^1} e^{(\pi-\ve)|u_k(t)|}dt&\ge \lim_{k\to\infty}\frac{1}{C}\int_{-\pi}^\pi \exp\left( \frac{\pi-\ve}{\pi}\int_{t-\pi}^{t+\pi}\log \left(\frac{\pi}{|\theta-t|}\right)f_k(\theta)d\theta \right)dt\\
&=\frac{1}{C}\int_{-\pi}^\pi \exp\left( \frac{\pi-\ve}{\pi}\log \left(\frac{\pi}{|t|}\right)\right)dt\\
&=\frac{1}{C}\int_{-\pi}^{\pi} \left(\frac{\pi}{|t|}\right)^{1-\frac{\ve}{\pi}}dt= \frac{C_1}{\ve},
\end{split}\]
which proves \eqref{stimaMT4} and also \eqref{stimaMT4bis} since $\bar u_k=0$.
\hfill$\square$

   \begin{Lemma}[$\varepsilon$-regularity Lemma]\label{eLemma}
 Let $u\in L^1(S^1)$ be a solution of  
 \begin{equation}\label{liou6b}
 (-\Delta)^\frac{1}{2}u=\kappa e^{u}-1,
\end{equation}
with $\kappa\in L^\infty(S^1)$ and $e^{u}\in L^1(S^1)$ and $\Lambda:=\|\kappa e^u\|_{L^1}$.
Assume that for some arc $A\subset S^1$
\begin{equation}\label{condgammaexpb}
 \int_A |\kappa| e^u  d\theta  \le \pi-\ve\,, 
\end{equation}
for some  $\varepsilon>0$.
Then for every arc $A'\Subset A$ with $\dist (A^c, A')=\delta$
\begin{equation}\label{estinfty}
\|u-\bar u\|_{L^{\infty}(A')} \le C(\delta,\varepsilon,\Lambda).
\end{equation}
 
\end{Lemma}
 
\begin{proof} Set $f:=(-\Delta)^\frac{1}{2}u$.
We split  $f=f_1+f_2$ where
$$f_1=\kappa e^{u}\chi_A,\quad f_2=\kappa e^{u}\chi_{A^c}.$$
Let us now define 
$$
u_i(t):=G*f_i(t)=\int_{S^1}G(t-\theta)f_i(\theta)d\theta,\quad i=1,2,
$$
where $G$ is as in Lemma \ref{lemmafund4}. From  \eqref{eqF} and \eqref{rapr} it follows that 
$$u-\bar u=G*(\kappa e^u-1)=G*(\kappa e^u)=u_1 +u_2.$$
 
Choose now an arc $A''$ with $A'\Subset A''\Subset A$ and $\dist(A'', A^c)=\dist(A',(A'')^c)=\frac{\delta}{2}$. With \eqref{Gdec} we easily bound
\begin{equation}\label{estu2}
\|u_2\|_{L^\infty(A'')} \le C_1= C_1(\Lambda,\delta).
\end{equation}
It follows from \eqref{condgammaexpb} and Theorem \ref{MT4} that $\|e^{|u_1|}\|_{L^p(S^1)}\le C_{p,\ve}$ for some $p>1$, and consequently also $e^{\bar u}\le C$. Then for $t\in A'$ we have 
\[\begin{split}
u_1(t)&\le\int_{A}G(t-\theta)(|\kappa| e^{u_1(\theta)}e^{u_2(\theta)+\bar u}-1)d\theta\\
&\le \|\kappa\|_{L^\infty}\bigg(e^{C_1+\bar u}\underbrace{\int_{A''}G(t-\theta)e^{u_1(\theta)}d\theta}_{(1)}+\underbrace{\int_{A\setminus A''}G(t-\theta)e^{u(\theta)}d\theta}_{(2)}+C\bigg)\\
&\le C,
\end{split}\]
where in $(1)$ we use that   $G\in L^q(S^1)$ for $q\in [1,\infty)$ and in $(2)$ we use that $G\in L^{\infty}(A'\times (A\setminus A''))\,.$
\end{proof}

\begin{Lemma}\label{L2infty} Let $\lambda: S^1\to S^1$ satisfy $(-\Delta)^\frac12 \lambda \in L^1(S^1)$ and let $\tilde \lambda$ be the harmonic extension of $\lambda$ to $D^2$. Then
\begin{equation}\label{stimaLo1}
\|\nabla \tilde \lambda\|_{L^{(2,\infty)}(D^2)}\le C\|(-\Delta)^\frac12 \lambda\|_{L^1(S^1)},
\end{equation}
and for any ball $B_r(x_0)$
\begin{equation}\label{stimaLo2}
\frac{1}{r}\int_{B_r(x_0)\cap D^2}|\nabla \tilde \lambda|dx \le C\|\nabla \tilde \lambda\|_{L^{(2,\infty)}(B_r(x_0)\cap D^2)}.
\end{equation}
\end{Lemma}

\begin{proof}
 Let $\lambda: S^1\to S^1$ satisfy $(-\Delta)^\frac12 \lambda \in L^1(S^1)$ and let $\tilde \lambda$ be the harmonic extension of $\lambda$ to $D^2$. Then
 we can write
 \begin{equation}\label{lambdal2infty}
\tilde\lambda(x)=\int_{S^1}G(x,y)\frac{\partial\tilde\lambda}{\partial \nu}(y)dy=\int_{S^1}G(x,y)(-\Delta)^\frac12 \lambda(y)dy\end{equation}
where $G$ is the Green function associated to the Neumann problem.
It is know that $\nabla_x(G(x,y))\in L^{(2,\infty)}(S^1)$ (see e.g. \cite{CK}). Therefore $\nabla \tilde\lambda(x)\in L^{(2,\infty)}(D^2)\,$ as well and
 \rec{stimaLo1} holds\,.\par
  
The proof of \eqref{stimaLo2} follows from O'Neil's inequality \cite{Oneil}
$$\int_{A}|\nabla \tilde\lambda|dx \le \|\chi_A\|_{L^{(2,1)}(A)} \|\nabla \tilde \lambda\|_{L^{(2,\infty)}(A)}=\sqrt{|A|}\|\nabla \tilde \lambda\|_{L^{(2,\infty)}(A)}$$
for any $A\subset D^2$.
\end{proof}

\begin{Theorem}\label{convergencephi} 
Let $(\lambda_k)$ be a sequence as in Theorem \ref{mainth}, and let $(\Phi_k)\subset C^1(\bar {D}^2, \C)$ be holomorphic immersions with $\lambda_k(z)=\log\left|\Phi_k'(z)\right|$ for $z\in S^1$ and $\Phi_k(1)=0$ (compare to Theorem \ref{propcostrlambdaintr}) 
Then, up to extracting a subsequence, the following set is finite
\begin{equation}\label{defB}
B:=\left\{a\in S^1:\lim_{r\to 0^+}\limsup_{k\to\infty} \int_{B(a,r)\cap S^1} |\kappa_k| e^{\lambda_k}d\theta\ge \pi \right\}=\{a_1,\ldots,a_N\}\,,
\end{equation}
and for functions $v_\infty\in L^1(S^1,\R)$ and $\Phi_\infty\in W^{1,2}(D^2,\C)$ we have for $1\le p <\infty$
\begin{equation}\label{convlambdak}
\lambda_k-\bar\lambda_k \rightharpoonup v_\infty \quad\text{in }W^{1,p}_{\loc}(S^1\setminus B)\,,\quad \bar\lambda_k:=\frac{1}{2\pi}\int_{S^1}\lambda_k d\theta\,,
\end{equation}
and
\begin{equation}\label{convPhik}
\Phi_k \rightharpoonup \Phi_\infty \quad\text{in }W^{2,p}_{\loc}(\bar D^2\setminus B,\C)\text{ and in }W^{1,2}(D^2,\C).
\end{equation}
Moreover, one of the following alternatives holds:\par
1. The sequence $(\lambda_k)\subset\R$ is bounded and $\Phi_\infty$ is a holomorphic immersion of $\bar D^2\setminus B$ (i.e. it is holomorphic in $D^2$ and $\de_z\Phi_\infty\ne 0$ for $z\in \bar D^2\setminus B$).\par
2. $\lambda_k\to -\infty$ locally uniformly as $k\to +\infty,$ and $\Phi_\infty\equiv Q$ for some constant $Q\in \C$.
\end{Theorem}

\noindent {\bf Proof.}
The sequence of measures $|\kappa_k|e^{\lambda_k}d\theta$ on $S^1$ is bounded (for the total variation norm), hence up to extracting a subsequence we have $|\kappa_k|e^{\lambda_k}dx\stackrel{*}{\rightharpoonup} \mu$ weakly in the sense of measures for a Radon measure $\mu\in \mathcal{M}(S^1)$. Let $B:=\{a\in S^1:\mu(\{a\})\ge \pi\}$. Then $B$ is clearly finite, say $B=\{a_1,\dots, a_N\}$, and is characterised by the first identity in \eqref{defB}. Indeed if $\mu(\{a\})\ge \pi$, for every $r>0$ and $\varphi\in C^0(S^1)$ supported in $B(a,r)\cap S^1$ such that $0\le \varphi\le 1=\varphi(a)$ one has
$$\limsup_{k\to\infty}\int_{B(a,r)\cap S^1}|\kappa_k|e^{\lambda_k}d\theta\ge \limsup_{k\to\infty}\int_{S^1}|\kappa_k|e^{\lambda_k}\varphi d\theta =\int_{S^1}\varphi d\mu\ge \pi \varphi(a)=\pi,$$
and conversely if $\mu(\{a\})<\pi$, then $\mu(B(a,r_0)\cap S^1)<\pi$ for some $r_0>0$, hence  taking $\varphi\in C^0(S^1)$ supported in $B(a,r_0)\cap S^1$, with $0\le \varphi\le 1$ and $\varphi\equiv 1$ on $B(a,r_0/2)\cap S^1$, one gets
$$\limsup_{k\to\infty}\int_{B(a,r_0/2)\cap S^1}|\kappa_k|e^{\lambda_k}d\theta\le \limsup_{k\to\infty}\int_{S^1}|\kappa_k|e^{\lambda_k}\varphi d\theta =\int_{S^1}\varphi d\mu\le \mu(B(a,r_0))<\pi.$$

We now show that for every compact $K\subset S^1\setminus B$ there exits a constant $c_K$ depending on $\bar L$ and $\bar \kappa$ in \eqref{boundlength}-\eqref{boundcurv} 
such that
\begin{equation}\label{conda}
\|e^{\lambda_k}\|_{L^{\infty}(K)}\le c_K\,.
\end{equation}
and
\begin{equation}\label{condb}
\|{\lambda_k-\bar \lambda_k}\|_{L^{\infty}(K)}\le c_K.
\end{equation}
Indeed cover $K$ with finitely many arcs $A_i\cap S_1$ so that
$$\int_{A_i\cap S^1}|\kappa_k|e^{\lambda_k}d\theta< \pi .$$
From Lemma \ref{eLemma} it follows that $\lambda_k-\bar \lambda_k$ is bounded in each $A_i$, and \eqref{condb} follows. Moreover, considering that $\|e^{\lambda_k}\|_{L^1(S^1)}=L_k \le \bar L$, it follows that $\bar \lambda_k$ and $\lambda_k$ are upper bounded, and this proves \eqref{conda}. Now writing $\lambda_k-\bar \lambda_k= G*(\kappa_k e^{\lambda_k}-1)$ as in \eqref{rapr} of Lemma \ref{lemmafund4} we can bootstrap regularity and obtain that $\lambda_k-\bar \lambda_k$ is bounded in $W^{1,p}(K)$ for every $p<\infty$, and \eqref{convlambdak} follows from weak compactness.

Let $\tilde\lambda_k$ be the harmonic extension of $\lambda_k$.
From  \eqref{condb}, 
 \eqref{stimaLo1} and \eqref{stimaLo2} we get
\begin{equation*}
\|\tilde\lambda_k -\bar\lambda_k\|_{L^\infty(\de (D^2\setminus \cup_{i=1}^N B(a_i,\delta)))}\le C_\delta\quad \text{for every }\delta>0,
\end{equation*}
hence
\begin{equation}\label{lambda1}
(\tilde \lambda_k-\bar \lambda_k)\quad \text{ is bounded in }W^{1,p}_{\loc}(\bar D^2\setminus B).
\end{equation}
Since  $\Phi_k$   is harmonic and conformal,  the following estimate holds
   \begin{equation}\label{estDirichEnergy}
   \int_{D^2}|\nabla \Phi_{k}(z)|^2\le \frac{1}{2}L_k^2\,.
   \end{equation}
Since $\Phi_k(1)=0$ it follows that the sequence $(\Phi_{k})$ is bounded in $W^{1,2}(D^2)$ and, up to a subsequence, $\Phi_{k} \rightharpoonup \Phi_\infty$  weakly in $W^{1,2}(D^2)$, where $\Phi_\infty$ is holomorphic.
%

From \eqref{lambdaPhi} it follows that $|\nabla \Phi_k|$ is bounded in $W^{1,p}_{\loc}(S^1\setminus B)$, hence $\Phi_k$ is bounded in $W^{2,p}_{\loc}(S^1\setminus B)$
and up to a subsequence one gets $\Phi_k\rightharpoonup \Phi_\infty$ in $W^{2,p}_{\loc}(D^1\setminus B)$, as wished.

Further, if $\bar\lambda_k\to -\infty$, then \eqref{lambda1} yields $\nabla\Phi_k\to 0$ uniformly locally in $\bar D^2\setminus B$, hence $\Phi_\infty$ is constant. Similarly, if $\lambda_k \ge -C$, then $|\nabla\Phi_k|$ is locally uniformly lower bounded on $D^2\setminus B$, hence $\nabla\Phi_\infty\ne 0$ in $D^2\setminus B$.
\hfill$\Box$

 \subsection{Blow-up  Analysis}
In this section we associate to a sequence $(\lambda_k)$ satisfying \eqref{boundlength}-\eqref{liouvfrack}-\eqref{boundcurv} a sequence of curves $(\gamma_k)\subset W^{2,\infty}(S^1,\C)$ with bounded lengths $L_k\le\bar L$, curvatures bounded by $\bar \kappa$, $|\dot \gamma_k|\equiv \frac{L_k}{2\pi}$, a sequence $(\Phi_k)\subset C^1(\bar D^2,\C)$ of holomorphic immersions so that $|(\Phi_k')|_{S^1}|=e^{\lambda_k}$ and a sequence of diffeomorphisms $\sigma_k:S^1\to S^1$ such that $\Phi_k\circ \sigma_k=\gamma_k$. Up to a translation we can assume that $\Phi_k(1)=0$, and by Arzel\`a-Ascoli's theorem $\gamma_k\to \gamma_\infty$ in $C^1(S^1,\C)$ for a curve $\gamma_\infty\in W^{2,\infty}(S^1,\C)$.

Notice that $(\Phi_k)$ and $(\lambda_k)$ satisfy the hypothesis of Theorem \ref{convergencephi}, and up to a subsequence we can assume that \eqref{convlambdak} and \eqref{convPhik} hold for a finite set $B=\{a_1,\dots, a_N\}$ and functions $v_\infty\in L^1(S^1,\R)$ and $\Phi_\infty\in W^{1,2}(D^2,\C)$. Moreover, either 1. or 2. in Theorem \ref{convergencephi} holds.

We introduce the following distance function $D_k\colon S^1\times S^1 \to \R^+$.
 \begin{eqnarray}\label{dk}
 D_{k}(q,q')&=&\inf\bigg\{\left(\int_0^1 |\Phi'_k(\Delta_k(t))|^2|\Delta_k'(t)|^2dt\right)^\frac12,\nonumber\\
 &&~~~ \Delta_k\in W^{1,2}([0,1],\bar D^2), ~\Delta_k(0)=\sigma_k(q), ~\Delta_k(1)=\sigma_k(q')\bigg\},
 \end{eqnarray}
  It is well-known that the infimum in (\ref{dk}) is attained by a path $\Delta_k$ such that $|\Phi'_k(\Delta_k(t))||\Delta_k'(t)|=const\,.$ For such path we then have
$$\left(\int_0^1 |\Phi'_k(\Delta_k(t))|^2|\Delta'(t)|^2dt\right)^\frac12=\int_0^1 |\Phi'_k(\Delta_k(t))| |\Delta_k'(t)| dt=:\int_{\Delta_k} |\Phi'_k(z)||dz|\,.$$
In the sequel we  sometimes identify the parametrization of a curve $\Delta$ with its image\,.\par
 \begin{Proposition}\label{PropDK}
 1) The function $D_k$ is Lipschitz continuous with
 $\|\nabla D_k\|_{L^{\infty}}\le 1$ 
and it converges uniformly.

 2) The infimum in \rec{dk} is attained by a curve $\Delta_k$ in normal parametrization such that the curvature of 
 $\Phi_k\circ\Delta_k$ is bounded by $\|\kappa_k\|_{L^\infty}$.
 \end{Proposition}
 {\bf Proof.} 
 1. Let $q,q',\tilde q,\tilde q'\in S^1$. The following estimate holds \begin{eqnarray*}
  D_{k}(q,q')&\le&  D_{k}(\tilde q,\tilde q')+\mbox{arc}{(\gamma_k(q),\gamma_k(\tilde q))}+\mbox{arc}{(\gamma_k(q'),\gamma_k(\tilde q'))}\\
  &\le&  D_{k}(\tilde q,\tilde q')+{|q-\tilde q|+ |q'-\tilde q'|}.
 \end{eqnarray*}
By exchanging $(q,q')$ and $(\tilde q,\tilde q')$ we get
  that
$$ |D_k(q,q')-   D_{k}(\tilde q,\tilde q')|\le |q-\tilde q|+| q'-\tilde q'|\,,$$
and we conclude\,.\par
2. For a geodesic $\Delta$ with respect to $D_k$, the curve $\Phi_k\circ \Delta$ is a geodesic in $\mathbb{C}$ under the constraint that $\Phi_k\circ\Delta\subset \Phi_k(\bar D^2)$. This must be a union of segments (contained in $\Phi_k(D^2)$) and arcs of the curve $\gamma_k$, where the segments touch the curve $\gamma_k$ tangentially. Hence the curvature of $\Phi_k\circ \Delta$ is bounded by $\|\kappa_k\|_{L^\infty}$.  \par

 This completes the proof of Proposition \ref{PropDK}\,.
~\hfill$\Box$
  \par\medskip
  
  We give next the definition of a {\em pinched point} for the curve $\gamma_{\infty}$\,.\par

\begin{Definition}[Pinched point]\label{pinched}
A point $p\in S^1$ is called {\em pinched point} for the sequence $(\gamma_k)$ if there exists $p^{\prime}\in S^1$, $p\neq p^{\prime}$  such that 
$\lim_{k\to+\infty} D_k(p,p')=0$\,. We call $p'$ the  {``dual"} of $p$. We denote by  $\mathcal{P}$ the sets of the pinched points of $\gamma_{\infty}\,.$
    \end{Definition}
    \begin{Remark}
    The definition of pinched point is independent of $\Phi_k$ and $\sigma_k$ in the sense that if  $\tilde\Phi_k=\Phi_k\circ f_k$ where $f_k\colon\bar D^2\to \bar D^2$ is a M\"obius transformation
    and if $\tilde\sigma_k=f_k^{-1}\circ \sigma_k   $,    then 
    $$
    \lim_{k\to +\infty}\int_0^1 |\Phi'_k(\Delta(t))||\Delta'(t)|dt=0, ~~\mbox{if and only if} ~\lim_{k\to +\infty} \int_0^1 |\tilde\Phi'_k(\tilde\Delta(t))| |{\tilde\Delta}'(t)|dt=0\,.$$
    \end{Remark}

 \begin{Proposition}\label{limconst} 
Assume 
that we are in case 2 of Theorem \ref{convergencephi}, i.e.
 $\Phi_k\to Q$ in $C^{1,\alpha}_{\loc}( \bar D^2\setminus\{a_1,\ldots,a_N\})$ for a constant $Q\in\C$. Then $N \in \{1,2\}$. If $N=2$, let $\mathcal{C}_+$ and $\mathcal{C}_-$ be the connected components of $S^1\setminus\{a_1,a_2\}$. Then $\sigma_k^{-1}\to p^{\pm}$ locally uniformly on $\mathcal{C}_\pm$, where $p^+,p^-\in \mathcal{P}$ are dual. Moreover $Q=\gamma_{\infty}(p^+)=\gamma_\infty(p^-)$ and $\dot\gamma_\infty(p^+)=-\dot\gamma_\infty(p^-)$, $\kappa_k e^{\lambda_k}\weak \pi(\delta_{a_1}+\delta_{a_2})$ and $v_k:=\lambda_k-\bar\lambda_k \rightharpoonup  v_\infty$ in $W^{1,p}_{\loc}(S^1\setminus \{a_1,a_2\})$, where $v_\infty$ solves \eqref{liouvlimit3}. If $N=1$ then $v_k\to v_\infty$ where $v_\infty$ solves \eqref{liouvlimit2}.
 \end{Proposition}

\noindent {\bf Proof.} By Theorem \ref{convergencephi} we have $\bar \lambda_k\to -\infty$ and $\lambda_k\to -\infty$ uniformly locally in $S^1\setminus B=\{a_1,\dots, a_N\}$. In particular, since the signed radon measures $\kappa_k e^{\lambda_k}dx$ are uniformly bounded, we have $\mu_k\weak \mu$ for a Radon measure supported in $B$,  which then we can write as
$\mu =\sum_{i=1}^N \alpha_i\delta_{a_i}.$ Moreover, since
$$\int_{S_1}\kappa_k e^{\lambda_k}\,d\theta=2\pi\,,$$
we infer that $\sum_{i=1}^N \alpha_i =2\pi$.

Let us assume that $N\ge 2$. We want to prove that $\alpha_i=\pi$ for every $i$, hence necessarily $N=2$. In order to prove that $\alpha_i=\pi$, up to a rotation we can reduce to proving that $\alpha_1=\pi$ and assume that $a_1=i$. We can also assume that $N=2$ and $a_2=-i$. If this is not the case, it suffices to compose $\Phi_k$ with M\"obius diffeomorphisms $f_k(z)=\frac{z-it_k}{1+it_k z}$ with $t_k\uparrow 1$ slowly enough so that $\tilde\Phi_k:=\Phi_k\circ f_k$ is still as in case 2 of Theorem \ref{convergencephi}, with $B=\{a_1=i, a_2=-i\}$.

Then let $\Phi_k$ be as above, with $\Phi_k\rightharpoonup Q$ in $W^{2,p}_{\loc} (\bar D^2\setminus \{i,-i\})$.
Set
$$V_k(z)=e^{-\bar \lambda_k}(\Phi_k(z)-\Phi_k(0))\,,\quad v_k=\log |V_k'|_{S^1}|=\lambda_k-\bar\lambda_k\,.$$
By Theorem \ref{convergencephi} we have
$$v_k \rightharpoonup v_\infty\quad \text{in }W^{1,p}_{\loc}(S^1\setminus \{i,-i\})\text{ and in }\mathcal{D}'(S^1),$$
where $v_\infty$ solves
\begin{equation}\label{eqhatlambda}
(-\Delta)^\frac12 v_\infty=\alpha\delta_i+(2\pi-\alpha)\delta_{-i}-1,
\end{equation}
for some $\alpha\in \R$.
Similarly $V_k\rightharpoonup V_\infty$ in $W^{2,p}_{\loc}(\bar D^2\setminus\{i,-i\})$. Solutions to \eqref{eqhatlambda} can be computed explicitly using Lemma \ref{lemmafund4}, so that
$$v_\infty(e^{i\theta})= -\frac{\alpha}{2\pi}\log(2(1-\sin\theta)) -\frac{2\pi -\alpha}{2\pi}\log(2(1+\sin\theta)).$$
Notice that writing $z=x+iy$, for $z=e^{i\theta}\in S^1$ we have
$$2(1-\sin\theta)= x^2+y^2-2y+1=|z-i|^2,$$
and similarly  $2(1+\sin\theta)=|z+i|^2$. In particular the $v_\infty$ can be extended to a holomorphic function
\begin{equation}\label{lambdatilde}
\tilde v_\infty(z):=-\frac{\alpha}{2\pi}\log(|z-i|^2) -\frac{2\pi -\alpha}{2\pi}\log(|z+i|^2),\quad z\in \bar D^2\setminus\{i,-i\}.
\end{equation}
The estimate \rec{lambda1} together with \eqref{lambdaPhi} implies that
$$c_{\delta}^{-1}\le|V'_k|\le c_{\delta}\quad \text{on }\bar D^2\setminus(B(i,\delta)\cup B(-i,\delta))\,\quad \text{for every }\delta>0.$$
Therefore
$V_k\rightharpoonup V_{\infty}$ as  $k\to +\infty$ in $W_{\loc}^{2,p}( \bar D^2\setminus\{i,-i\})$, where $V_{\infty}$ is a conformal immersion of $\bar D^2\setminus\{i,-i\}$\,. Moreover, still using \eqref{lambdaPhi}, from \eqref{lambdatilde} we obtain
$$|V_\infty'(z)|=\frac{1}{|z-i|^\frac{\alpha}{\pi}|z+i|^{2-\frac{\alpha}{\pi}}}.$$
Since $V_\infty'$ is holomorphic in $D^2$, up to a rotation (i.e. multiplication by a constant $e^{i\theta_0}$) we obtain
$$V_\infty'(z)=\frac{1}{(z-i)^\frac{\alpha}{\pi}(z+i)^{2-\frac{\alpha}{\pi}}},\quad V_\infty(z)=\int_0^z\frac{dz}{(z-i)^\frac{\alpha}{\pi}(z+i)^{2-\frac{\alpha}{\pi}}}.$$
Up to possibly switching $i$ with $-i$ we may assume that $\alpha\le \pi$.
The function $V_\infty$ is also known as Schwarz-Christoffel mapping\footnote{up to composition with a conformal transformation, since Schwarz-Christoffel maps are usually defined on the half plane $\{z\in \mathbb{C}:\Re z>0\}$ instead of the unit disk.}
and sends the two arcs of $\mathcal{C}_+,\mathcal{C}_-\subset S^1$ joining $i$ and $-i$ (chosen so that $\pm1 \in \mathcal{C}_{\pm}$) into two parallel straight lines if $\alpha=\pi$ and into  two half-lines meeting at $V_\infty(i)$, forming there an angle of $\pi-\alpha$ if $\alpha<\pi$.

\par

\noindent{\bf Claim 1.} As $k\to +\infty$ we have $\sigma_k^{-1}\to p^\pm$ in $L^\infty_{\loc}(\mathcal{C}_\pm)$, where $p^+,p^-\in S^1$,  with $p^+\ne p^-\,.$\par

\noindent{\bf Proof of Claim 1.}
Notice that $\Phi_k\rightharpoonup Q$ in $W^{2,p}_{\loc}( \bar D^2\setminus\{i,-i\})$ implies that
$$\frac{\partial\sigma_k^{-1}}{\partial{\theta}} \to 0\quad \text{uniformly locally in }S^1\setminus\{i,-i\} \text{ as }k\to +\infty\,.$$
This proves the first part of the claim.
Assume by contradiction that $p^+=p^-$.
Set $p_k^\pm=\sigma_k^{-1}(\pm 1)\to p^\pm$. By assumption $|\mathrm{arc}(p^+_k, p^-_k)|\to 0$ (here $\mathrm{arc}(p_k^+,p_k^-)$ denotes the shortest arc connecting $p^+_k$ to $p^-_k$).
Since $\sigma_k$ is a diffeomorphism, for small $\delta>0$, $\sigma_k(\mathrm{arc}(p_k^+,p_k^-))$ contains either $S^1\cap B(i,\delta)$ or $S^1\cap B(-i,\delta)$\,. Suppose it contains
 $S^1\cap B(i,\delta)$\,.
Then
\begin{equation}\label{estlambdak}
\begin{split}
\int_{S^1\cap B(i,\delta)}e^{\lambda_k} d\theta&= \int_{S^1\cap B(i,\delta)}|\Phi_k^{\prime}(e^{i\theta})| d\theta\\
&\le \int_{\mathrm{arc}(p_k^+,p_k^-)}| \dot \gamma_k| d\theta\\
&= \frac{L_k}{2\pi}|\mathrm{arc}(p^+_k, p^-_k)|\to 0\,,
\end{split}
\end{equation}
as $k\to \infty$. This contradicts that $i\in B$, and concludes the proof of the claim 1.\hfill$\square$\par
 
 \noindent{\bf Claim 2.} $p^+$ is  a pinched point and $p^-$ is dual to it.

 \noindent{\bf Proof of Claim 2.} Let $p_k^\pm=\sigma_k^{-1}(\pm 1)$ be as above. 
Consider the path
$$\Delta_k=\mathrm{arc}(\sigma_k(p^+),1)\cup \mathrm{arc}(\sigma_k(p^-),-1)\cup [-1,1],$$ where $[-1,1]$ is the segment in $\bar D^2$ joining $-1$ to $1$.
Since as $k\to \infty$ we have
\begin{equation}\label{arco0}
\int_{\mathrm{arc}(\sigma_k(p^\pm),\pm 1)} |\Phi_k'(e^{i\theta})|d\theta =\int_{\mathrm{arc}(p_k^\pm,p^\pm)}|\dot\gamma_k|d\theta=\frac{L_k |\mathrm{arc}(p_k^\pm,p^\pm)|}{2\pi}\to 0
\end{equation}
and
$$\int_{[-1,1]} |\Phi_k'||dz| \le 2 \sup_{[-1,1]} |\Phi_k'||dz|\to 0,$$
we immediately infer that
$$\int_{\Delta_k}|\Phi_k'||dz|\to 0,$$
hence $p^+$ is dual to $p^-$.
This proves claim 2. \hfill$\square$

Now
\begin{equation}\label{gammalim}
\frac{2\pi}{L_k}\dot \gamma_k (p^\pm_k)=\frac{\frac{\de\Phi_k(\pm1)}{\de\theta}}{|\frac{\de\Phi_k(\pm1)}{\de\theta}|}= \frac{\frac{\de\Phi_k(\pm 1)}{\de\theta}}{e^{\bar\lambda_k}e^{\lambda_k(\pm1)-\bar\lambda_k}}=\frac{\de V_\infty(\pm 1)}{\de\theta} e^{\bar\lambda_k -\lambda_k(\pm1)}+o(1)\quad \text{as }k\to\infty.
\end{equation}
In particular, denoting by $(v,w)^\wedge$ the angle between two vectors, we have
\begin{equation}\label{anglealpha}
(\dot\gamma_k(p^+_k),\dot\gamma_k(p^-_k))^\wedge\to \left(\frac{\de V_\infty(1)}{\de\theta},\frac{\de V_\infty(-1)}{\de\theta}\right)^\wedge=\alpha.
\end{equation}

We consider now different cases.

\noindent{\bf Case 1: $0<\alpha<\pi$}. Since $p_k^\pm\to p^\pm$ and $p^+$ is pinched to $p^-$, and since
\[\begin{split}
|\gamma_k(p_k^+)-\gamma_k(p_k^-)|&\le D_k(p_k^+,p_k^-)\\
&\le D_k(p^+,p^-)+\frac{L_k}{2\pi}(|\mathrm{arc}(p^+,p_k^+)|+|\mathrm{arc}(p^-,p_k^-)|)\\
&\to 0\quad \text{as }k\to\infty
\end{split}\]
and taking \eqref{anglealpha} and the bound $\bar \kappa$ on the curvature of $\gamma_k$ into account, we see that for positive numbers $\delta_k^{\pm}\to 0$ as $k\to\infty$ we have
\begin{equation}\label{alpha1}
\gamma_k(p_k^+e^{i\delta_k^+})=\gamma_k(p_k^- e^{-i \delta_k^-}),
\end{equation}
i.e. the two curves $t\mapsto \gamma_{k}(p_k^\pm e^{\pm i t})$ cross in short time (see Figure \ref{fig2}).
Because $\delta_k^\pm\to 0$ we have
\begin{equation}\label{alpha2}
D_k(p_k^+e^{i\delta_k^+}, p_k^- e^{-i \delta_k^-})\le D_k(p_k^+,p_k^-)+\frac{L_k(\delta_k^++\delta_k^-)}{2\pi}\to 0,\quad \text{as }k\to\infty.
\end{equation}
Let now $\Delta_k:[0,1]\to \bar D^2$ be a geodesic realising the distance on the left-hand side of \eqref{alpha2}. Then \eqref{alpha1} implies that $\Phi_k\circ \Delta_k$ is a closed curve (non-constant, since $p_k^+e^{i\delta_k^+}\ne p_k^-e^{-i\delta_k^-}$ for $k$ large) so that the integral of its curvature is at least $\pi$ (see Lemma \ref{closed} below). On the other hand  Proposition \ref{PropDK} implies that the curvature of $\Phi_k\circ\Delta_k$ is bounded by $\bar\kappa$, and since the length of such geodesic is going to 0 according to \eqref{alpha2}, we get a contradiction.

\begin{figure}\label{f:curve}
\begin{center}
\begin{small}
\setlength{\unitlength}{1cm}
\begin{picture}(11.5,9.3)%
    \put(0,0){\includegraphics[width=12cm]{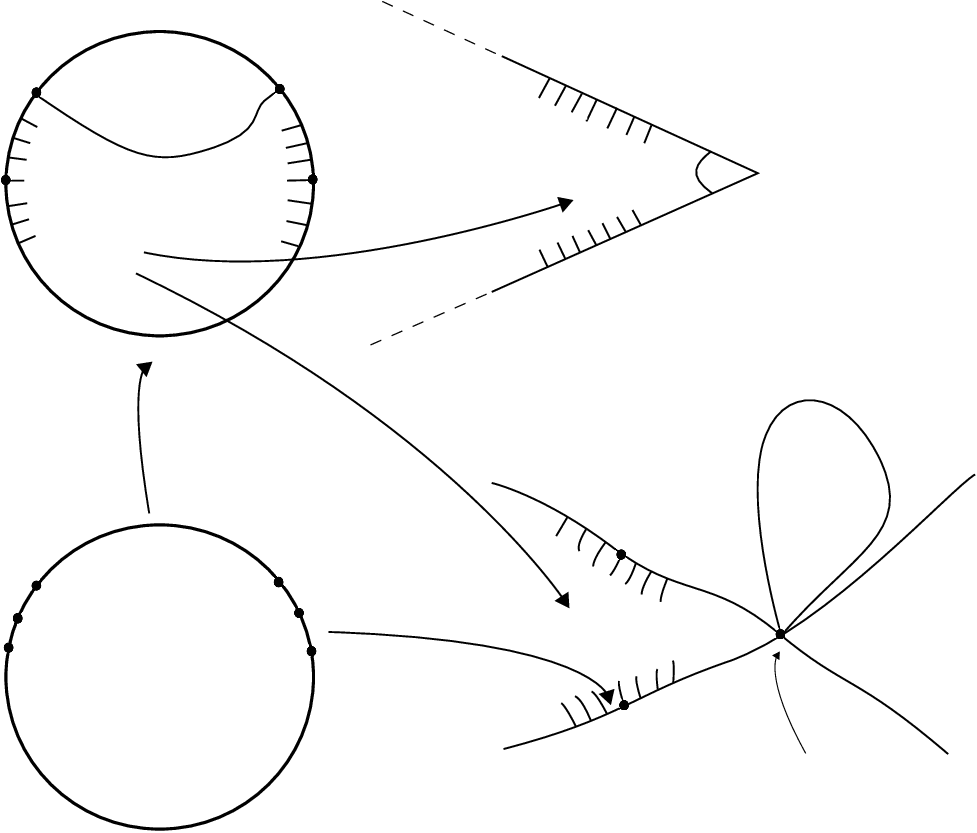}}%
    \put(1,5){$\sigma_k$}
    \put(-0.3,2.7){$p_k^-$}
    \put(-0.4,2.2){$p^-$}
    \put(-0.7,3.2){$p_k^-e^{-i\delta_k^-}$}
    \put(4,2){$p^+$}
    \put(3.8,2.7){$p_k^+$}
    \put(3.6,3.2){$p_k^+e^{i\delta_k^+}$}
    \put(6,2){$\gamma_k$}
    \put(7.8,1.3){$\gamma_k(p_k^+)$}
    \put(7.7,3.6){$\gamma_k(p_k^-)$}
    \put(4.5,5.3){$\Phi_k$}
    \put(2,8.5){$\Delta_k$}
    \put(7.5,8){$\pi-\alpha$}
    \put(4.5,6.8){$V_\infty$}
    \put(10.4,5.3){$\Phi_k(\Delta_k)$}
    \put(9,0.5){$\gamma_k(p_k^-e^{-i\delta_k^-})=\gamma_k(p_k^+e^{i\delta_k^+})$}
    \put(-1.6,9.3){$\sigma_k(p_k^- e^{-i\delta_k^-})$}
    \put(3.3,9.4){$\sigma_k(p_k^+ e^{i\delta_k^+})$}
    \put(-2.2,7.9){$\sigma_k(p_k^-)=-1$}
    \put(4,7.9){$1=\sigma_k(p_k^+)$}
  \end{picture}
\end{small}
\end{center}
\caption{Case 1}\label{fig2}
\end{figure}\par

\noindent{\bf Case 2: $\alpha =0$.} Similar to case 1, if the curves $\gamma_k(p_k^\pm e^{\pm it})$ cross for small times $\delta_k^\pm\to 0$, we conclude as before. If not, we can at least say that up to a rotation of the axis
\begin{equation}\label{Vinf}
V_\infty(D^2)=\{x+iy:y < 0\}
\end{equation}
and that for small times $\delta_k^{\pm}\to 0$ we have
\begin{equation}\label{alpha3}
\Re (\gamma_k(p_k^+e^{i\delta_k^+}))=\Re (\gamma_k(p_k^- e^{-i \delta_k^-}))
\end{equation}
and without loss of generality
\begin{equation}\label{alpha4}
\Im (\gamma_k(p_k^+e^{i\delta_k^+}))>\Im (\gamma_k(p_k^- e^{-i \delta_k^-})),
\end{equation}
where for $x,y\in \R$ we used the notation $\Re(x+iy)=x$, $\Im(x+iy)=y$ (see Figure \ref{fig3}).
Moreover since the curvature of $\gamma_k$ is uniformly bounded and $\delta_k^\pm\to 0$, using \eqref{gammalim} and \eqref{Vinf} we infer
\begin{equation}\label{alpha5} 
\frac{\dot \gamma_k(p_k^\pm e^{\pm i \delta_k^\pm})}{|\dot \gamma_k(p_k^\pm e^{\pm i \delta_k^\pm})|}=\frac{\dot \gamma_k(p_k^\pm )}{|\dot \gamma_k(p_k^\pm)|}+o(1)=- 1 +o(1),\footnote{the symbol $\dot \gamma_k(p_k^\pm e^{\pm i \delta_k^\pm})$ denotes the derivative of the curve $t\mapsto \gamma_k({e^{it}})$ evaluated for $e^{it}=p_k^{\pm}e^{\pm i\delta_k^{\pm}}$, and \emph{not} the derivative of the curve $t\mapsto \gamma_k(p_k^\pm e^{\pm i t})$ evaluated for $t=\delta_k^\pm$.}
\end{equation}
i.e. the curves $t\mapsto \gamma_k(p_k^\pm e^{\pm i t})$ at the time $t=\delta_k^\pm$ are almost horizontal and pointing into opposite directions (notice that change of orientation between the curves $t\mapsto \gamma_k(e^{it})$ and $t\mapsto \gamma_k(p_k^- e^{-it})$).
As before \eqref{alpha2} holds, so let $\Delta_k:[0,1]\to \bar D^2$ be geodesic realising the distance in \eqref{alpha2}, with $\Delta_k(0)=\gamma_k(p_k^+e^{i\delta_k^+})$ and $\Delta_k(1)=\gamma_k(p_k^- e^{-i\delta_k^-})$. Up to a reparametrization we can assume that $\tilde \Delta_k:=\Phi_k\circ \Delta_k:[0,L]\to \C$ satisfies $|\dot{\tilde \Delta}_k(t)|\equiv 1$. Since the map $\Phi_k$ preserves the orientation, from \eqref{alpha5} we infer
$$\Im(\dot {\tilde\Delta}_k(0))\le 0+o(1),\quad \Im(\dot {\tilde\Delta}_k(1))\ge 0+o(1),$$
i.e. up to a $o(1)\to 0$ as $k\to \infty$ we have that $\dot {\tilde\Delta}_k(0)$ points downwards, while $\dot {\tilde\Delta}_k(1)$ points upwards. Now using \eqref{alpha3} we see that the curve $\tilde \Delta_k$ has total curvature at least $\frac{\pi}{2}-o(1)$ (see Lemma \ref{closed2} below) again contradicting Proposition \ref{PropDK} and \eqref{alpha2}.

\begin{figure}\label{f:curve}
\begin{center}
\begin{small}
\setlength{\unitlength}{1cm}
\begin{picture}(9.5,10)%
    \put(0,0){\includegraphics[width=12cm]{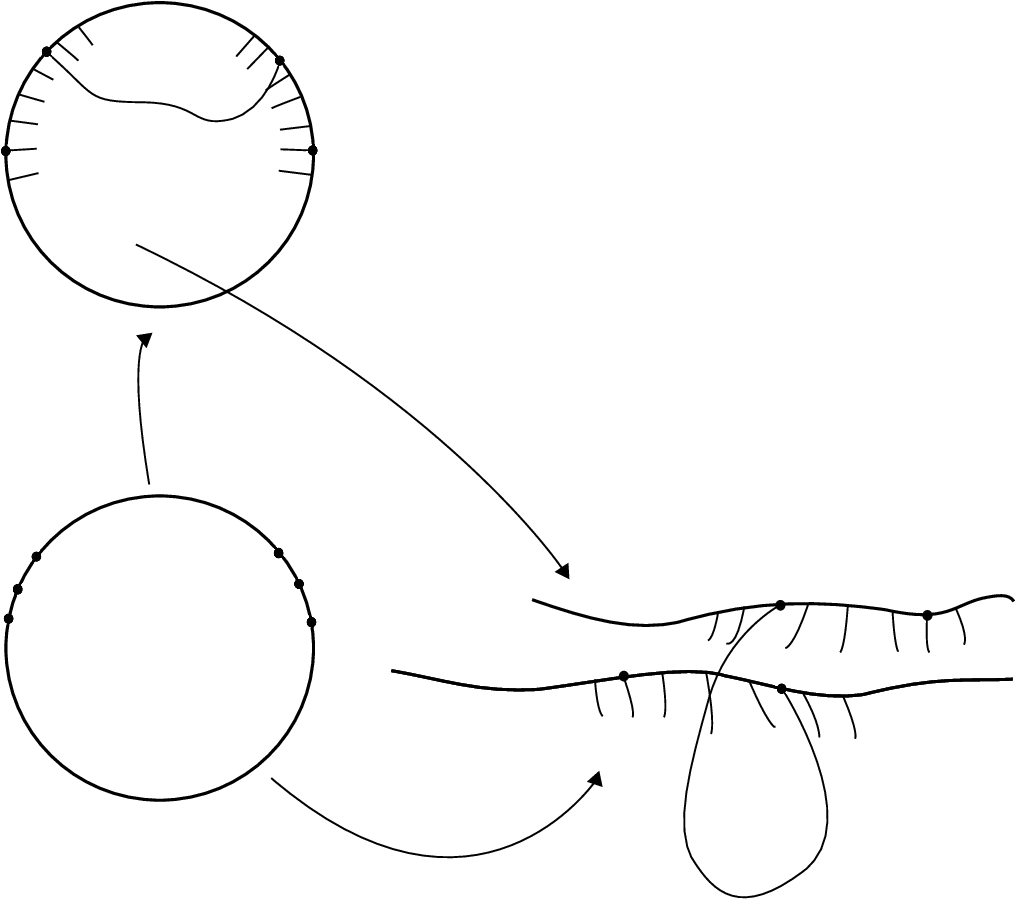}}%
    \put(1.1,6){$\sigma_k$}
    \put(-0.3,3.8){$p_k^-$}
    \put(-0.4,3.3){$p^-$}
    \put(-0.7,4.2){$p_k^-e^{-i\delta_k^-}$}
    \put(3.8,3.2){$p^+$}
    \put(3.7,3.7){$p_k^+$}
    \put(3.5,4.2){$p_k^+e^{i\delta_k^+}$}
    \put(5,0.7){$\gamma_k$}
    \put(10.5,3.7){$\gamma_k(p_k^+)$}
    \put(6.7,2.9){$\gamma_k(p_k^-)$}
    \put(4.5,5.3){$\Phi_k$}
    \put(1.9,9.5){$\Delta_k$}
       \put(9.5,0.1){$\Phi_k(\Delta_k)$}
    \put(8.5,3.8){$\gamma_k(p_k^+e^{i\delta_k^+})$}
    \put(9.2,2.7){$\gamma_k(p_k^-e^{-i\delta_k^-})$}
    \put(-1.7,10){$\sigma_k(p_k^- e^{-i\delta_k^-})$}
    \put(3.4,10){$\sigma_k(p_k^+ e^{i\delta_k^+})$}
    \put(-2.2,8.9){$\sigma_k(p_k^-)=-1$}
    \put(3.8,8.9){$1=\sigma_k(p_k^+)$}
  \end{picture}
\end{small}
\end{center}
\caption{Case 2}\label{fig3}
\end{figure}\par

\medskip

\noindent{\bf Case 3: $\alpha <0$}. Let $\Delta$ be the straight segment in $\bar D^2$ (seen as a smooth path) joining $-1$ to $1$. Since $\Delta\subset \bar D^2\setminus \{i,-i\}$ we have that $V_k\circ \Delta \to V_\infty\circ \Delta$, and by the explicit form of $V_\infty$ we deduce that the unit tangent vector of the curve $V_\infty\circ \Delta$ describes an arc in $S^1$ of lenght at least $|\alpha|+\pi$ (we are using that $\Delta$ touches $S^1$ perpendicularly, and $V_\infty$ is conformal). This implies that for $k$ large enough, any $C^1$-curve of the form $\Phi_k\circ \tilde\Delta$ for a curve $\tilde \Delta\in C^1([0,1],\bar D^2)$ with $\tilde \Delta(0)=-1$, $\tilde \Delta(1)=1$ has a unit tangent vector describing an arc of length no less than $|\alpha|-o(1)$. If such a curve is minimizing $D_k$, since by Proposition \ref{PropDK} its curvature is bounded by $\bar \kappa$, its length cannot go to zero as $k\to \infty$. But this contradicts that $p^+$ and $p^-$ are pinched points , since if $\Delta_k$ is a geodesic minimizing $D_k(\sigma_k(p^+),\sigma_k(p^-))$ (with lenght going to $0$ since $p^+$ and $p^-$ are pinched), then joining $\Delta_k$ with the two arcs $\mathrm{arc}(\sigma_k(p^\pm),\pm 1)$ and using \eqref{arco0} one would obtain paths joining $-1$ to $1$ of $D_k$-length going to 0.

\medskip

The only case left is $\alpha=\pi$, and this completes the proof.
~\hfill$\Box$
\par
\medskip

In the proof of Proposition \ref{limconst} we have used the following.

\begin{Lemma}\label{closed}
Let $\Delta\in W^{2,\infty}([0,L],\C)$ be a curve satisfying $|\dot\Delta(t)|=1$ for every $t\in [0,L]$ and $\Delta(0)=\Delta(L)$. Then
$$\int_0^L  |\kappa (t)|dt>\pi,$$
where $\kappa$ is the curvature of $\Delta$.
\end{Lemma}

\begin{proof} Let $\theta:[0,L]\to \R$ be a continuous function such that $\dot \Delta(t)=e^{i\theta(t)}$ for $t\in [0,L]$. Then it is easy to see that $\dot \theta =\kappa$. We have
$\theta([0,L])=[\theta_-,\theta_+]\subset\R$ for some $\theta_-,\theta_+\in \R$.
Assume now that
\begin{equation}\label{thetabound}
\theta_+-\theta_-\le \pi,
\end{equation}
and set
$$\bar \theta :=\frac{\theta_+-\theta_-}{2},\quad v:= e^{i\bar\theta}.$$
Then since $|\theta(t)-\bar \theta|\le \frac{\pi}{2}$ for every $t\in [0,L]$, we have
$$\frac{d}{dt}\langle \Delta(t),v\rangle=\langle \dot\Delta(t),v\rangle =\langle e^{i\theta(t)}, e^{i\bar\theta}\rangle\ge 0,$$
with identity possible only for a proper subset of $[0,L]$,  where $|\theta(t)-\bar\theta|= \frac{\pi}{2}$. But this contradicts that $\Delta(0)=\Delta(L)$. In particular \eqref{thetabound} cannot hold, and we get
$$\int_0^L |\kappa(t)|dt =\int_0^L|\dot \theta(t)|dt\ge \mathrm{osc}(\theta)=\theta_+-\theta_->\pi.$$
\end{proof}

\begin{Lemma}\label{closed2}
Let $\Delta\in W^{2,\infty}([0,L],\C)$ be a curve satisfying $\dot\Delta(t)= 1$ for every $t\in [0,L]$. Assume that
\begin{equation}\label{condgamma1}
\Re(\Delta(0))=\Re(\Delta(L)),
\end{equation}
and  that for some (small) $\ve>0$ one has
\begin{equation}\label{condgamma2}
\Im(\dot\Delta(0))<\ve,\quad \Im(\dot\Delta(L))>-\ve.
\end{equation}
Then
$$\int_0^L  |\kappa (t)|dt>\frac{\pi}{2}-C\ve,$$
where $\kappa$ is the curvature of $\Delta$ and $C$ is a universal constant.
\end{Lemma}

\begin{proof} Let $\theta\in W^{1,\infty}([0,L],\R)$ be as in the proof of Lemma \ref{closed}. Then \eqref{condgamma1} implies that for some $t_1,t_2\in [0,L]$ one has $\Re(e^{i\theta(t_1)})\le 0$, $\Re(e^{i\theta(t_2)})\ge 0$ (otherwise $\dot\Delta$ would be pointing always right or always left). Condition \eqref{condgamma2} implies that $\Im(e^{i\theta(0)})\le \ve$, $\Im(e^{i\theta(L)})>-\ve$. Then we immediately infer that the oscillation of $\theta$ is at least $\frac{\pi}{2}-C\ve$, and we conclude as in the proof of Lemma \ref{closed}, using that $\kappa=\dot \theta$.
\end{proof}

Next we prove some properties concerning the set $\mathcal{P}$\,.\par

\begin{Lemma}\label{uniquepinched}
Let $p^+$, $p^-$ be dual pinched points, and assume that $\sigma_k(p^\pm)=\pm 1$. Then $\Phi_k$ is as in case 2 of Theorem \ref{convergencephi}, $B=\{a_1,a_2\}$ and $\pm 1\not\in B$.
Moreover every pinched point $p$ has only one dual $p'$ and
$|\mathrm{arc}(p,p')|\ge \frac{C}{\bar \kappa}$.
\end{Lemma}

\begin{proof}
Let us start from the first claim. If $\Phi_k$ is as in case $1$ of Theorem \ref{convergencephi}, then
\begin{equation}\label{stimaphi1}
\int_{\Delta_k} |\Phi_k'(z)| |dz|\ge C\quad \text{for every }\Delta_k\text{ with }\Delta_k(0)=-1,\, \Delta_k(1)=1,
\end{equation}
in contrast with the fact that $p^+$ and $p^-$ are pinched. Then we are in case 2 of Theorem \ref{convergencephi} and by Proposition \ref{limconst} we have $N\in \{1,2\}$. Assume now that $a_1=1=\sigma_k(p^+)$ (the reasoning is similar if $a_1=-1$). Then we compose $\Phi_k$ with M\"obius diffeomorphism $f_k(z)=\frac{z-t_k}{1-t_kz}$ where $t_k\uparrow 1$ is chosen so that for a fixed small $\delta>0$ we have for $k$ large enough
\begin{equation}\label{sttildePhi}
\int_{S^1\cap B_\delta(1)} |(\Phi_k\circ f_k)'(z)| |dz| =\frac{\pi}{2\bar \kappa}.
\end{equation}
In other words the effect of $f_k$ is to stretch the disk to remove the concentration at the point $a_1=1$, concentrating the disk towards $-1$. Then $\tilde\Phi_k:=\Phi_k\circ f_k$ is necessarily as in case 1 of Theorem \ref{convergencephi}. Moreover the corresponding $\tilde \sigma_k:= f_k^{-1}\circ \sigma_k$ still satisfies $\tilde \sigma_k(p^\pm)=\pm 1$, since $f_k$ leaves $\pm 1$ fixed. This together with \eqref{sttildePhi} contradicts that $p^+$ and $p^-$ are pinched, since by conformality and convergence of $\tilde\Phi_k$, in a neighborhood $B_{\delta/2}(1)$ we have $|\tilde\Phi_k'|\ge C$, hence  \eqref{stimaphi1} holds with $\tilde \Phi_k$ instead of $\Phi_k$. Therefore, going back to the original maps $\Phi_k$ we have proven that $\pm 1\not\in B$.

To rule out the case $N=1$ it suffices to observe that in this case $\sigma_k(p^+)$ and $\sigma_k(p^-)$ would belong to the same connected component of $S^1\setminus B$, hence, since $\Phi_k$ is as in case 2 of Theorem \ref{convergencephi}, we would get $|\mathrm{arc}(\sigma_k^{-1}(1),\sigma_k^{-1}(1))|\to 0$, which is absurd, since $\sigma_k^{-1}(\pm 1)=p^{\pm}$ and $p^+\ne p^-$.

\medskip

Let us now prove that every pinched point $p$ has a unique dual $p'$. In order to do that, it suffices to prove that given any 2 pinched points $p^+$, $p^-$ dual to each other, then $\dot \gamma_\infty(p^+)=-\dot\gamma_\infty(p^-)$. Let us therefore consider two pinched points $p^+$, $p^-$ dual to each other. Up to considering $\tilde \Phi_k:=\Phi_k\circ f_k$ and $\tilde \sigma_k=f_k^{-1}\circ \sigma_k$ for suitable M\"obius transformations $f_k$, we can assume that $\tilde \sigma_k(p^\pm)=\pm 1$. Then, by the previous part of the lemma, $\tilde \Phi_k$  blows up at two points $a_1,a_2$ different from $\pm1$. To such $\tilde \Phi_k$ we can then apply Proposition \ref{limconst} with $\mathcal{C}_\pm$ being the connected component of $S^1\setminus \{a_1,a_2\}$ containing $\pm1$. We then infer that $\dot \gamma_\infty(p^+)=-\dot\gamma_\infty(p^-)$. 

The last  claim follows from the fact that both arcs $\mathcal{A}_1$, $\mathcal{A}_1$ joining $\tilde \sigma_k(p^\pm)=\pm1$ contain a blow up point $a_1$ or $a_2$, for which
$$\int_{\mathcal{A}_i} |\tilde \kappa_k|e^{\tilde \lambda_k}|dz| = \int_{f_k(\mathcal{A}_i)} |\kappa_k|e^{\lambda_k}|dz| \ge \pi-o(1).$$
\end{proof}

\begin{Lemma}\label{close}
The set $\mathcal{P}$  is closed.
\end{Lemma}

\begin{proof}
Let $\{p_n\}$ and $\{p_n^{\prime}\}$ be respectively a sequence of pinched points and their duals, with $p_n\to p_{\infty}$ and $p_n^{\prime}\to p^{\prime}_{\infty}$ as $k\to +\infty\,.$
    \par
    We first observe that 
    $|p_n-p^{\prime}_n|\ge C>0$ for all $n\ge 0$, hence $p_{\infty}\ne p^{\prime}_{\infty}\,.$\par
    For all $p_n$ there exists curves $\Delta_{n,k}\subseteq \bar{D}^2$ with
    $\partial \Delta_{n,k}=\{\sigma_k(p_n),\sigma(p^{\prime}_n)\}$ and
$$\lim_{k\to+\infty}\int_{\Delta_{n,k}}|\Phi_k^{\prime}(z)||dz|=0\,.$$
    Since $\gamma_k\to\gamma_{\infty}$ in $C^1(S^1)$ as $k\to+\infty\,,$ 
    we have
    \begin{equation}\begin{split}\label{limgamma}
 & \lim_{k\to +\infty} \lim_{ n\to+\infty} \int_{\mathrm{arc}(p_n,p_\infty)}|  \dot\gamma_k(t)|dt =0\\
 &    \lim_{k\to +\infty} \lim_{ n\to+\infty}\int_{\mathrm{arc}(p'_n,p'_\infty)}| \dot\gamma_k(t)|dt =0\,.
\end{split}
 \end{equation}
    We set
    $$
    \tilde\Delta_{n,k}:=\Delta_{n,k}\cup\mbox{arc$(\sigma_k(p_n),\sigma_k(p_\infty))$}\cup\mbox{arc$(\sigma_k(p_n),\sigma_k(p_\infty))$}\,.$$
    For all $k$, we have $ \tilde\Delta_{n,k}\to \tilde \Delta_{\infty,k}$  as $n\to +\infty$ with
     $\partial \tilde\Delta_{k,\infty}=\{\sigma_k(p_\infty),\sigma_k(p'_\infty)\}$ and since $\Phi_k\circ \sigma_k=\gamma_k$ on $S^1$ from \rec{limgamma} we have
   $$
 \lim_{ k\to+\infty }\int_{\tilde\Delta_{k,\infty}}|\Phi_k^{\prime}(z)||dz|=  \lim_{k\to +\infty} \lim_{ n\to+\infty}\int_{\tilde \Delta_{n,k}}|\Phi_k^{\prime}(z)| |dz|=0\,.$$
 Hence $p_\infty$ is by definition a pinched point and $p'_\infty$ is its dual.
\end{proof}

We introduce now the following equivalence relation on the set $S^{1}\setminus\{{\mathcal{P}}\}\,.$
\begin{Definition}\label{equiv}
 Given $p,q\in S^{1}\setminus\{{\mathcal{P}}\}$ we say that $p\sim q$  if and only if
 there exists a sequence of paths $\Delta_k\colon[0,1]\to \bar D^2$ with $\Delta_k(0)=\sigma_k(p),
\Delta_k(1)= \sigma_k(q)$ such that
\begin{equation}\label{Dk>C}
\liminf_{k\to +\infty} d_k(\Delta_k,\sigma_k(\mathcal{P}))>0\,,
\end{equation}
where $d_k:\bar D^2\times \bar D^2\to \R^+$ is the distance defined as
\[\begin{split}
d_{k}(z,w)=&\inf\bigg\{\left(\int_0^1 |\Phi'_k(\Delta(t))|^2|\dot \Delta(t)|^2dt\right)^\frac12, \\
&\qquad \Delta\in W^{1,2}([0,1],\bar D^2), ~\Delta(0)=z, ~\Delta(1)=w\bigg\}.
\end{split}\]
\end{Definition}

 \begin{Proposition}\label{eqcl}
 Let $q\in S^{1}\setminus\{{\mathcal{P}}\}$, $\mathcal{A}_q$ and $\mathcal{B}_q$ be respectively the    equivalence class and  the connected component   containing $q\,.$ Then $\mathcal{B}_q\subseteq\mathcal{A}_q$\,.
 \end{Proposition}
 {\bf Proof of Proposition \ref{eqcl}\,.}   Let $q\in S^{1}\setminus\{{\mathcal{P}}\}$.
 We show that $\mathcal{A}_q\cap \mathcal{B}_q$ is open and closed in $\mathcal{B}_q$\,. 
 
 \par
{\bf  1. $\mathcal{A}_q\cap \mathcal{B}_q$ is open in $\mathcal{B}_q$}.
 Let $\delta>0$ small enough so that $e^{it}q \in  S^{1}\setminus\{{\mathcal{P}}\}$ for $t\in [-2\delta,2\delta]$ and
\begin{equation}\label{epscond}
\int_{\sigma_k(\mathrm{arc}(e^{-2\delta i}q,e^{2\delta i}q))}|\Phi_k'(z)| |dz| <\frac{\pi}{2\bar\kappa}\,.
\end{equation}
 Now set $q_0=e^{-i\delta}q$, $q_1=q$ and $q_2=e^{i\delta}q$. Let $f_k$ be the sequence of M\"obis transformations of $\bar D^2$  such that
 $\tilde \sigma_k(q_0)=1,\tilde \sigma_k(q_1)=e^{\frac{2\pi i}{3}},\tilde \sigma_k(q_2)=e^{\frac{4\pi i}{3}}\,.$  We apply Theorem \ref{convergencephi} to $\tilde\Phi_k:=\Phi_k\circ f_k$ and notice that if we are in case 2 of Theorem \ref{convergencephi}, then there are one or two blow-up points. In the latter case away from the blow-up points $\{a_1,a_2\}$ we have that $\sigma_k^{-1}$ locally converges to two pinched points, which implies that one of the $q_i$'s lies in $\mathcal{P}$, contradiction. In the former case for one couple of points, say $q_1$ and $q_2$ one has
$$\int_{\mathrm{arc}(q_1,q_2)}|\dot \gamma(t)|dt=\int_{\mathrm{arc}(\tilde \sigma_k(q_1),\tilde\sigma_k(q_2))}|\tilde\Phi_k'(z)||dz|\to 0,$$
contradicting that $|\dot\gamma_k|$ is bounded away from $0$ and $|\mathrm{arc}(q_1,q_2)|=\delta$.
 
Therefore we are in case 1 of Theorem \ref{convergencephi} and $\tilde\Phi_k \rightharpoonup\tilde\Phi_\infty$ in $W^{1,2}(\bar D^2)$  and in $W^{2,p}_{\loc}(\bar D^2\setminus B)$, where $\tilde \Phi_{\infty}$ is a holomorphic immersion in $\bar D^2\setminus B$, $B=\{a_1,\dots,a_N\}$ and $e^{\frac{j2\pi i}{3}}\not \in B$ for $i=0,1,2$.
Since $|\tilde \Phi_{\infty}'|>C_\delta>0$ in $\bar D^2\setminus \cup_{i=1}^N B_\delta(a_i)$, for every $p\in \mathrm{arc}(q_0,q_2)$, choosing as $\Delta_k$ the segment joining $\sigma_k(p)$ to $\sigma_k(q)$ satisfies \eqref{Dk>C}, showing that $B_\delta(q)\cap S^1\subset \mathcal{A}_q$.

\medskip


{\bf  2. $\mathcal{A}_q\cap \mathcal{B}_q$ is closed in $\mathcal{B}_q$\,}. 
Let $q_n\in \mathcal{A}_q\cap \mathcal{B}_q$ be such that
 $q_n\to q_{\infty}\in \mathcal{B}_q$.
 For every $n$ there exists $\Delta_n^k$ with 
  $ \Delta_n^k(0)= \sigma_k(q_n )$ and $ \Delta_n^k(1)=\sigma_k(q)$.
and
\begin{equation}\label{liminfnzero}
\liminf_{k\to+\infty}d_k(\Delta^k_n, \sigma_k(\mathcal{P}))>0\,.\end{equation}
Consider now the path
$\Sigma_n^k=\mathrm{arc}(\sigma_k(q_\infty),\sigma_k(q_n ))\cup \Delta_n^k$, joining
$\sigma_k(q_\infty)$ to $\sigma_k(q )\,.$ 
We claim that
$$\liminf_{k\to+\infty}d_k(\Sigma^k_n, \sigma_k(\mathcal{P}))>0\,.$$
Indeed, considering \eqref{liminfnzero}, it suffices to prove that for $n$ sufficiently large
\begin{equation}\label{liminfnzerobis}
\liminf_{k\to+\infty}d_k(\mathrm{arc}(\sigma_k(q_\infty),\sigma_k(q_n )), \sigma_k(\mathcal{P}))>0\,.
\end{equation}
Assume by contradiction that the liminf in \eqref{liminfnzerobis} is zero. 

For every $k$ and $n$, let  $q_n^k \in \mathrm{arc}(q_\infty,q_n)$ and  $p_n^k\in \mathcal{P}$ such that
$$\liminf_{k\to+\infty} D_k(q_n^k,p_n^k)=0.$$
Up to subsequence $q_n^k\to    q_{\infty}$ and $p_n^k\to p_\infty\in  {{\mathcal{P}}}$ as $n,k\to\infty$ and
 $$\lim_{k\to+\infty}\lim_{n\to +\infty}D_k(q_n^k,p_n^k)=\lim_{k\to+\infty}D_k(q_\infty,p_\infty)=0\,,$$
but   this contradicts that $q_\infty\notin  {{\mathcal{P}}} \,.$  This contradiction proves that $q_\infty\in \mathcal{A}_q\cap \mathcal{B}_q\,,$ hence
 $\mathcal{A}_q\cap \mathcal{B}_q$ is closed in $\mathcal{B}_q$\,.
~\hfill $\Box$
 
  \begin{Proposition}\label{eqclbis}
 Let $\mathcal{A}$ be an equivalence class in $S^{1}\setminus\{{\mathcal{P}}\}$\,.
 Then there exists a sequence $f_k\colon \bar D^2\to \bar D^2$  of M\"obius transformations such that
$\tilde \Phi_k:=\Phi_k\circ f_k\rightharpoonup \tilde\Phi_\infty$ in $W^{2,p}_{\loc}(\bar D^2\setminus B$), $B=\{a_1,\ldots,a_N\}$, and letting as usual $\tilde \sigma_k$ be such that $\gamma_k=\tilde\Phi_k\circ\tilde\sigma_k$, one has $\tilde\sigma_k^{-1}\rightharpoonup\psi_\infty$ in $W^{2,p}_{\loc}(S^1\setminus B)$,
\begin{equation}\label{Acond}
\psi_\infty(S^1\setminus B)=\mathcal{A}
\end{equation}
and $\gamma_{\infty}(\mathcal{A})=\tilde\Phi_{\infty}(S^1\setminus B)\,.$ In fact $(\gamma_\infty)_*[\mathcal{A}]=(\tilde \Phi_\infty)_*[S^1\setminus B].$
\end{Proposition}
\begin{proof} Given $q\in \mathcal{A}$ take $f_k$ as in the proof of Proposition 
\ref{eqcl} and set $\tilde\Phi_k:=\Phi_k\circ f_k$. 
We have shown that $\tilde \Phi_k \rightharpoonup\tilde\Phi_\infty$ in $W^{1,2}(\bar D^2)$ and in $W^{2,p}_{\loc}(\bar D^2\setminus B)$ for a finite set $B=\{a_1,\ldots,a_N\}$, where $\tilde \Phi_{\infty}$ is a holomorphic immersion (Theorem \ref{convergencephi}, case 1). In particular this implies that $\psi_k:=\tilde \sigma_k^{-1}$ is bounded in $W^{2,p}_{\loc}(S^1\setminus B)$ and up to a subsequence $\psi_k\rightharpoonup \psi_\infty$ in $W^{2,p}_{\loc}(S^1\setminus B)$. Clearly
$$\psi_\infty(S^1\setminus B)\subset \mathcal{A}.$$
Conversely, given $p\not\in \psi_\infty(S^1\setminus B)$, we want to show that $p\not\in \mathcal{A}$. Given such $p$ we have $\tilde \sigma_k(p)\to a_i$ for some $a_i\in B$, since otherwise we would have $p=\psi_k\circ\tilde \sigma_k(p)\to \psi_\infty(p_*)$ for $p_*\in S^1\setminus B$.
Since 
$\nabla \tilde\Phi_\infty\in L^2(D^2)$, from Fubini's Theorem we can find a sequence $\delta^i_n\to 0$ such that
\begin{equation}\label{conai}
\lim_{n\to +\infty}\int_{\partial B(a_i,{\delta^i_n})\cap \bar D^2}|\nabla \tilde\Phi_\infty(z)|^2|dz|=0\,.
\end{equation}
For every $a_i$, set $\{p_{k,n}^{i,-},p_{k,n}^{i,+}\}={\tilde\sigma}^{-1}_k(\partial B(a_i,{\delta^i_n})\cap S^1)\,.$
We have $|p_{k,n}^{i,-}-p_{k,n}^{i,+}|>C_0$ for any $n$ and $k$ large enough, since by definition of the blow-up points one has for $k$ large enough

$$\int_{\mathrm{arc}(p_{k,n}^{i,-},p_{k,n}^{i,+})}|\dot \gamma_k(t)|dt=\int_{B(a_i,{\delta^i_n})\cap S^1} e^{\lambda_k(z)}|dz| >\frac{\pi}{ 2}\,.$$ Therefore up to subsequence
$p_{k,n}^{i,-}\to p^{i,-}_\infty$ and $p_{k,n}^{i,+}\to p^{i,+}_\infty$ with $p^{i,+}_\infty\ne  p^{i,-}_\infty$ and 
$$\lim_{k\to\infty}D_k(\tilde\sigma_k(p^{i,-}_\infty),  \tilde \sigma_k(p^{i,+}_\infty))=0$$ In particular $p^{i,-}_\infty$ and $p^{i,+}_\infty$ are pinched. 
Then condition \rec{conai} implies that any path $\Delta_k$ joining $\tilde \sigma_k(q)$ and $\tilde \sigma_k(p)$ for $k$ large enough it close to $\tilde \sigma_k(p^{i,-}_{\infty})\in \tilde \sigma_k(\mathcal{P})$, hence $p\in S^1\setminus \mathcal{A}$.

Finally
\[\begin{split}
(\gamma_\infty)_*[\mathcal{A}]&=\lim_{\delta\to 0}(\gamma_\infty)_*[\psi_\infty(S^1\setminus \cup_{a_i\in B} B(a_i,\delta))]\\
&=\lim_{\delta\to 0}\lim_{k\to\infty}(\gamma_k)_*[\tilde\sigma_k^{-1}(S^1\setminus \cup_{a_i\in B} B(a_i,\delta))] \\
&=\lim_{\delta\to 0}\lim_{k\to\infty}(\tilde\Phi_k)_*[S^1\setminus \cup_{a_i\in B} B(a_i,\delta)]\\
&=\lim_{\delta\to 0}(\tilde\Phi_\infty)_*[S^1\setminus \cup_{a_i\in B} B(a_i,\delta)]\\
&=(\tilde\Phi_\infty)_*[S^1\setminus B].
\end{split}
\]
\end{proof}

 \subsection{Quantization result: Proof of Theorems \ref{mainth} and \ref{mainth2}}
 In this section we prove Theorems \ref{mainth} and \ref{mainth2}. In Theorem \ref{mainth} we
will show that under the hypothesis of Theorem \ref{convergencephi}
 $\kappa_k e^{\lambda_k}\rightharpoonup\mu$ weakly in the sense of Radon measures  where $\mu$ is a Radon measure which is the sum of a locally bounded (possibly vanishing) function and a (possibly empty) sum of Dirac masses.
 We also give precise estimates on the coefficients of the Dirac masses.
 In the Theorem \ref{mainth2} we show that up to a suitable choice of M\"obius transformations we can ``detect" all the connected components arising in the limit.
 
 \medskip
 
\noindent   {\bf Proof of Theorem \ref{mainth}\,.} 
From Theorem \ref{convergencephi} there is a (possibly empty) set 
 $B=\{a_1,\ldots, a_N\}\subset S^1$ such that \eqref{convlambdak} holds. 
Moreover from \eqref{boundlength} and \eqref{boundcurv} it follows that
 $\|(-\Delta)^\frac12\lambda_k\|_{L^1(S^1)}\le C$.  Therefore \eqref{rapr} implies
$$\|\lambda_k-\bar \lambda_k\|_{L^q(S^1)}\le C\quad \text{for every }q<+\infty.$$
Up to extracting a further subsequence we have $v_k:=\lambda_k-\bar \lambda_k \rightharpoonup v_{\infty}$ in $L^q(S^1)$ and
\begin{equation}\label{radon}
\kappa_k e^{\lambda_k} \weak \mu, \quad (-\Delta)^\frac12 v_k\weak (-\Delta)^\frac12 v_\infty =\mu- 1\quad \text{ in }\mathcal{M}(S^1),
\end{equation}
where $\mathcal{M}(S^1)$ denotes the space of finite signed measures on $S^1$.
Up to a subsequence we also have $\kappa_k\stackrel{*}{\rightharpoonup}\kappa_{\infty}$ in $L^{\infty}(S^1)\,.$ We now distinguish three cases.

\noindent{\bf Case 1.} Suppose that we are in case 2 of Theorem \ref{convergencephi}  and $N=1$, i.e. $\lambda_k\to -\infty$ locally uniformly in $S^1\setminus\{a_1\}$. Then $\mu=c_1\delta_{a_1}$, and since
$$\int_{S^1}\kappa_k e^{\lambda_k}d\theta=2\pi,$$
it follows at once that $c_1=2\pi$. The explicit form of $v_\infty$ follows from Lemma \ref{lemmafund4}.

\noindent{\bf Case 2.} Suppose that we are in case 2 of Theorem \ref{convergencephi}  and $N>1$. Then we conclude applying Proposition \ref{limconst}, which in particular implies that $N=2$ and $\mu=\pi \delta_{a_1}+\pi\delta_{a_2}$. Again the explicit form of $v_\infty$ follows from Lemma \ref{lemmafund4}.

\noindent{\bf Case 3.} Suppose that we are in case 1 of Theorem \ref{convergencephi}, i.e. $\lambda_k\ge -C$. Then $\lambda_k\rightharpoonup \lambda_\infty$ weakly in $W^{1,p}_{\loc}(S^1\setminus B)$ and for every $\varphi\in C^\infty_c(S^1\setminus B)$ we have
$$0=\lim_{k\to\infty}\int_{S^1}(\lambda_k(-\Delta)^\frac12\varphi -(\kappa_ke^{\lambda_k}-1)\varphi)d\theta=\lim_{k\to\infty}\int_{S^1}(\lambda_\infty(-\Delta)^\frac12\varphi -(\mu-1)\varphi)d\theta.$$
In particular the distribution
$$T_\infty:=(-\Delta)^\frac12 \lambda_\infty - \mu+1$$
is supported in $B$, and since by \eqref{radon} $T_\infty\in \mathcal{M}(S^1)$, the order of $T_\infty$ (as distribution) is $0$, hence
$$T_\infty=\sum_{j=1}^N c_j\delta_{a_j}.$$
In order to compute the coefficients $c_j$ let $\chi_\delta:S^1\to\R$ be $1$ on $S^1\cap \cup_{j=1}^n B(a_j,\delta)$ and $0$ otherwise.
We rewrite the equation \rec{liouvfrack} as follows:
\begin{equation} \label{liouvfrac2}
(-\Delta)^\frac12 \lambda_k=(1-\chi_{\delta})\kappa_k e^{\lambda_k}+\chi_{\delta}\kappa_k e^{\lambda_k}-1\,.
\end{equation}
Since
$$\lim_{k\to\infty}(1-\chi_{\delta})\kappa_k e^{\lambda_k}= (1-\chi_{\delta})\kappa_\infty e^{\lambda_\infty}\quad \text{in }\mathcal{D}'(S^1)\,,$$
testing \eqref{liouvfrac2} with $\varphi\in C^\infty(S^1)$ and letting $k\to \infty$ we get
\begin{equation*}
\int_{S^1}(\lambda_\infty(-\Delta)^\frac12 \varphi -(1-\chi_\delta)\kappa_\infty e^\lambda_\infty \varphi +\varphi)d\theta =\lim_{k\to\infty}\int_{S^1}\chi_\delta \kappa_k e^{\lambda_k}\varphi d\theta\,,
\end{equation*}
and letting $\delta\to 0$ we infer
$$ \langle T_\infty,\varphi\rangle =\lim_{\delta\to 0} \lim_{k\to\infty}\int_{S^1}\chi_\delta \kappa_k e^{\lambda_k}\varphi d\theta.$$
By choosing $\varphi =1$ in a neighborhood of $a_j$ for a fixed $j$, and $\varphi=0$ in a neighborhood of $B\setminus\{a_j\}$ we get
$$c_j=\lim_{\delta\to 0} \lim_{k\to\infty}\int_{S^1\cap B(a_j,\delta)} \kappa_k e^{\lambda_k} d\theta.$$
We now want to compute $c_j$ for a fixed $j\in \{1,\dots,N\}$. Consider the M\"obius transformation $f_k(z)=\frac{z-t_k a_j}{1-t_k\bar a_jz}$, and $\tilde \Phi_k:=\Phi_k\circ f_k$, for a sequence $t_k\uparrow 1$ to be chosen. By Corollary \ref{cormob} we have
$$\tilde \lambda_k := \log |\tilde \Phi_k'|= \lambda_k\circ f_k +\log|f_k'|,\quad \tilde\kappa_k:=\kappa_k\circ f_k,$$
and
$$(-\Delta)^\frac12 \tilde\lambda_ k=\tilde \kappa_k e^{\tilde \lambda_k}-1.$$
Since since $\log|f_k'| \to -\infty$ locally uniformly in $\bar D^2\setminus\{a_j\}$, and $\log|f_k'(a_k)|\to\infty$ it is not difficult to see that if $t_k\uparrow 1$ slowly enough, then $\tilde\lambda_k \to -\infty$ uniformly locally in $\bar D^2\setminus\{a_1,-a_1\}$ and we can apply Proposition \ref{limconst} to $\tilde \Phi_k$, and obtain that
$$\tilde\kappa_ke^{\tilde\lambda_k}\weak \pi(\delta_{a_j}+\delta_{-a_j}).$$
With a change of variable we then get
$$\pi=\lim_{\delta\to 0}\lim_{k\to\infty}\int_{S^1\cap B(a_j,\delta)}\tilde \kappa e^{\tilde\lambda_k}d\theta =\lim_{\delta\to 0}\lim_{k\to\infty}\int_{f_k(S^1\cap B(a_j,\delta))} \kappa e^{\lambda_k}d\theta =c_j,$$
where the last identity holds up to having $t_k\uparrow 1$ slowly enough.
\hfill$\square$

\medskip

\noindent {\bf Proof of Theorem \ref{mainth2}\,.}  From Proposition \ref{eqcl}     it follows that $S^1\setminus\{\mathcal{P}\}=\cup_{j\in J}\mathcal{A}_i$
where $J$ is an at most countable set and $\mathcal{A}_j$ is an equivalence class generated by the relation in Definition \ref{equiv}\,.
From Proposition \ref{eqclbis} it follows that for every class $\mathcal{A}_j$  there is a sequence of M\"obius transformations $f_k^j(z)$ such that
$$\tilde \Phi_k^j:=\Phi_k\circ f_k^j \rightharpoonup\tilde\Phi_{\infty}^j,~~\mbox{in $W^{2,p}_{\loc}(\bar D^2\setminus B_j)$},\quad B_j=\{b^j_1,\ldots b^j_{N_j}\}\,,$$
where $\tilde\Phi_{\infty}^j\colon \bar D^2 \setminus B_j \to\R^2$ is a conformal  immersion and $\gamma_\infty(\mathcal{A}_j)=\tilde \Phi_\infty^j(S^1\setminus B_j) $\,.
Moreover we have
$$(\gamma_\infty)_*[S^1\setminus \mathcal{P}]=\sum_{j\in J}(\tilde\Phi_\infty^j)_*[S^1\setminus B_j].$$
We have
$$\sum_{j\in J}(\gamma_\infty)_* [\mathcal{A}_j]= \sum_{j\in J}(\tilde \Phi_\infty^j)_*[S^1\setminus B_j],$$
and it remains to prove that
$$(\gamma_\infty)_*[\mathcal{P}]=0.$$
In order to do that 
let $\tau:\mathcal{P} \to \mathcal{P}$ be the bijection which to a pinched point $p$ associates its dual. For a differential form $\phi:\C \to  L(\C,\C) $ we have
\begin{equation}\label{intP}
(\gamma_\infty)_*[\mathcal{P}](\phi)=\int_{\mathcal{P}}\phi(\gamma_\infty(t))( \dot\gamma_\infty(t))dt.
\end{equation}
Now recall that
\begin{equation}\label{proptau}
\gamma_\infty(t)=\gamma_\infty(\tau(t)),\quad \dot\gamma_\infty(t)=-\dot\gamma_\infty(\tau(t)).
\end{equation}
For a sequence $t_n\in\mathcal{P}^+$ with $t_n\to t\in \mathcal{P}^+$ as $n\to \infty$ we have
\begin{equation}\label{proptau2}
\begin{split}
\gamma_\infty(t_n)&=\gamma_\infty(t)+\dot \gamma_\infty(t)(t_n-t)+o(t_n-t),\\
\gamma_\infty(\tau(t_n))&=\gamma_\infty(\tau(t))+\dot \gamma_\infty(\tau(t))(\tau(t_n)-\tau(t))+o(\tau(t_n)-\tau(t)),
\end{split}
\end{equation}
where for simplicity of notation we identified $S^1$ with the interval $[0,2\pi]$, with zero corresponding to a point in $S^1\setminus\mathcal{P}.$
Using \eqref{proptau} and \eqref{proptau2} we infer that
$$\lim_{n\to\infty}\frac{\tau(t_n)-\tau(t)}{t_n-t}=-1.$$
Then at a density point of $\mathcal{P}$ we have $\frac{d\tau}{dt}=-1$ in the sense of approximate differentials (if the density of $\mathcal P$ is everywhere $0$ then $|\mathcal P|=0$ and we are done). Therefore
\[\begin{split}
\int_{\mathcal{P}}\phi(\gamma_\infty(t))( \dot\gamma_\infty(t))dt &=- \int_{\mathcal{P}}\phi(\gamma_\infty(\tau(t)))(\dot\gamma_\infty(\tau(t)))dt\\
&=-\int_{\tau(\mathcal{P})=\mathcal{P}}\phi(\gamma_\infty(t))( \dot\gamma_\infty(t))dt,
\end{split}\]
where in the first identity we used \eqref{proptau} and in the second identity we made a change of variable. This proves that the integral in \eqref{intP} vanished for every differential form $\phi$, hence $(\gamma_\infty)_*[\mathcal{P}]=0.$

Since for every $j\in J$ the sequence $(\tilde \Phi_k^j)$ is as in case 1 of Theorem \ref{convergencephi}, i.e. setting $\lambda_k^j:= \log|(\tilde \Phi_k^j)'|_{S^1}|$ we have $|\bar \lambda_k^j|\le C$, we can apply Theorem \ref{mainth}, part iii, and it follows at once that the blow-up set of $\lambda_k^j$ is $B_j$.
~\hfill$\Box$


 \section{Relation between the Liouville equations in  $\R$ and $S^1$}

Consider the conformal map $G:D^2\to \R^2$ given by
$$G(z)=\frac{iz+1}{z+i}=\frac{z+\bar z +i(|z|^2-1)}{1+|z|^2+i(\bar z-z)}.$$
 We will use  on the domain $D^2$ the coordinate $z=\xi+i\eta$ and on the target $\R^2$ the coordinates $(x,y)$ or $x+iy$.
Writing $G$ in components,
$$G^1(z)=\Re G(z)=\frac{2\xi}{(1+\eta)^2+\xi^2},\quad G^2(z)=\Im G(z)=\frac{\xi^2+\eta^2-1}{(1+\eta)^2+\xi^2}$$
and using the polar coordinates $(r,\theta)$ on $D^2$ one easily verifies
$$
\frac{\partial G^1}{\partial r}\bigg|_{r=1}=0,\quad \frac{\partial G^2}{\partial r}\bigg|_{r=1} =\frac{1}{1+\eta},\quad 
\frac{\partial G^1}{\partial \theta}\bigg|_{r=1}= -\frac{1}{1+\eta},\quad\frac{\partial G^2}{\partial \theta }\bigg|_{r=1}=0\,.$$
Notice that $G|_{S^1}(\xi+i\eta)=\frac{\xi}{1+\eta}$, i.e. $\Pi:= G^1|_{S^1}$ is the classical stereographic projection from $S^1\setminus \{-i\}$ onto $\R$. Its inverse is
\begin{equation}\label{formulaPi}
\Pi^{-1}(x)=\frac{2x}{1+x^2}+i\left(-1+\frac{2}{1+x^2}\right).
\end{equation}
If we write $\Pi^{-1}(x)=e^{i\theta(x)}$ we get the following useful relation 
\begin{equation}\label{sintheta}
1+\sin(\theta(x))=\frac{2}{1+x^2},\quad \frac{2}{1+\Pi(\theta)^2}=1+\sin\theta,
\end{equation}
which follows easily from $\sin(\theta(x))=\Im(\Pi^{-1}(x))=\frac{1-x^2}{1+x^2}$.

\begin{Proposition}\label{proppull} Given $u:\R\to\R$ set $v:=u\circ \Pi:S^1\to\R$, where $\Pi:=G^1|_{S^1}$. Then $u\in L_\frac{1}{2}(\R)$ if and only if $v\in L^1(S^1)$. In this case
\begin{equation}\label{eqlapv1}
\lapfr v(e^{i\theta})=\frac{((-\Delta)^\frac12u)(\Pi(e^{i\theta}))}{1+\sin\theta}  ~~\mbox{in $\mathcal{D}'(S^1\setminus\{-i\})$},
\end{equation}
i.e.
$$\langle \lapfr v,\varphi\rangle=\langle \lapfr u, \varphi\circ\Pi^{-1}\rangle\quad \text{for every }\varphi\in C^\infty_0(S^1\setminus\{-i\}).$$

Further if $\lapfr u\in L^1(\R)$, or equivalently  $\lapfr v|_{S^1\setminus\{-i\}}\in L^1(S^1)$, then
\begin{equation}\label{eqlapv2}
\lapfr v(e^{i\theta})=\frac{((-\Delta)^\frac12 u)(\Pi(e^{i\theta}))}{1+\sin\theta}-\gamma \delta_{-i} \quad \text{in }\mathcal{D}'(S^1),\quad \gamma=\int_{\R}\lapfr u dx\,.
\end{equation}
\end{Proposition}

\noindent {\em Proof of Proposition \ref{proppull}\,.} Since
$$\int_{S^1}|v|d\theta= \int_{\R}\frac{2|v(\Pi^{-1}(x))|}{1+x^2}dx$$
it is clear that $v\in L^1(S^1)$ if and only if $u\in L_\frac{1}{2}(\R)$.


Given now $\varphi\in C^\infty_c(S^1\setminus \{-1\})$ set $\psi:=\varphi\circ \Pi^{-1}\in C^\infty_c(\R)$ and let $\tilde \varphi \in C^\infty (\bar D^2)$ and $\tilde \psi\in C^\infty\cap L^\infty(\bar R^2_+)$ be the harmonic extensions of $\varphi$ and $\psi$ given by the Poisson formulas \eqref{Poisson1} and \eqref{Poisson2} respectively. It is not difficult to see that $\tilde \psi\circ G|_{\bar D^2}$ is continuous, harmonic in $D^2$ and it coincides with $\tilde\varphi$ on $S^1$. Then by the maximum principle $\tilde \varphi =\tilde\psi \circ G$ in $\bar D^2\setminus \{-i\}$.

Using polar coordinates we compute
\begin{eqnarray*}
\frac{\partial \tilde \varphi}{\partial r}\bigg|_{r=1}\circ \Pi^{-1} 
&=&-\frac{\partial (\tilde \varphi \circ G^{-1})}{\partial x}\,\frac{\partial G^1}{\partial r}\bigg|_{r=1}-\frac{\partial (\tilde \varphi \circ G^{-1})}{\partial y}\,\frac{\partial G^2}{\partial r}\bigg|_{r=1}\\
&=& -\frac{\partial \tilde \psi}{\partial y}\bigg|_{y=0} \frac{1+x^2}{2}.
\end{eqnarray*}
Then using Propositions \ref{lapeq2} and \eqref{lapeq} we get
\begin{equation*}
\begin{split}
\langle (-\Delta)^\frac12 v,\varphi\rangle&=\int_{S^1} v\,\frac{\de\tilde\varphi}{\de r}\bigg|_{r=1} d\theta\\
&=\int_{\R} (v\circ \Pi^{-1}(x))\, \bigg(\frac{\de\tilde\varphi}{\de r}\bigg|_{r=1}\circ \Pi^{-1}(x)\bigg)\frac{2}{1+x^2}dx\\
&=-\int_{\R} u\, \frac{\partial \tilde \psi}{\partial y}\bigg|_{y=0} dx\\
&=\langle (-\Delta)^\frac12 u, \psi\rangle,
\end{split}
\end{equation*}
so that \eqref{eqlapv1} is proven.

In order to prove \eqref{eqlapv2} set $f:=(\lapfr v)|_{S^1\setminus\{-i\}}\in \mathcal{D}'(S^1\setminus\{-i\})$ 
and notice that
$$\|f\|_{L^1(S^1)}=\| \lapfr u\|_{L^1(\R)}=\gamma.$$
Since $f \in L^1(S^1)\subset\mathcal{D}'(S^1)$, we have
\begin{equation}\label{defT}
T:=\lapfr v- f\in \mathcal{D}'(S^1)
\end{equation}
and $\mathrm{supp}(T)\subset \{-i\}$. We claim that $T=c\delta_{-i}$ for some constant $c$.
Up to a rotation of $S^1$, it is convenient to assume that $T$ is supported at $\{1\}$. In this case we can write
$$T=\sum_{k=0}^N c_k D^k\delta_0,$$
for some $N\in \mathbb{N}$ and $c_0,\dots, c_N\in \mathbb{C}$, which leads to
\begin{equation}\label{formulaT}
\langle T,\varphi\rangle =\sum_{k=0}^N c_k(-1)^k D^k\varphi_0=\sum_{k=0}^Nc_k\sum_{n\in \mathbb{Z}}(-in)^k\overline{\hat\varphi(n)},\quad \text{for }\varphi\in \mathcal{D}(S^1).
\end{equation}
 On the other hand according to \eqref{fraclapl6} we have for $\varphi\in \mathcal{D}(S^1)$
\begin{equation}\label{formulalapv}
\begin{split}
\langle \lapfr v,\varphi\rangle&=\int_{S^1}v(\theta)\sum_{n\in\mathbb{N}}|n|\overline{\hat\varphi(n)}e^{-in\theta}\, d\theta\\
&=\sum_{n\in \mathbb{N}}|n|\overline{\hat\varphi(n)}\int_{S^1}v(\theta)e^{-in\theta}d\theta\\
&=2\pi \sum_{n\in \mathbb{N}}|n|\hat v(n)\overline{\hat\varphi(n)},
\end{split}
\end{equation}
where the sum can be moved outside the integral because $\sum_{n\in\mathbb{N}}|n||\hat\varphi(n)|<\infty$.
Similarly
\begin{equation}\label{formulaf}
\langle f,\varphi\rangle= 2\pi\sum_{n\in\mathbb{N}}\hat f(n)\overline{\hat\varphi(n)},\quad\text{for }\varphi \in \mathcal{D}(S^1).
\end{equation}
Clearly \eqref{defT}, \eqref{formulaT}, \eqref{formulalapv} and \eqref{formulaf} are compatible only if $c_k=0$ for $k=1,\dots, N$, hence proving (up to rotating back) that $T=c_0\delta_{-i }$, as claimed.
Finally, testing with $\varphi=1$ we obtain
$$0=\langle \lapfr v,1\rangle=\langle f,1\rangle+\langle T,1\rangle=\|\lapfr u\|_{L^1}+c_0,$$
which implies that $c_0=-\|\lapfr u\|_{L^1}.$~\hfill$\Box$
\par
\medskip

Given now $u\in L_{\frac{1}{2}}(\R)$ we want to define a function $\lambda\in L^1(S^1)$ such that
$$\Pi^*(e^{2u}|dx|^2)= e^{2\lambda}|d\theta|^2,$$
where 
$\Pi^*$ denotes the pull-back of the stereographic projection, while $|dx|^2$ and $|d\theta|^2$ are the standard metrics on $\R{}$ and $S^1$ respectively.
Since $$\Pi^*(e^{2u}|dx|^2) =\left(\frac{\de \Pi}{\de \theta}\right)^2 e^{2u(\Pi(\theta))}|d\theta|^2$$
we find
\begin{equation}\label{deflambda2}
\lambda(\theta)=u(\Pi(\theta))+\log\left|\frac{\de \Pi}{\de \theta}\right| =u(\Pi(\theta))- \log\left(1+\sin\theta\right),
\end{equation}
or equivalently and using \eqref{sintheta}
\begin{equation}\label{deflambda3}
u(x)=\lambda(\Pi^{-1}(x))+\log\left(\frac{2}{1+x^2}\right).
\end{equation}
Using Proposition \ref{proppull} we can now easily relate $\lapfr u$ and $\lapfr \lambda$.

\begin{Proposition}\label{proppull2} Given $u:\R\to\R$ set $\lambda$ as in \eqref{deflambda2}. Then $u\in L_\frac{1}{2}(\R)$ if and only if $\lambda\in L^1(S^1)$, and $(-\Delta)^\frac12 u\in L^1(\R)$ if and only if $(-\Delta)^\frac{1}{2} \lambda\in L^1(S^1\setminus\{-i\})$. In this case $u$ solves \eqref{liou5int} if and only if $\lambda$ solves
\begin{equation}\label{liou6}
(-\Delta)^\frac{1}{2}\lambda=\kappa\, e^{\lambda}-1 +\left(2\pi -c\right)\delta_{-i}\quad \text{in }S^1.
\end{equation}
with $\kappa= V\circ\Pi$ and $c=\|\lapfr u\|_{L^1(\R)}$.
\end{Proposition}

\begin{proof} This follows at once from Proposition \ref{proppull2} and Lemma \ref{lemmaG} below.
\end{proof}

\begin{Lemma}\label{lemmaG}
We have
$$(-\Delta)^\frac{1}{2}\log(1+\sin\theta)=1-2\pi \delta_{-i}.$$
\end{Lemma}

\begin{proof}
Notice that by \eqref{sintheta} we can write
$$\log(1+\sin\theta)=u_{1,0}(\Pi(\theta)),\quad u_{1,0}(x)=\log\left(\frac{2}{1+x^2}\right).$$
Then Propositions \ref{sollog} and \ref{proppull} imply
\begin{equation*}
\begin{split}
\lapfr \log(1+\sin\theta)&=\frac{\lapfr u(\Pi(\theta))}{1+\sin\theta}- \|\lapfr u\|_{L^1}\delta_{-i}\\
&= \frac{e^{u_{1,0}(\Pi(\theta))}}{1+\sin\theta}- \delta_i\int_{\R}e^{u_{1,0}(x)}dx\\
&=1-2\pi\delta_{-i}.
\end{split}
\end{equation*}
\end{proof}


 

 

\section{Proof of Theorem \ref{trm1} and Proposition \ref{nonex}}\label{sec:proofs}
  



Before proving Theorem \ref{trm1} we show that the functions defined in \eqref{specsol} are indeed solutions of \eqref{liou}-\eqref{area}.

\begin{Proposition}\label{sollog} For every $\mu>0$ and $x_0\in\R{}$ the function $u_{\mu,x_0}$ defined in \eqref{specsol}
belongs to $L_\frac12(\R)$ satisfies \eqref{area} with $L=2\pi$ and solves \eqref{liou}.
\end{Proposition}
{\bf Proof\,.} That $u_{\lambda,x_0}\in L_{\frac{1}{2}}(\R)$ and
$\int_{\R}e^{u_{\lambda,x_0}}dx=2\pi$
is elementary.
The equation is  invariant under translations and  dilations in the sense that for all $x_0\in\R$ and $\lambda >0$ if $u$ is a solution of \eqref{liou} then $u(\lambda (x+x_0))+\log(\lambda)$ is a solution of \eqref{liou} as well, hence it suffices to prove that $u_{1,0}(x)=\log\big(\frac{2}{1+x^2}\big)$ is a solution.
From Proposition \ref{lapeq} we get with integration by parts 
\[
\begin{split}
\pi(-\Delta)^\frac{1}{2}u_{1,0}(x)&=\lim_{\ve\to 0}\int_{\R{}\setminus[x-\ve,x+\ve]}\frac{\log\Big(\frac{1+y^2}{1+x^2}\Big)}{(x-y)^2}dy\\
&=\lim_{\ve\to 0}\bigg\{-\frac{\log\Big(\frac{1+y^2}{1+x^2}\Big)}{y-x}\bigg|_{-\infty}^{x-\ve}-\frac{\log\Big(\frac{1+y^2}{1+x^2}\Big)}{y-x}\bigg|_{x+\ve}^{\infty}\\
&\qquad +\int_{\R{}\setminus[x-\ve,x+\ve]}\frac{2y}{(y-x)(1+y^2)}dy\bigg\}\\
&=\lim_{\ve\to 0}\bigg\{\frac{2\arctan(y)+x\log\Big(\frac{(y-x)^2}{1+y^2} \Big)}{1+x^2}\bigg|_{-\infty}^{x-\ve}\\
&\qquad+\frac{2\arctan(y)+x\log\Big(\frac{(y-x)^2}{1+y^2} \Big)}{1+x^2}\bigg|_{x+\ve}^{\infty}\bigg\}\\
&=\frac{2\pi}{1+x^2}=\pi e^{u_{1,0}(x)}.
\end{split}
\]

\begin{Theorem}\label{lemmaBM} There exist constants $C_1, C_2>0$ such that for any $\ve\in (0,\pi)$ one has
\begin{equation}\label{stimaMT5}
C_1\le  \sup_{u\in \tilde H^{1,1}_\Delta (I),\, \|(-\Delta)^\frac12 u\|_{L^1(I)}\le 1 }\frac{\ve}{|I|} \int_{I} e^{(\pi-\ve)|u|}  d\theta \le C_2,
\end{equation}
where $\tilde H^{1,1}_\Delta(I):=\{u\in L^1(\mathbb{R}):\mathrm{supp}(u)\subset \bar I,\,(-\Delta)^\frac12 u\in L^1(\mathbb{R})\}$.
\end{Theorem}

\begin{Lemma}\label{lemmagreen4} 
The Green function of $(-\Delta)^\frac{1}{2}$ on the interval $I=(-1,1)$ 
can be decomposed as 
$$G_\frac12 (x,y)= F_\frac12(|x-y|)+H_\frac12(x,y),$$
where $F_\frac12(x):=\frac{1}{\pi}\log\frac{1}{|x|}$
and $H_\frac12$ is upper bounded.
\end{Lemma}

\begin{proof} 
This follows from the explicit expression of $G(x,y)$ (see e.g \cite{BGR} or \cite{Buc}), namely
$$G(x,y)=\frac{1}{2\pi}\int_0^{r_0(x,y)}\frac{1}{\sqrt{r(r+1)}}dr=\frac{1}{\pi} \log(\sqrt{r_0(x,y)}+\sqrt{r_0(x,y)+1}),$$
where
$$r_0(x,y):=\frac{(1-|x|^2)(1-|y|^2)}{|x-y|^2}.$$
\end{proof}

\noindent\emph{Proof of Theorem \ref{lemmaBM}.}
Up to a translation and dilation we can assume that $I=(-1,1)$. With Lemma \ref{lemmagreen4}  we write for $u\in \tilde H^{1,1}_\Delta(I)$ and $f:=(-\Delta)^\frac12 u$
$$|u(x)|=\left|\int_I G(x,y) f(y)dy\right|,$$
and  we bound
$$G(x,y)\le \frac{1}{\pi}\log\left(\frac{2}{|x-y|}\right)+C,\quad x,y,\in I,$$
hence
\begin{equation}\label{stimaimp}
|u(x)|\le \frac{1}{\pi}\int_I \log\left(\frac{2}{|x-y|}\right)|f(y)|dy + C,
\end{equation}
and exactly as in \eqref{eqpfMT4} one gets
$$\int_I e^{(\pi-\ve) |u(x)|}dx\le C \int_I |f(y)|\int_I \left(\frac{2}{|x-y|}\right)^{1-\frac\ve\pi} dxdy\le\frac{C}{\ve}.$$
The rest of the proof is also similar to the proof of Theorem \ref{MT4}.
\hfill$\square$

\begin{Remark} A slight modification of \eqref{stimaMT5} is
\begin{equation}\label{stimaMT6}
C_1\le  \sup_{u=F_\frac12*f,\, \mathrm{supp}(f)\subset \bar I, \,\|f\|_{L^1(I)}\le 1 }\frac{\ve}{|I|} \int_{I} e^{(\pi-\ve)|u|}  d\theta \le C_2,
\end{equation}
where $F_\frac12$ is as in Lemma \ref{lemmagreen4}. The proof of \eqref{stimaMT6} is similar to the proof of \eqref{stimaMT5}, since $u=F_\frac12*f$ obviously satisfies \eqref{stimaimp}. An alternative proof of a non-sharp version of \eqref{stimaMT6}, namely
$$\sup_{u=F_\frac12*f,\, \mathrm{supp}(f)\subset \bar I, \,\|f\|_{L^1(I)}\le 1 } \int_{I} e^{\delta |u-\bar u|}  d\theta \le C_2,\quad \text{for some }\delta>0,\quad \bar u:=\Intm_I udx,$$
can be obtained noticing that for $u=F_\frac12*f$ one has $[u]_{BMO(I)}\le C[F_\frac{1}{2}]_{BMO(\R{})}\|f\|_{L^1(I)}$, and one can apply the John-Niremberg inequality.
\end{Remark}

\begin{Proposition}\label{trmrappr} Let $u\in L_{\frac{1}{2}}(\R{})$ satisfy \eqref{liou}-\eqref{area}. Then there is a constant $C_0\in \R$ such that
\begin{equation}\label{liouint}
u(x)=\frac{1}{\pi}\int_{\R}\log\left(\frac{1+|y|}{|x-y|}\right)e^{u(y)}dy+C_0.
\end{equation}
\end{Proposition}

In the proof of Proposition \ref{trmrappr} we use two lemmata.

\begin{Lemma}\label{green} For any $f\in L^1(\R)$ the function
\begin{equation}\label{defw}
w(x):=\mathcal{I}[f](x):=\frac{1}{\pi}\int_{\R}\log\left(\frac{1+|y|}{|x-y|}\right) f(y)dy
\end{equation}
is well defined, belongs to $L_\frac12(\R)$ and satisfies 
\begin{equation}\label{deltaw}
(-\Delta)^\frac{1}{2}w=f \text{ in }\mathcal{S}'.
\end{equation}
\end{Lemma}
{\bf Proof of Lemma \ref{green}\,.}
Let us first assume that $f$ belongs to the Schwartz space $\mathcal{S}$. 
Remember that for $F(x):=\frac{1}{\pi}\log\left(\frac{1}{|x|}\right)$ we have (see e.g. \cite[page 132]{vla})
\begin{equation}\label{fourierlog0}
\hat F(\xi)= \mathcal{P}\frac{1}{|\xi|}+C\delta_0\quad \text{in }\mathcal{S}',
\end{equation}
where $\mathcal{P}\frac{1}{|\xi|}\in \mathcal{S}'$ is the tempered distribution defined by
\begin{equation}\label{fourierlog}
\left\langle \mathcal{P}\frac{1}{|\xi|},\varphi\right\rangle = \int_{|\xi|\le 1}\frac{\varphi(\xi)-\varphi(0)}{|\xi|}d\xi+\int_{|\xi|>1}\frac{\varphi(\xi)}{|\xi|}d\xi,\quad \varphi\in\mathcal{S}.
\end{equation}
For every $f \in C^\infty_c(\R{})$ one easily sees that $F*f\in C^\infty(\R)$ and $F*f\in L_\frac12(\R)$.
Then
\begin{equation}\label{K*f}
\begin{split}
\langle (-\Delta)^\frac12(F*f),\varphi\rangle &:=\int_{\R} (F*f)\,\mathcal{F}^{-1}(|\xi|\hat \varphi) dx\\
& = \int_{\R} F \, (\tilde f*\mathcal{F}^{-1}(|\xi|\hat \varphi)) dx \\
&=  \int_{\R} F \,\mathcal{F} (\mathcal{F}^{-1}(\tilde f*\mathcal{F}^{-1}(|\xi|^{2\sigma}\hat \varphi))) dx \\
&= \frac{1}{2\pi} \int_{\R} F \,\mathcal{F} (\hat f |\xi|\hat{\tilde \varphi} ) dx \\
&= \frac{1}{2\pi} \int_{\R} \hat f \hat{\tilde \varphi} d \xi =\int_{\R}f\varphi dx,
\end{split}
\end{equation}
where in order to apply \eqref{fourierlog} in the fifth identity can approximate the function $\psi(\xi)=\hat f |\xi|\hat\varphi$ by a sequence of functions $\psi_\ve = \hat f \eta_\ve \hat{\tilde \varphi} \in \mathcal{S}(\R)$ with $\eta_\ve\in C^\infty(\R{})$ suitably chosen (see for instance \cite{JMMX}). 
Hence
$\lapfr(F*f)=f$ in $\mathcal{D}'(\R),$
and since $f\in \mathcal{D}(\R)$ the indentity also holds in a strong sense.
Moreover, since obviously
$$\lapfr \left(\frac{1}{\pi}\int_{\R}\log(1+|y|) f(y)dy\right)=0$$
we see that \eqref{deltaw} is satisfied when $f\in \mathcal{D}(\R)$.

For a general function $f\in L^1(\R)$ we can find a sequence $(f_k)\subset \mathcal{D}(\R)$ with $f_k\to f$ in $L^1(\R)$ and take $\varphi\in \mathcal{S}(\R)$.
Then
$$(I)_k:=\left\langle \lapfr \mathcal{I}[f_k],\varphi \right\rangle =\langle f_k,\varphi\rangle \to \langle f,\varphi\rangle,$$
as $k\to\infty$, while
\begin{equation*}
(I)_k=\left\langle  \mathcal{I}[f_k],\lapfr\varphi \right\rangle = \int_{\R}\mathcal{I}[f_k](x)\, \psi(x)dx
\end{equation*}
where $\psi:=\lapfr\varphi$ satisfies
\begin{equation}\label{stimapsi}
|\psi(x)|\le C(1+|x|^2).
\end{equation}
It remains to show that
$$\int_{\R} \mathcal{I}[f_k-f](x)\, \psi(x)dx\to 0\quad \text{as }k\to\infty.$$
Define $g_k:=f_k-f\to 0$ in $L^1(\R)$.
Then from $\|h_1*h_2\|_{L^1}\le \|h_1\|_{L^1}\,\|h_2\|_{L^1}$ we get
$$\left|\int_{B(x,1)}\log\left(\frac{1+|y|}{|x-y|}\right)g_k(y)dy\right|\le \log(2+|x|)\|g_k\|_{L^1(\R)}+C\|g_k\|_{L^1},$$
and using that for $|x-y|\ge 1$ we have $\log\left(\frac{1+|y|}{|x-y|}\right)\le C(1+\log(|x|))$
$$\left|\int_{\R\setminus B(x,1)}\log\left(\frac{1+|y|}{|x-y|}\right)g_k(y)dy\right|\le C(1+\log|x|)\|g_k\|_{L^1}.$$
Therefore, taking \eqref{stimapsi} into account, we see that
$$(I)_k\to \left\langle  \mathcal{I}[f],\lapfr\varphi \right\rangle\quad \text{as }k\to\infty,$$
hence conclude that $\lapfr w= f$ in $\mathcal{S}'(\R)$.
\hfill $\square$

\begin{Lemma}\label{lemmaliou} Let $f\in L_{\frac{1}{2}}(\R)$ satisfy $(-\Delta)^\frac{1}{2} f=0$. Then $f$ is constant.
\end{Lemma}

\begin{proof} This is identical to the proof of Lemma 14 in \cite{JMMX}.
\end{proof}

\noindent \emph{Proof of Proposition \ref{trmrappr}.} Set $w(x)$ as in \eqref{defw} with $f(y):=e^{u(y)}.$
Then $(-\Delta)^{\frac{1}{2}}(u-w)=0$ by Lemma \ref{green}, hence by Lemma \ref{lemmaliou} $u-w\equiv C_0$ for some $C_0\in\R$.
\hfill $\square$

\begin{Proposition}\label{trmreg} Let $u\in L_{\frac{1}{2}}(\R{})$ satisfy \eqref{liou}-\eqref{area}. Then $u\in C^\infty(\R)$.
\end{Proposition}

\begin{proof} Up to scaling, assume that
$$\int_{-1}^1e^{u(x)}dx<\ve,$$
where $\ve$ will be fixed later.

Let us split $u=u_1+u_2$, where
\begin{equation}\label{liouint1}
u_1(x)=\frac{1}{\pi}\int_{-1}^1\log\left(\frac{1+|y|}{|x-y|}\right) e^{u(y)}dy+C_0=\frac{1}{\pi}\int_{-1}^1\log\left(\frac{1}{|x-y|}\right) e^{u(y)}dy+C_1.
\end{equation}
Then \eqref{liouint} implies that $u_2$ is defined by the same formula, integrating over $\R{}\setminus[-1,1]$ instead of $\R{}$.
It is easy to see that
$$\|u_2\|_{L^\infty([-1/2,1/2])}\le C\int_{\R{}} e^{u(x)}dx<\infty.$$
From \eqref{stimaMT6} if follows that 
given $p<\infty$, 
choosing $\ve>0$ small enough (depending on $p$) we have $e^{|u_1|}\in L^p([-1,1])$, hence $e^{u}\in L^p[-1/2,1/2].$

The same argument, together with translations and dilations, can be performed in a neighborhood of every point in $\R{}$, 
giving $e^u\in L^p_{\loc}(\R)$ for $1<p<\infty$.
Going back to \eqref{liouint} it is easy to bootstrap regularity and prove that $u$ is actually smooth.
\end{proof}

\begin{Corollary}\label{vreg} Every function $\lambda\in L^1(S^1)$ solving \eqref{eqlambda2} with $\lapfr  \lambda\in L^1(S^1)$ is smooth.
\end{Corollary}

\begin{proof} By Proposition \ref{proppull2} the function $u:\R\to\R$ given by \eqref{deflambda3} is in $L_\frac{1}{2}(\R)$ and it solves \eqref{liou}. Then by Proposition \ref{trmreg} $u$ is smooth, hence $\lambda\in C^\infty(S^1\setminus\{-i\})$. Since \eqref{eqlambda2} is invariant under rotations we have that actually $\lambda\in C^\infty(S^1)$.
\end{proof}


\begin{Lemma}\label{lemmapoho} For $u\in L_\frac{1}{2}(\R)\cap C^1(\R)$ solving \eqref{liou}-\eqref{area} set
$$\alpha:=\int_{\R}e^{u(x)}dx.$$
Then $\alpha=2\pi$.
\end{Lemma}

\begin{proof} This argument is taken from \cite{Xu} and is based on a Pohozaev-type identity. Differentiating \eqref{liouint} we obtain
$$x\frac{\de u}{\de x}=-\frac{1}{\pi}\int_{\R}\frac{x}{x-y}e^{u(y)}dy.$$
Multiplying by $e^{u(x)}$ and integrating with respect to $x$ on the interval $[-R,R]$ we get
$$(I):=\int_{-R}^R x\frac{\de u}{\de x}e^{u(x)}dx=-\int_{-R}^R\frac{1}{\pi}\int_{\R}\frac{x}{x-y}e^{u(y)}dy\,e^{u(x)}dx=:(II).$$
Integrating by parts we find
$$(I)=\int_{-R}^{R}x\frac{\de e^{u(x)}}{\de x}dx=R(e^{u(R)}+e^{u(-R)}) -\int_{-R}^R e^{u(x)}dx\to -\alpha,\quad \text{as }R\to\infty,$$
where we used that at least on a sequence $R(e^{u(R)}-e^{u(-R)})\to 0$ as $R\to\infty$, otherwise \eqref{area} would be violated.
As for $(II)$ we compute
\[
\begin{split}
(II)&=-\frac{1}{2\pi}\int_{-R}^R\int_{\R} e^{u(y)}dy\,e^{u(x)}dx-\frac{1}{2\pi}\int_{-R}^R\int_{\R} \frac{x+y}{x-y}e^{u(y)}dy\,e^{u(x)}dx\to -\frac{\alpha^2}{2\pi}+0,
\end{split}
\]
as $R\to\infty$. Therefore from $(I)=(II)$ we infer $\alpha=\frac{\alpha^2}{2\pi}$, i.e. $\alpha=2\pi$.
\end{proof}

\noindent{\bf Proof of Theorem \ref{trm1}.} Given $u\in L_\frac{1}{2}(\R)$ satisfying \eqref{liou}-\eqref{area}, by Proposition \ref{proppull2} the function
$\lambda(\theta):=u(\Pi(\theta))-\log(1+\sin\theta)$ solves
$$\lapfr \lambda=e^{\lambda} -1+(2\pi -\alpha)\delta_{-i}\quad \text{in }S^1.$$
and by Lemma \ref{lemmapoho} $\alpha=2\pi$, hence
$$\lapfr \lambda=e^{\lambda} -1 \quad \text{in }S^1.$$
By Corollary \ref{lemmaS1} $\lambda$ is of the form given by \eqref{solutions} for some $a\in D^2$.

To complete the proof write $a=\alpha e^{i\theta_0}=\alpha (t+is)$ with $\alpha,t,s\in\R$. We have
\[
u(x)=\lambda\circ\Pi^{-1}(x)+\log\left(\frac{2}{1+x^2}\right)=\log\left(\frac{2(1-\alpha^2)}{|1-\alpha (t+is)\Pi^{-1}(x)|^2(1+x^2)}\right).
\]
The right-hand side can be computed using \eqref{formulaPi}:
\[\begin{split}
u(x)&=\log\left(\frac{2(1-\alpha^2)}{\left|1+\alpha\frac{-2 t x+ s(1-x^2)}{1+x^2}-i\alpha\frac{2sx+t(1-x^2)}{1+x^2}\right|^2(1+x^2)}\right)\\
&=\log\left(\frac{2(1-\alpha^2)}{x^2(1-2\alpha s+\alpha^2)-4\alpha t x+1+2\alpha s+\alpha^2}\right).
\end{split}\]
Completing the square in the denominator on the right-hand side we get
\[
u(x)=\log\left(\frac{2(1-\alpha^2)}{(1-2\alpha s+\alpha^2)\left(x-\frac{2\alpha t}{1-2\alpha s+\alpha^2}\right)^2+\frac{(1-\alpha^2)^2}{1-2\alpha s+\alpha^2}}\right)=\log\left(\frac{2\mu}{1+\mu^2(x-x_0)^2}\right)\\
\]
with
$$x_0=\frac{2\alpha t}{1-2\alpha s+\alpha^2},\quad \mu=\frac{1-2\alpha s+\alpha^2}{1-\alpha^2}.$$
\hfill$\square$

\medskip
The following can been seen as a non-local version of the classical mean-value property of harmonic functions. It appears in \cite[Prop. 2.2.6]{sil} in a slightly different case, but with a proof which readily extends to the following case.

\begin{Proposition}\label{meanv} There exists a positive function $\gamma_1\in C^{1,1}(\R)$ with $\int_{\R}\gamma_1dx=1$ such that, setting $\gamma_\lambda(x):=\frac{1}{\lambda}\gamma_1\left(\frac{x}{\lambda}\right),$
we have
$$u(x_0)\ge u*\gamma_\lambda(x_0)$$
for every $\lambda>0$ and every $u\in L_\frac{1}{2}(\R)$ satisfying $(-\Delta)^\frac{1}{2}u \ge 0$.
\end{Proposition}

\noindent\textbf{Proof of Proposition  \ref{nonex}.} Since $(-\Delta)^\frac{1}{2}u\le 0$ we have by Proposition \ref{meanv} below
$$u(0)\le u*\gamma_\lambda(0)\quad \text{for every }\lambda>0,$$
where $\gamma_\lambda$ is as in Proposition \ref{meanv}. Since
$d\mu_\lambda(x):= \gamma_\lambda (-x)dx$ satisfies $\int_{\R} d\mu_\lambda=1$, from Jensen's inequality we get
$$\int_{\R} e^ud\mu_\lambda \ge \exp\left(\int_{\R}ud\mu_\lambda\right)=e^{u*\gamma_\lambda(0)}\ge e^{u(0)}.$$
On the other hand, since $d\mu_\lambda \le \frac{C}{\lambda}dx$, we estimate
$$\int_{\R}e^u dx\ge\frac{\lambda}{C}\int_{\R} e^u d\mu_\lambda \ge\frac{\lambda}{C}e^{u(0)}\to\infty\quad \text{as }\lambda\to\infty,$$
contradicting \eqref{area}.
\hfill $\square$

\subsection{Alternative proof of Theorem \ref{trm1}}

A more direct proof, which does not use a Pohozaev-type identity, can be given directly using an extension of Theorem \ref{propcostrlambdaintr}.

\begin{Proposition}\label{propcostrlambdabis} Let $\lambda\in L^1(S^1)$ with $L:=\|e^{\lambda}\|_{L^1(S^1)}<\infty$ satisfy 
\begin{equation} \label{liouvfracn1bis}
(-\Delta)^\frac12 \lambda=\kappa e^{\lambda}-1 +c\delta_{-i}\quad \text{in }S^1
\end{equation}
for some function $\kappa \in L^{\infty}(S^1,\R)$ and some constant $c\in\R$. Then there exists a holomorphic immersion $\Phi\in C^0(\bar D^2,\C)$ such that $\Phi|_{S^1}\in W^{2,p}_{\loc}(S^1\setminus\{-i\},\C)$ for $p<\infty$
\begin{equation}\label{cond1bis}
|\Phi'  (z)|=e^{\lambda(z)},~~z\in S^1\,,
\end{equation}
and the curvature of $\Phi|_{S^1\setminus\{-i\}}$ is $\kappa$.
\end{Proposition}

\begin{proof} The function $\lambda_1:=\lambda+\frac{c}{2\pi}\log(1+\sin(\theta))$ satisfies 
$$(-\Delta)^\frac12 \lambda_1=\kappa e^\lambda -1+\frac{c}{2\pi}\quad\text{in }S^1,$$
hence by regularity theory we have $\lambda_1\in L^p(S^1)$, and $\lambda\in L^p(S^1)\cap W^{1,p}_{\loc}(S^1\setminus \{-i\})$ for every $p<\infty$. Then $\rho:=\mathcal{H}(\lambda)\in L^p(S^1)\cap W^{1,p}_{\loc}(S^1\setminus\{-i\})$ and one can define the holomorphic extension $\widetilde{\lambda+i\rho}$ and $\phi:=e^{\widetilde{\lambda+i\rho}}$, and $\Phi$ as before as
$$\Phi(z):=\int_{\Sigma_{0,z}}\phi(w)dw$$
for any path $\Sigma_{0,z}$ in $\bar D^2$ connecting $0$ to $z$.
That $\Phi$ is well defined an continuous also at the point $-i$ depends on the following facts.
We have
$$\lim_{\delta\to 0} \bigg|\int_{-\delta}^\delta \phi (-ie^{it}) dt\bigg| \le \lim_{\delta\to 0} \int_{-\delta}^\delta e^{\lambda(-ie^{it})}  dt =0$$
since $e^\lambda\in L^1(S^1)$.
For almost every $\delta$ we have
$$\lim_{r\to 1^-} \phi(r(-ie^{\pm i\delta}))=\phi(-ie^{\pm i\delta}),$$
see e.g. \cite[Chapter III]{Kat}.
Finally $\phi$ is smooth in $D^2$, so that
$$\lim_{\delta\to 0} \bigg|\int_{-\delta}^\delta \phi ((1-\delta')(-ie^{it})) dt\bigg| =0,\quad \text{for every }\delta'>0.$$
Then one can construct closed oriented paths $\Sigma_{\delta,\delta'}$ by joining the $4$ paths
$$-ie^{it},\, t\in (-\delta,\delta),\quad (1-\delta')(-ie^{it}),\,t\in (-\delta,\delta),\quad (1-r)(-ie^{\pm \delta}),\, r\in (0,\delta')$$
around the singularity $-i$ with
$$\int_{\Sigma_{\delta_k,\delta_k'}}\phi(w)dw\to 0$$ for suitable sequences $\delta_k,\delta_k'\to 0$, as often done in complex function theory.

That $\kappa$ is the curvature of the curve $\Phi|_{S^1\setminus\{-i\}}$ follows as before.
\end{proof}

\noindent\textbf{Proof of Theorem \ref{trm1}} Given $u$ solving \eqref{liou} according to Proposition \ref{proppull2} the function $\lambda$ defined by \eqref{deflambda2} satisfies \eqref{liouvfracn1bis} with
$$c=2\pi-\int_{\R}(-\Delta)^\frac12 u dx\,.$$
By Proposition \ref{propcostrlambdabis} the function $\lambda$ determines a holomorphic immersion $\Phi:\bar D^2\setminus\{-i\}\to \C$ with the property that $\Phi_{S^1\setminus\{-i\}}$ is a curve of curvature $1$ which extends continuously to a closed curve $\Phi|{S^1}$. Then up to translations $\Phi|_{S^1}$ is a parametrization of the unit circle, possibly with degree $n$ different from $1$. But with Lemma \ref{lemmabla}, together with the fact that $\Phi$ is holomorphic and $\Phi'$ never vanishes in $D^2$, we immediately get that $n=1$, $c=0$ and $\Phi$ is a M\"obius diffeomorphism of $\bar D^2$. The explicit form of $\lambda$ and $u$ can be computed as in the previous proof of Theorem \ref{trm1}. \hfill $\square$

\appendix

\section{The fractional Laplacian}\label{sec:lapl}
\subsection{The half-Laplacian on $S^1$ }\label{fraclaps1}

Given $u\in L^1(S^1)$ we define its Fourier coefficients as
$$\hat{u}(n)=\frac{1}{2\pi} \int_{S^1}u(\theta)e^{-in\theta}d\theta,\quad n\in\mathbb{Z}.$$
If $u$ is smooth we can define
\begin{equation}\label{fraclapl5}
\lapfr u(\theta)=\sum_{n\in\mathbb{Z}} |n|\hat{u}(n) e^{in\theta}.
\end{equation}
For $u\in L^1(S^1)$ we can define $\lapfr u\in \mathcal{D}'(S^1)$ as distribution as
\begin{equation}\label{fraclapl6}
\langle \lapfr u,\varphi\rangle := \int_{S^1} u\lapfr\varphi d\theta,\quad \varphi\in C^\infty(S^1).
\end{equation}
Notice that $\varphi \in C^\infty(S^1)$ implies that $\lapfr \varphi \in C^\infty(S^1)$ (here $\lapfr \varphi$ is defined as in \eqref{fraclapl5}). In fact, given $\varphi\in L^1(S^1)$, we have $\varphi\in C^\infty(S^1)$ if and only if $\hat\varphi(n)=o(|n|^{-k})$ for every $k\ge 0$.

We can also give a definition of $\lapfr u$ in terms of harmonic extensions. If $u\in L^1(S^1)$, let $\tilde u(r,\theta)$ be its harmonic extension in $D^2$, explicitly given by the Poisson formula
\begin{equation}\label{Poisson1}
\tilde u(r,\theta)=\frac{1}{2\pi}\int_0^{2\pi}P(r,\theta-t)u(t)dt,\quad P(r,\theta)=\sum_{n\in\mathbb{Z}}r^{|n|}e^{in\theta}=\frac{1-r^2}{1-2r\cos\theta +r^2}
\end{equation}
Then one can define (using polar coordinates)
\begin{equation}\label{fraclapl7}
\lapfr u=\frac{\de \tilde u}{\de r}\bigg|_{r=1}\quad \text{in }\mathcal{D}'(S^1)
\end{equation}
where the distribution $\frac{\de \tilde u }{\de r}\big|_{r=1}$ is defined as
$$
\left\langle \frac{\de \tilde u}{\de r}\bigg|_{r=1},\varphi\right\rangle := \int_{S^1} u\,\frac{\de \tilde \varphi}{\de r}\bigg|_{r=1} d\theta,$$
where $\varphi\in C^{\infty}(S^1)$ and $\tilde\varphi$ is the harmonic extension of $\varphi\,$ in $D^2\,.$\par

Notice that if $u \in C^\infty(S^1)$ the equivalence of \eqref{fraclapl5}, \eqref{fraclapl6} and in fact \eqref{fraclapl7} is elementary, and \eqref{fraclapl7} holds pointwise. For instance the equivalence of \eqref{fraclapl5} and \eqref{fraclapl7} follows at once from
$$\tilde u(r,\theta)=\sum_{n\in\mathbb{Z}}\hat{u}(n) r^{|n|} e^{in\theta}.$$

\begin{Proposition}\label{lapeq2} The definitions \eqref{fraclapl6} and \eqref{fraclapl7} are equivalent.
\end{Proposition} 

\begin{proof}
Since \eqref{fraclapl7} holds pointwise for smooth functions, one has for $u\in L^1(S^1)$ and $\varphi\in C^\infty(S^1)$
$$\langle(-\Delta)^\frac12 u,\varphi\rangle :=\int_{S^1}u(-\Delta)^\frac12 \varphi dx=\int_{S^1} u\,\frac{\de \tilde \varphi}{\de \theta} d\theta=:\left\langle \frac{\de \tilde u}{\de r}\bigg|_{r=1},\varphi\right\rangle.$$
\end{proof}

For $u\in C^{1,\alpha}(S^1)$ there is also the following pointwise definition of $(-\Delta)^\frac12 u$:

\begin{Proposition} If $u\in C^{1,\alpha}(S^1)$ for some $\alpha\in (0,1]$, then $\lapfr u\in C^{0,\alpha}(S^1)$ and
\begin{equation}\label{fraclapl8}
\lapfr u(e^{i\theta})=\frac{1}{\pi}P.V.\int_0^{2\pi} \frac{u(e^{i\theta})-u(e^{it})}{2-2\cos(\theta-t)}dt,
\end{equation}
where the principal value is well-defined because $2-2r\cos(\theta-t)=(\theta-t)^2+O((\theta-t)^4)$ as $t\to\theta$.
\end{Proposition}

\begin{proof}
Considering Proposition \ref{lapeq2} it suffices to show the equivalence of \eqref{fraclapl7} and \eqref{fraclapl8}. Set $\tilde u$ as in \eqref{Poisson1}. Then
\begin{equation*}
\begin{split}
\frac{\de \tilde u(r,\theta)}{\de r}\bigg|_{r=1}&=\lim_{r\uparrow 1}\frac{\tilde u(r,\theta)-u(e^{i\theta})}{r-1}\\
&=\lim_{r\uparrow 1}\frac{1}{2\pi(r-1)}\int_0^{2\pi}\frac{(1-r^2)(u(e^{i\theta})-u(e^{it}))}{1-2r\cos(\theta-t)+r^2}dt\\
&=\lim_{r\uparrow 1}\frac{1}{2\pi}\int_0^{2\pi}\frac{(1+r)(u(e^{i\theta})-u(e^{it}))}{1-2r\cos(\theta-t)+r^2}dt\\
&=\frac{1}{\pi}P.V.\int_0^{2\pi}\frac{u(e^{i\theta})-u(e^{it})}{2-2r\cos(\theta-t)}dt.
\end{split}
\end{equation*}
\end{proof}

\subsection{The half-Laplacian on $\R$}

For $u\in \mathcal{S}$ (the Schwarz space of rapidly decaying functions) we set
\begin{equation}\label{fraclapl0}
\widehat{(-\Delta)^\frac{1}{2}u}(\xi)=|\xi|\hat u(\xi),\quad \hat{f}(\xi):=\int_{\R}f(x)e^{-ix\xi}dx.
\end{equation}
One can prove that it holds (see e.g.)
\begin{equation}\label{fraclapl}
(-\Delta)^\frac{1}{2} u(x)=\frac{1}{\pi} P.V.\int_{\R{}}\frac{u(x)-u(y)}{(x-y)^2}dy:=\frac{1}{\pi} \lim_{\varepsilon\to 0}\int_{\R{}\setminus [-\ve,\ve]}\frac{u(x)-u(y)}{(x-y)^2}dy,
\end{equation}
from which it follows that 
$$\sup_{x\in \R}|(1+x^2)(-\Delta)^{\frac{1}{2}}\varphi(x)|<\infty,\quad\text{for every }\varphi\in \mathcal{S}\,.$$
Then one can set
\begin{equation}\label{L12}
L_\frac{1}{2}(\R):=\left\{u\in L^1_{\loc}(\R):\int_{\R}\frac{|u(x)|}{1+x^2}dx<\infty   \right\},
\end{equation}
and for every $u\in L_{\frac{1}{2}}(\R)$ one defines the tempered distribution $(-\Delta)^\frac{1}{2}u$ as
\begin{equation}\label{fraclapl2}
\langle (-\Delta)^\frac{1}{2}u,\varphi\rangle :=\int_{\R} u(-\Delta)^\frac{1}{2} \varphi dx =\int_{\R}u\,\mathcal{F}^{-1}(|\xi|\hat \varphi(\xi))\,dx,\quad\text{for every }\varphi \in\mathcal{S}.
\end{equation}
An alternative definition of $\lapfr$ can be given via the Poisson integral. For $u\in L_{\frac{1}{2}}(\R)$ define the Poisson integral
\begin{equation}\label{Poisson2}
\tilde u(x,y):=\frac{1}{\pi}\int_{\R}\frac{yu(y)}{(y^2+(x-\xi)^2)}d\xi, \quad y>0,
\end{equation}
which is harmonic in $\R\times(0,\infty)$ and whose trace on $\R\times\{0\}$ is $u$.
Then we have
\begin{equation}\label{fraclapl3}
\lapfr u =- \frac{\de \tilde u}{\partial y}\bigg|_{y=0},
\end{equation}
where the identity is pointwise if $u$ is regular enough (for instance $C^{1,\alpha}_{\loc}(\R)$), and has to be read in the sense of distributions in general, with
\begin{equation}\label{fraclapl3b}
\bigg\langle -\frac{\de \tilde u}{\partial y}\bigg|_{y=0},\varphi\bigg\rangle:=\bigg\langle u, -\frac{\de \tilde \varphi}{\partial y}\bigg|_{y=0} \bigg\rangle,\quad \varphi\in\mathcal{S},\quad\tilde\varphi\text{ as in \eqref{Poisson2}}.
\end{equation}

More precisely:

\begin{Proposition}\label{lapeq} If $u\in L_{\frac{1}{2}}(\R)\cap C^{1,\alpha}_{\loc}((a,b))$ for some interval $(a,b)\subset\R$ and some $\alpha\in (0,1)$, then the tempered distribution $\lapfr u$ defined in \eqref{fraclapl2} coincides on the interval $(a,b)$ with the functions given by \eqref{fraclapl} and \eqref{fraclapl3}. For general $u\in L_\frac12(\R)$ the definitions \eqref{fraclapl2} and \eqref{fraclapl3} are equivalent, where the right-hand side of \eqref{fraclapl3} is defined by \eqref{fraclapl3b}.
\end{Proposition}

\begin{proof} Assume that $u\in L_{\frac{1}{2}}(\R)\cap C^{1,\alpha}_{\loc}((a,b))$. Following \cite{CS3} we have for $x\in (a,b)$
\begin{equation*}
\begin{split}
\frac{\de \tilde u(x,y)}{\partial y}\bigg|_{y=0}&=\lim_{y\to 0}\frac{\tilde u (x,y)-\tilde u(x,0)}{y}\\
&=\lim_{y\to 0}\frac{1}{\pi}\int_{\R}\frac{u(\xi)-u(x)}{y^2+(\xi-x)^2}d\xi\\
&=\frac{1}{\pi}P.V.\int_{\R}\frac{u(\xi)-u(x)}{(\xi-x)^2}d\xi,
\end{split}
\end{equation*}
where the last convergence follows from dominated convergence outside $B_1(x)$ and by a Taylor expansion in a neighborhood of $x$. This proves the equivalence of \eqref{fraclapl} and \eqref{fraclapl3}. The equivalence between \eqref{fraclapl} and \eqref{fraclapl2} amounts to showing that
\begin{equation}\label{fraclapl4}
\int_{\R}u\mathcal{F}^{-1}(|\xi|\hat\varphi(\xi))dx =\frac{1}{\pi}\int_{\R}PV \int_{\R}\frac{u(x)-u(y)}{(x-y)^2}dy \,\varphi(x)dx,
\end{equation}
whenever $\varphi\in \mathcal{S}$ is supported in $(a,b)$.
When $u\in \mathcal{S}$ then the equivalence is shown e.g. in \cite{CS3} (passing through the definition given in \eqref{fraclapl0}). In the general case one approximate $u$ with functions $u_k\in \mathcal{S}$ converging to $u$ uniformly locally in $(a,b)$ and in $L_{\frac{1}{2}}(\R)$, as shown in Proposition 2.1.4 of \cite{sil} (in order to have convergence in \eqref{fraclapl4} as $u_k\to u$, it is convenient to consider $\varphi$ compactly supported first, in case $(a,b)$ is not bounded).

The last statement follows at once by noticing that applying \eqref{fraclapl3} to $\varphi\in \mathcal{S}$, one gets
$$\bigg\langle u,-\frac{\de\tilde\varphi}{\de y}\bigg|_{y=0}\bigg\rangle=\langle u,(-\Delta)^\frac12 \varphi\rangle.$$
\end{proof}

 \section{Useful results from complex analysis}

\begin{Lemma}\label{lemmahconst} Let $h\in C^0(\bar D^2,\mathbb{C})$ be holomorphic in $D^2$ with $h(S^1)\subset S^1$ and $0\not\in h( D^2)$. Then $h$ is constant.
\end{Lemma}

\begin{proof} Since $h$ never vanishes, $\log|h|$ is well defined, harmonic and vanishes on $S^1$, hence everywhere. This implies that $|h|\equiv 1$ and from the conformality of $h$ it follows that $h$ is constant.
\end{proof}

The following is a generalization of Lemma \ref{lemmahconst}.

\begin{Lemma}[Burckel \cite{bur}]\label{lemmabla} Let $h\in C^0(\bar D^2,\mathbb{C})$ be holomorphic in $D^2$ with $h(S^1)\subset S^1$ and $\deg h|_{S^1}=n\ge 0$. Then $h$ is a Blaschke product of degree $n$, i.e.
$$h(z)=e^{i\theta_0}\prod_{k=1}^{n} \frac{z-a_k}{1-\bar a_k z},\quad a_1,\dots, a_n \in D^2,\;\theta_0\in\R.$$
\end{Lemma}



\begin{thebibliography}{2}
\bibitem{AD}  \textsc{D.R. Adams}, \newblock{\em  A note on Riesz potentials}, \newblock{Duke Math. J.  \textbf{42}  (1975), no. 4, 765--778.} 
\bibitem{AIM} \textsc{K. Astala, T. Iwaniec, G. Martin,} \emph{Elliptic Equations and Quasiconformal Mappings in the Plane,} Princeton Mathematical Series, vol. 47, Princeton University Press, 2009.
\bibitem{BGR} \textsc{R. M. Blumenthal, R. K. Getoor, D. B. Ray}, \emph{On the distribution of first hits for the symmetric stable processes}, Trans. Amer. Math. Soc. \textbf{99} (1961), 540-554.
\bibitem{BM} \textsc{H. Brezis, F. Merle}, \emph{Uniform estimates and blow-up behaviour for solutions of $-\Delta u=V(x)e^u$ in two dimensions}, Comm. Partial Differential Equations \textbf{16} (1991), 1223-1253.
\bibitem{Buc} \textsc{C. Bucur}, \emph{Some observations on the Green function for the ball in the fractional Laplace framework}, preprint (2015).
\bibitem{bur} \textsc{R. Burckel}, \emph{An introduction to classical complex analysis, Vol.1}, Pure and applied mathematics, Vol. 82, Academic Press, New York, 1979.
\bibitem{cha} \textsc{S-Y. A. Chang}, 
Non-linear Elliptic Equations in Conformal Geometry, Zurich lecture notes in advanced mathematics, EMS (2004).
\bibitem{ch} \textsc{Chang, S.-Y. A.,}  \emph{Conformal invariants and partial differential equations,} Bull. Amer. Math. Soc. (N.S.) \textbf{42}, 2005, 365-393.
\bibitem{CY2} \textsc{S-Y. A. Chang, P. Yang}, 
\emph{On uniqueness of solutions of $n$-th order differential equations in conformal geometry}, Math. Res. Lett. \textbf{4} (1997), 91-102.
\bibitem{CS3} \textsc{L. Caffarelli, L. Silvestre,} \emph{An extension problem related to the fractional Laplacian}, Comm. Partial Differential Equations \textbf{32} (2007), 1245-1260.
\bibitem{CL} \textsc{W. Chen, C. Li,} \emph{Classification of solutions of some nonlinear elliptic equations}, Duke Math. J. \textbf{63} (3) (1991), 615-622.
 \bibitem{DL} \textsc{F. Da Lio,}    {\it  Compactness and Bubbles Analysis for 1/2-harmonic Maps}, arXiv:1210.2653\,, accepted in Ann Inst. Henri Poincare, Analyse non lineaire\,.\bibitem{DLR1} \textsc{F. Da Lio, T. Rivi\`ere}, \emph{3-Commutators Estimates and the Regularity of  ${1/2}$-Harmonic Maps into  Spheres,} APDE \textbf{4} (2011), 149-190.
\bibitem{DLR2} \textsc{F. Da Lio, T. Rivi\`ere}, \emph{ Sub-criticality of non-local Schr\"odinger systems with antisymmetric potentials and applications to $1/2$-harmonic maps,}    Advances in Mathematics \textbf{227}, (2011), 1300-1348.
\bibitem{DLR}\textsc{F. Da Lio, P. Laurain, T. Rivi\`ere}, \emph{Quantification analysis of linear nonlocal Schr\"odinger-type systems and applications to fractional maps}, in preparation.
\bibitem{DR} \textsc{O. Druet, F. Robert,} \emph{Bubbling phenomena for fourth-order four-dimensional PDEs with exponential growth}, Proc. Amer. Math. Soc \textbf{3} (2006), 897-908.
\bibitem{GH} \textsc{D. Gilbarg, L. H\"ormander}, \emph{Intermediate Schauder estimates}, Archive for Rational Mechanics and Analysis \textbf{74} (1980), 297-318.
\bibitem{GT} \textsc{D. Gilbarg, N. S. Trudinger}, Elliptic Partial Differential Equations of Second Order, Springer, Berlin 2001, reprint of 1998 edition.
\bibitem{Gra}\textsc{L. Grafakos}, \emph{Classical Fourier Analysis}, Springer, New York, 2008.
\bibitem{JMMX}\textsc{T. Jin, A. Maalaoui, L. Martinazzi, J. Xiong}, \emph{Existence and asymptotics for solutions of a non-local Q-curvature equation in dimension three}, Calc. Var. \textbf{52} (2015) 469-488.
\bibitem{Kat} \textsc{Y. Katznelson}, \emph{An introduction to harmonic analysis. Third edition}. Cambridge University Press, Cambridge, 2004. 
\bibitem{CK}  \textsc{C. Kenig}, \emph{Harmonic Analysis Techniques for second order elliptic boundary value problems,}  CBMS, 83, 1994\,.
\bibitem{LS} \textsc{Y. Li, I. Shafrir,} \emph{Blow-up analysis for solutions of $-\Delta u = Ve^u$ in dimension $2$}, Indiana Univ. Math. J. \textbf{43} (1994), 1255-1270.
\bibitem{lin1}\textsc{C-S. Lin,}
\emph{A classification of solutions of a conformally invariant fourth order equation in $\R{n}$},
Comment. Math. Helv. \textbf{73} (1998), no. 2, 206-231. 
\bibitem{mal} \textsc{A. Malchiodi}, \emph{Compactness of solutions to some geometric fourth-order equations}, J. reine angew. Math. \textbf{594} (2006), 137-174.
\bibitem{mar1} \textsc{L. Martinazzi,}  \emph{Classification of solutions to the higher order Liouville's equation on $\mathbb{R}^{2m}$}, Math. Z. \textbf{263} (2009), 307--329.
\bibitem{mar2} \textsc{L. Martinazzi,}  \emph{Conformal metrics on $\mathbb{R}^{2m}$ with constant $Q$-curvature}, Rend. Lincei. Mat. Appl. \textbf{19} (2008), 279-292.
\bibitem{mar3} \textsc{L. Martinazzi,} \emph{Concentration-compactness phenomena in higher order Liouville's equation}, J. Funct. Anal. \textbf{256} (2009), 3743-3771
\bibitem{MR}\textsc{A. Mondino, T. Rivi\`ere}, \emph{Immersed Spheres of Finite Total Curvature into Manifolds}, submitted.
\bibitem{Oneil} \textsc{R. O'Neil,} \emph{Convolution operators and $L(p,q)$ spaces}, Duke Math. J. \textbf{30} (1963), 129-142.
\bibitem{Riv}\textsc{T. Rivi\`ere}, \emph{Conformally Invariant Variational Problems}, arXiv:1206.2116.
\bibitem{sil}\textsc{L. Silvestre}, \emph{Regularity of the obstacle problem for a fractional power of
the Laplace operator}, Comm. Pure Appl. Math. \textbf{60} (2007), 67-122.
\bibitem{St} \textsc{E. M. Stein}, \emph{ Singular integrals and differentiability properties of functions},  Princeton University Press, 1970.
\bibitem{vla}\textsc{V. S. Vladimirov}, \emph{Equations of mathematical physics}, Marcel Dekker, New York (1971).
\bibitem{IT} 
\textsc{Y. Imayoshi, M. Taniguchi,}\newblock{\em  An Introduction to Teichm\"uller Spaces}, Springer 1992.
\bibitem{Xu}\textsc{X. Xu,} \emph{Uniqueness and non-existence theorems for conformally invariant equations}, J. Funct. Anal. \textbf{222} (2005), no. 1, 1--28. 
\end{thebibliography}
     \end{document}